\documentclass[11pt]{article}

\usepackage{amstext,amssymb,amsmath,amsbsy}
\usepackage{hyperref}
\usepackage{amscd}
\usepackage{amsfonts}
\usepackage{indentfirst}
\usepackage{verbatim}
\usepackage{amsmath}
\usepackage{amsthm}
\usepackage{enumerate}
\usepackage{graphicx}
\usepackage{color}
\usepackage[OT1]{fontenc}
\usepackage[latin1]{inputenc}
\usepackage[english]{babel}
\usepackage{amssymb}
\usepackage{subfig}
\usepackage{algorithm}
\usepackage{algpseudocode}
\usepackage{authblk}

\newcommand{\xx}{{\bf x}}

\newtheorem{theorem}{Theorem}[section]

\newtheorem{proposition}{Proposition}[section]
\newtheorem{corollary}{Corollary}[section]
\newtheorem{remark}{Remark}[section]

\newtheorem*{Assumption*}{Assumption}

\newtheorem{problem}{Problem}[section]
\newtheorem*{problem*}{Problem}
\numberwithin{equation}{section}

\begin{document}

\title{A convergent numerical method for a multi-frequency inverse source problem in inhomogenous media}

\author{Loc H. Nguyen}
\author{Qitong Li}
\author{Michael V. Klibanov}

\affil{Department of Mathematics and Statistics, University of North Carolina, Charlotte, Charlotte, NC, 28223, USA, email: loc.nguyen@uncc.edu, qli13@uncc.edu, mklibanv@uncc.edu}
\date{}
\maketitle

\begin{abstract}
	A new numerical method to solve an inverse source problem for the Helmholtz
equation in inhomogenous media is proposed. This method reduces the original
inverse problem to a boundary value problem for a coupled system of elliptic
PDEs, in which the unknown source function is not involved. The Dirichlet
boundary condition is given on the entire boundary of the domain of interest
and the Neumann boundary condition is given on a part of this boundary. To
solve this problem, the quasi-reversibility method is applied. Uniqueness
and existence of the minimizer are proven. A new Carleman estimate is
established. Next, the convergence of those minimizers to the exact solution
is proven using that Carleman estimate. Results of numerical tests are
presented.
\end{abstract}



\section{Introduction and the problem statement}

\label{sec 1}

In this paper, we propose a new numerical method to solve an inverse source
problem for the Helmholtz equation in the multi-frequency regime. This is
the problem of determining the unknown source from external measurement of
the wave field. It is worth mentioning that the inverse source problem has
uncountable real-world applications in electroencephalography, biomedical
imaging, etc., see e.g., \cite{AlbaneseMonk:ip2006,
AmmariBaoFlemming:SIAM2002, BaoLinTriki:jde2010, ChengIsakovLu:jde2016,
DassiosKariotou:jmp2003,
HeRomanov:wm1998, BadiaDuong:ip2000}.

Below $\mathbf{x}=\left( x_{1},...,x_{n-1}, z\right) \in \mathbb{R}^{n}.$
Let $\Omega $ be the cube $(-R,R)^{n}\subset \mathbb{R}^{n}$, $R \geq 1$,
and
\begin{equation}
\Gamma _{+}=\{\mathbf{x}\in \partial \Omega : z=R\}.  \label{1.1}
\end{equation}%
For $i,j=1,...,n$, let functions $a_{ij} \in C^1(\mathbb{R}^n) ,b_{j} \in C(%
\mathbb{R}^n),c \in C(\mathbb{R}^n) $ be such that:

\begin{enumerate}
\item For all $\mathbf{x}\in \mathbb{R}^{n}$
\begin{equation}
a_{ij}(\mathbf{x})=a_{ji}(\mathbf{x})\quad 1\leq i,j\leq n.  \label{1.2}
\end{equation}

\item There exist two constants $\mu _{1}$ and $\mu _{2}$ such that $0<\mu
_{1}\leq \mu _{2}$ and
\begin{equation}
\mu _{1}|\mathbf{\xi }|^{2}\leq \sum_{i,j=1}^{n}a_{ij}\left( \mathbf{x}%
\right) \xi _{i}\xi _{j}\leq \mu _{2}|\mathbf{\xi }|^{2}\quad \mbox{for all }%
\mathbf{x}\in \mathbb{R}^{n},\mathbf{\xi }\in \mathbb{R}^{n}.  \label{1.3}
\end{equation}

\item For all $\mathbf{x}\in \mathbb{R}^{n}\setminus \Omega $
\begin{equation}
a_{ij}(\mathbf{x})=\left\{
\begin{array}{ll}
1 & \mbox{if }i=j, \\
0 & \mbox{if }i\neq j.%
\end{array}%
\right.  \label{1.300}
\end{equation}

\item For all $\mathbf{x}\in \mathbb{R}^{n}\setminus \Omega $,
\begin{equation}
b_{j}(\mathbf{x})=c(\mathbf{x})=0.  \label{1.30}
\end{equation}
\end{enumerate}

We introduce the uniformly elliptic operator $L$ as follows
\begin{equation}
Lu=\sum_{i,j=1}^{n}a_{ij}(\mathbf{x})u_{x_{i}x_{j}}+\sum_{i=1}^{n}b_{i}(%
\mathbf{x})u_{x_{i}}+c(\mathbf{x})u\quad \mbox{for }u\in H^{2}(\mathbb{R}%
^{n}).  \label{100}
\end{equation}%
The principal part of this operator is%
\begin{equation}
L_{0}u=\sum_{i,j=1}^{n}a_{ij}(\mathbf{x})u_{x_{i}x_{j}}.  \label{101}
\end{equation}

Let $k>0$ be the wave number and $u=u(\mathbf{x},k)$ be the complex valued
wave field of wave number $k$, generated by the source function which has
the form of separable variables $g(k)f(\mathbf{x}),$ where functions $g\in
C^{1}[0,\infty )$ and $f\in C^{1}\left( \mathbb{R}^{n}\right) $. The wave
field $u(\mathbf{x},k)\in C^{2}(\mathbb{R}^{n}),k>0,$ satisfies the equation
\begin{equation}
Lu+k^{2}\mathbf{n}^{2}(\mathbf{x})u(\mathbf{x},k)=g(k)f(\mathbf{x}),\quad
\mathbf{x}\in \mathbb{R}^{n}  \label{1.41}
\end{equation}%
and the Sommerfeld radiation condition
\begin{equation}
\partial _{|\mathbf{x}|}u(\mathbf{x},k)-\mathrm{i}ku(\mathbf{x},k)=o(|%
\mathbf{x}|^{(1-n)/2}),\quad |\mathbf{x}|\rightarrow \infty .  \label{1.42}
\end{equation}%
Here, the function $\mathbf{n}\in C^{1}(\mathbb{R}^{n})$ is the refractive
index . We assume that
\begin{equation}
\mathbf{n}\left( \mathbf{x}\right) =1\quad \quad \mbox{for }\mathbf{x}\in
\mathbb{R}^{n}\setminus \Omega .  \label{1.6}
\end{equation}%
See \cite{ColtonKress:2013} for the well-posedness of problem \eqref{1.41}--%
\eqref{1.42} in the case $L=\Delta $. Given numbers $\underline{k}$ and $%
\overline{k}$ such that $0<\underline{k}<\overline{k}<\infty $ and assuming
that the function $g:[\underline{k},\overline{k}]\rightarrow \mathbb{R}$ is
known, we are interested in the following problem.

\begin{problem}[Inverse source problem with Cauchy data]
Assume that conditions (\ref{1.1})-(\ref{1.6}) are in place. Reconstruct the
function $f(\mathbf{x})$ for $\mathbf{x}\in \Omega $, given the functions $F$
and $G$, where
\begin{align}
F(\mathbf{x},k)& =u(\mathbf{x},k),\quad \mathbf{x}\in \partial \Omega ,k\in (%
\underline{k},\overline{k}),  \label{1.7} \\
G(\mathbf{x},k)& =\partial _{z}u(\mathbf{x},k),\quad \mathbf{x}\in \Gamma
_{+},k\in (\underline{k},\overline{k}),\label{1.8}
\end{align}%
where $u(\mathbf{x},k)$ is the solution of problem \eqref{1.41}, \eqref{1.42}%
. \label{ISP}
\end{problem}

Problem \ref{ISP} is somewhat over-determined due to the additional data $G(%
\mathbf{x}, k)$ measured on $\Gamma_+ \times [\underline k, \overline k]$.
We need this data for the convergence theorem. However, we notice in our
numerical experiments that our method works well without that additional
data. More precisely, in addition to Problem \ref{ISP}, we also consider the
following non-overdetermined problem.

\begin{problem}[Inverse source problem with Dirichlet data]
Assume that conditions (\ref{1.1})-(\ref{1.6}) are in place. Reconstruct the
functions $f(\mathbf{x})$, $\mathbf{x}\in \Omega $, given the following data
\begin{equation}
F(\mathbf{x},k) = u(\mathbf{x},k),\quad \mathbf{x}\in \partial \Omega ,k\in (%
\underline{k},\overline{k}),  \label{1.9}
\end{equation}%
where $u(\mathbf{x}, k)$ is the solution of \eqref{1.41}--\eqref{1.42}.\label%
{ISP1}
\end{problem}

\begin{remark}
In fact, the Dirichlet boundary data (\ref{1.9}) implicitly contain the
Neumann boundary data for the function $u$ on the entire boundary $\partial
\Omega .$ Indeed, for each $k\in (\underline{k},\overline{k})$ one can
uniquely solve equation (\ref{1.41}) with the radiation condition (\ref{1.42}%
) and boundary condition (\ref{1.9}) in the unbounded domain $\mathbb{R}%
^{n}\setminus \Omega.$
The resulting solution provides the Neumann boundary
condition $\partial_{\nu }u(\mathbf{x},k)$ for $\mathbf{x}\in \partial
\Omega,$ $k\in (\underline{k},\overline{k}),$ where $\nu $  is the unit
outward normal vector at $\partial \Omega .$
\end{remark}

This and similar inverse source problems for Helmholtz-like PDEs were
studied both analytically and numerically in \cite{BaoLinTriki:jde2010,
IsakovLu:SIAM2018}. In particular, in works \cite{EntekhabiIsakov:ip2018,
IsakovLu:SIAM2018} uniqueness and stability results were proven for the case
of constant coefficients in \eqref{1.41} and it was also shown that the
stability estimate improves when the frequency grows. In \cite%
{IsakovLu:ipi2018} uniqueness was proven for non constant coefficients. To
the best of our knowledge, past numerical methods for these problems are
based on various methods of the minimization of mismatched least squares
functionals. Good quality numerical solutions are obtained in \cite%
{BaoLinTriki:CRM2011, BaoLinTriki:cm2011, IsakovLu:ipi2018}. However, those
minimization procedures do not allow to establish convergence rates of
minimizers to the exact solution when the noise in the data tends to zero.
On the other hand, we refer here to the work \cite{WangMaGuoLi:jde2018, WangGuoZhangLiu:ip2017,  WangSongGuoLiLiu:jcam2019, ZhangGuo:ip2009, ZhangGuoLiu:ip2018},
in which a non-iterative method, based on a fresh idea, was proposed to
solve the inverse source problem for a homogenous medium.
This method is called the Fourier method for solving multifrequency inverse source problems.
 Uniqueness and
stability results were proven in \cite{WangGuoZhangLiu:ip2017} and good
quality numerical results were presented.
We would like to refer the reader to \cite{CaoLiu:preprint2018,  LiLiuSun:IPI2018, LiuUhlmann:ip2015, WangGuoLiLiu:ip2017, XiangSun:ip2019} for  some works studying inverse source problems that are related to the inverse problems in this paper.

In this paper, we solve the inverse source problem for inhomogeneous media.
We propose a new numerical method which enables us to establish convergence
rate of minimizers of a certain functional of the Quasi-Reversibility Method
(QRM) to the exact solution, as long as the noise in the data tends to zero.
Our method is based on four ingredients:

\begin{enumerate}
\item Elimination of the unknown source function $f\left( \mathbf{x}\right) $
from the original PDE via the differentiation with respect to $k$ of the
function $u(\mathbf{x},k)/g\left( k\right) .$

\item The use of a newly published \cite{Klibanov:jiip2017} orthonormal
basis in $L^{2}\left( \underline{k},\overline{k}\right) $ to obtain an
overdetermined boundary value problem for a system of coupled elliptic PDEs
of the second order.

\item The use of the QRM to find an approximate solution of that boundary
value problem.

\item The formulation and the proof of a new Carleman estimate for the
operator $L_{0}$ in (\ref{101}).

\item  In the case of Problem \ref{ISP}, the use of this Carleman estimate
for establishing the convergence rate of the minimizers of the QRM to the
exact solution, as long as the noise in the data tends to zero.
\end{enumerate}

Recently a similar idea was applied to develop a new numerical method for
the X-ray computed tomography with a special case of incomplete data \cite%
{KlibanovNguyen:ip2019} as well as to the development of a globally
convergent numerical method for a 1D coefficient inverse problem \cite%
{KlibanovKolesov:ip2018}. The above items 1, 4 and 5 have roots in the
Bukhgeim-Klibanov method, which was originally introduced in \cite%
{BukhgeimKlibanov:smd1981}. Even though there exists now a significant
number of publications on this method, we refer here only to a few of them
\cite{BeilinaKlibanovBook,
BellassouedYamamoto:SpKK2017,KlibanovTimonov:u2004,Klibanov:jiipp2013} since the current paper
is not about that method. The original goal of \cite%
{BukhgeimKlibanov:smd1981} was to prove uniqueness theorems for coefficient
inverse problems. Nowadays, however, ideas of this method are applied for
constructions of numerical methods for coefficient inverse problems and
other ill-posed problems, see, e.g. \cite%
{Klibanov:jiip2017,KlibanovKolesov:ip2018,KlibanovLiZhang:ip2019,KlibanovThanh:sjam2015}.

The quasi-reversibility method was first introduced by Latts and Lions in \cite{LattesLions:e1969} for numerical solutions of ill-posed problems for
partial differential equations. It has been studied intensively since then,
see e.g., \cite{Becacheelal:AIMS2015, Bourgeois:ip2006,
BourgeoisDarde:ip2010, BourgeoisPonomarevDarde:ipi2019, ClasonKlibanov:sjsc2007, Dadre:ipi2016, KaltenbacherRundell:ipi2019,
KlibanovSantosa:SIAMJAM1991, Klibanov:jiipp2013, LocNguyen:ip2019}. A survey
on this method can be found in \cite{Klibanov:anm2015}. The solution of the
system of the above item 2 due to the quasi-reversibility method is called
regularized solution in the theory of ill-posed problems \cite%
{Tihkonov:kapg1995}. Thus, by item 5 a new Carleman estimate allows us to
prove convergence of regularized solutions to the exact one as the noise in
the data tends to zero. We do this only for \ref{ISP}. In contrast we do not
investigate convergence of our method for Problem \ref{ISP1}. This problem
is studied only numerically below.

The paper is organized as follows. In Section \ref{sec 2}, we present the
numerical methods to solve Problems \ref{ISP} and \ref{ISP1}. Next, in
Section \ref{sec 3}, we discuss about the QRM for Problem \ref{ISP}. We
prove a new Carleman estimate in Section \ref{sec 4}. In section \ref{sec 5}%
, we prove the convergence of the regularized solutions to the true one. In
Section \ref{sec 6}, we describe the numerical implementations for both
Problems \ref{ISP} and \ref{ISP1} and present numerical results.

\section{Numerical Methods for Problems \protect\ref{ISP} and \protect\ref%
{ISP1}}

\label{sec 2}

We first recall a special basis of $L^{2}(\underline{k},\overline{k}),$
which was first introduced in \cite{Klibanov:jiip2017}.

\subsection{A special orthonormal basis in $L^{2}(\protect\underline{k},
\overline{k})$}

\label{sec 2.1}

For each $m\geq 1$, define $\phi _{m}(k)=(k-k_{0})^{m-1}\exp (k-k_{0})$
where $k_{0}=(\underline{k}+\overline{k})/2$. The sequence $\{\phi
_{m}\}_{m=1}^{\infty }$ is complete in $L^{2}(\underline{k},\overline{k})$.
Applying the Gram-Schmidt orthonormalization procedure to the sequence $%
\{\phi _{m}\}_{m=1}^{\infty }$, we obtain an orthonormal basis in $L^{2}(%
\underline{k},\overline{k}),$ denoted by $\{\Psi _{m}\}_{m=1}^{\infty }$. It
is not hard to verify that for each $m,$ the function $\Psi _{m}(k)$ has the
form
\begin{equation*}
\Psi _{m}(k)=P_{m-1}(k-k_{0})\exp (k-k_{0}),
\end{equation*}%
where $P_{m-1}$ is a polynomial of the degree $(m-1)$. This leads to the
following result, which plays an important role in our analysis.

\begin{proposition}[see \protect\cite{Klibanov:jiip2017}]
For $m,r\geq 1$, we have
\begin{equation}
d_{mr}=\int_{\underline{k}}^{\overline{k}}\Psi _{m}(k)\Psi _{r}^{\prime
}(k)dk=\left\{
\begin{array}{ll}
1 & \mbox{if }r=m, \\
0 & \mbox{if }r<m.%
\end{array}%
\right.  \label{1}
\end{equation}%
Consequently, let $N>1$ be an integer. Then the $N\times N$ matrix
\begin{equation}
D_{N}=(d_{mr})_{m,r=1}^{N}  \label{matrix D}
\end{equation}
has determinant $1$ and is invertible. \label{prop MK}
\end{proposition}

\begin{remark}
The basis $\{\Psi _{m}\}_{m=1}^{\infty }$ was first introduced in \cite%
{Klibanov:jiip2017}. Then, it was successfully used to numerically solve
nonlinear coefficient inverse problems \cite{KlibanovLiZhang:ip2019,
KlibanovKolesov:ip2018} and the inverse problem of X-ray tomography with
incomplete data \cite{KlibanovNguyen:ip2019}. \label{rem the uses of basis}
\end{remark}

\subsection{Truncated Fourier series}

\label{sec 2.2}

Assume that in (\ref{1.41}) $g(k)\neq 0,\forall k\in \lbrack \underline{k},%
\overline{k}]$. Introduce the function $v(\mathbf{x},k),$
\begin{equation}
v(\mathbf{x},k)=\frac{u(\mathbf{x},k)}{g(k)},\quad \mathbf{x}\in \Omega
,k\in \lbrack \underline{k},\overline{k}].  \label{2.3}
\end{equation}%
Let $L$ be the elliptic operator defined in (\ref{100}). By \eqref{1.41}
\begin{equation}
L\left( v(\mathbf{x},k)\right) +k^{2}\mathbf{n}^{2}\left( \mathbf{x}\right)
v(\mathbf{x},k)=f(\mathbf{x}),\quad \mathbf{x}\in \Omega ,k\in \lbrack
\underline{k},\overline{k}].  \label{2.4}
\end{equation}%
To eliminate the unknown right hand side $f(\mathbf{x})$ from equation (\ref%
{2.4}), we differentiate it with respect to $k$ and obtain
\begin{equation}
L\left( \partial _{k}v(\mathbf{x},k)\right) +k^{2}\mathbf{n}^{2}(\mathbf{x}%
)\partial _{k}v(\mathbf{x},k)+2k\mathbf{n}^{2}(\mathbf{x})v(\mathbf{x}%
,k)=0,\quad \mathbf{x}\in \Omega ,k\in \lbrack \underline{k},\overline{k}].
\label{2.5}
\end{equation}%
It follows from (\ref{1.7}), (\ref{1.8}) and (\ref{2.3}) that in the case of
Problem \ref{ISP} the function $v$ satisfies the following boundary
conditions
\begin{equation}
v(\mathbf{x},k)=\frac{F(\mathbf{x},k)}{g(k)},\text{ \ \ }\mathbf{x}\in
\partial \Omega ,k\in \lbrack \underline{k},\overline{k}],  \label{2.6}
\end{equation}%
\begin{equation}
\partial _{z}v(\mathbf{x},k)=\frac{G(\mathbf{x},k)}{g(k)},\quad \mathbf{x}%
\in \Gamma _{+},k\in \lbrack \underline{k},\overline{k}].  \label{2.7}
\end{equation}%
In Problem \ref{ISP1} only condition (\ref{2.6}) holds.

Fix an integer $N\geq 1$. Recalling the orthonormal basis $\{\Psi
_{r}\}_{r=1}^{\infty }$ of $L^{2}(\underline{k},\overline{k})$ in Section %
\ref{sec 2.1}, we approximate
\begin{align}
v(\mathbf{x},k)&=\sum_{m=1}^{N}v_{m}(\mathbf{x})\Psi _{m}(k) \quad \mathbf{x}%
\in \Omega ,k\in \lbrack \underline{k},\overline{k}],  \label{v appr} \\
\partial _{k}v(\mathbf{x},k) &=\sum_{m=1}^{N}v_{m}(\mathbf{x})\Psi
_{m}^{\prime }(k) \quad \mathbf{x}\in \Omega ,k\in \lbrack \underline{k},%
\overline{k}],  \label{vk appr}
\end{align}
where
\begin{equation}
v_{m}(\mathbf{x})=\int_{\underline{k}}^{\overline{k}}v(\mathbf{x},k)\Psi
_{m}(k)dk \quad \mathbf{x}\in \Omega, m = 1, 2, \dots, N.  \label{2.8}
\end{equation}

\begin{remark}
Similarly with \cite{Klibanov:jiip2017,KlibanovLiZhang:ip2019,
KlibanovKolesov:ip2018,KlibanovNguyen:ip2019}, we assume here that the
truncated Fourier series \eqref{v appr} satisfies equation \eqref{2.4} and
that truncated Fourier series \eqref{v appr} and \eqref{vk appr}, taken
together, satisfy equation \eqref{2.5}. This is our \textbf{approximate
mathematical model}. Since we work with a numerical method, we accept this
approximation scheme. Our main goal below is to find numerically Fourier
coefficients $v_{m}(\mathbf{x}),$ $m=1,2,\dots ,N,$ of $v(\mathbf{x},k)$,
see \eqref{2.8}. If those Fourier coefficients are approximated, the target
unknown function $f(\mathbf{x})$ can be approximated as the right hand side
of \eqref{2.4}.
\end{remark}

\begin{remark}
The number $N$ is chosen numerically. Proving convergence of our method as $%
N\rightarrow \infty $ is very challenging and such proofs are very rare in
the field of ill-posed problems. Indeed, the intrinsic reason of this is the
ill-posedness of those problems. Therefore, we omit the proof of convergence
of our method as $N\rightarrow \infty .$ Nevertheless, a rich numerical
experience of a number of previous publications, see, e.g. \cite%
{KabanikhinSatybaevShishlenin:svp2005, KabanikhinShishlenin:jiip2011,
KabanikhinSabelfeldNovikovShishlenin:jiip2015, KlibanovThanh:sjam2015,
KlibanovLiZhang:ip2019, KlibanovKolesov:ip2018, KlibanovNguyen:ip2019}
indicates that this truncation technique still leads to good numerical
results.
\end{remark}

We now compare numerically the true function $v(\mathbf{x},k)$ with its
approximation \eqref{v appr}. and observe that their difference is small,
see Figure \ref{fig 1} for the illustration.
\begin{figure}[h]
\begin{center}
\subfloat[The real parts of the true and test
functions]{\includegraphics[width = .4\textwidth]{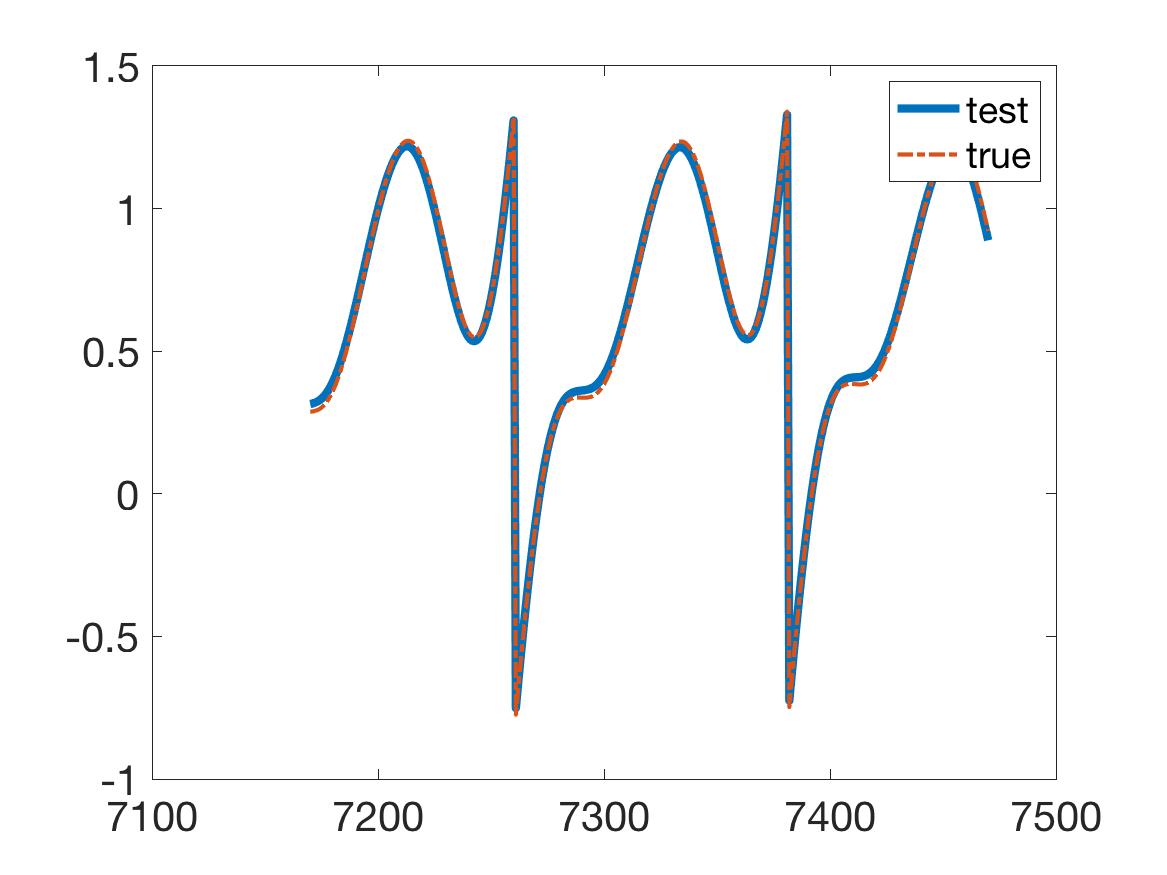}}\quad
\subfloat[The imaginary parts of the true and test
functions]{\includegraphics[width = .4\textwidth]{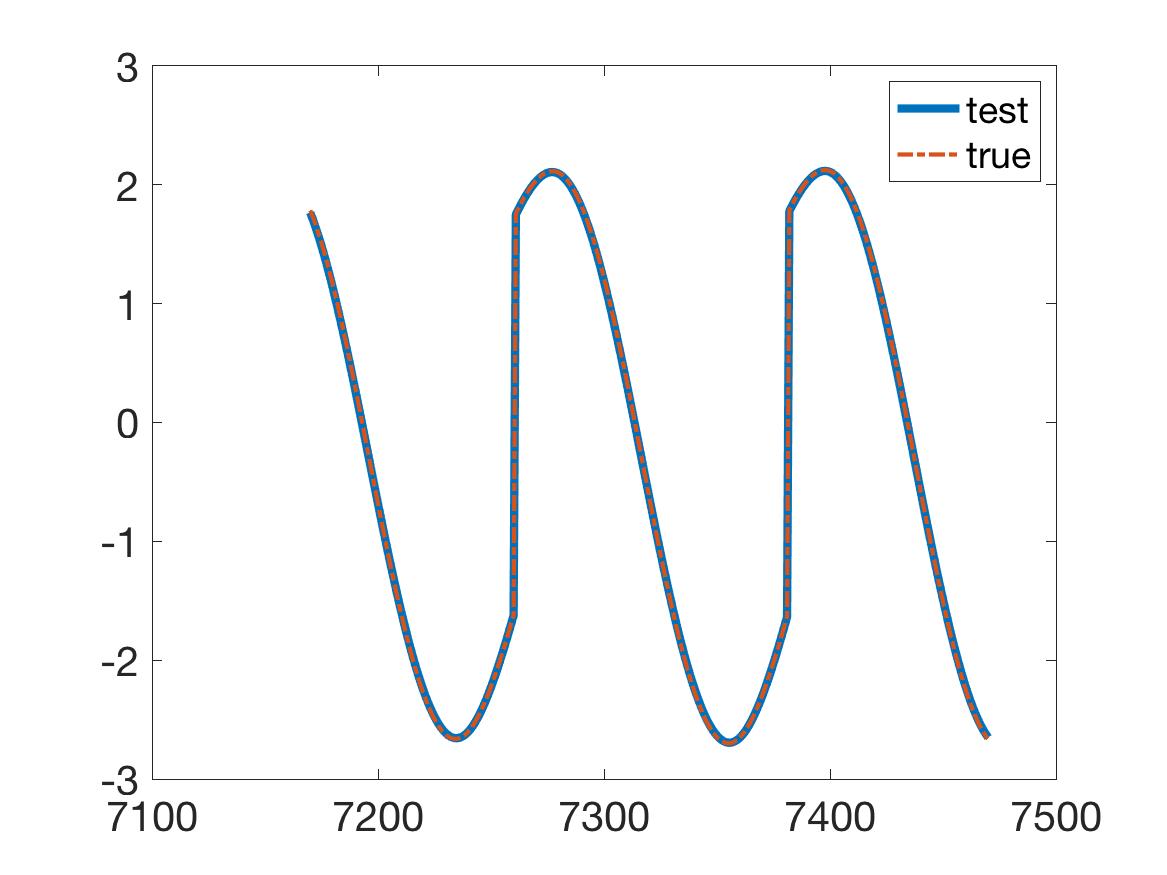}}
\end{center}
\caption{\textit{The comparison of the true function $v(\cdot
,k=1.5)=\sum_{m = 1}^{\infty }v_{m}(\mathbf{x})\Psi _{m}(k)$ and the test
function $\sum_{m=1}^{10}v_{m}(\cdot )\Psi _{m}(k)$ in Test 5, see Section
\protect\ref{sec 4}. In this test, we consider the case $n=2$ and $\Omega
=(-2,2)^{2}$. On $\Omega ,$ we arrange a uniform grid of $121\times 121$
points in $\Omega $. Those points are numbered from $1$ to $121^{2}$. In (a)
and (b), we respectively show the real and imaginary parts of the two
functions at 300 points numbered from 7170 to 7470. It is evident that
reconstructing the first 10 terms of the Fourier coefficients of $v(\mathbf{x%
},k)$ is sufficient to solve our inverse source problems.}}
\label{fig 1}
\end{figure}

Plugging (\ref{v appr}) and (\ref{vk appr}) in equation (\ref{2.5}), we
obtain
\begin{equation}
\sum_{r=1}^{N}\left( Lv_{r}\left( \mathbf{x}\right) \right) \Psi
_{r}^{\prime }(k)+\sum_{r=1}^{N}\left( \mathbf{n}^{2}(\mathbf{x})v_{r}\left(
\mathbf{x}\right) \right) \left( k^{2}\Psi _{r}^{\prime }(k)+2k\Psi
_{r}(k)\right) =0,\text{ }\mathbf{x}\in \Omega .  \label{2.9}
\end{equation}%
For each $m=1,...,N$, we multiply both sides of (\ref{2.9}) by the function $%
\Psi _{m}(k)$ and then integrate the resulting equation with respect to $%
k\in \left( \underline{k},\overline{k}\right) .$ We obtain%
\begin{equation*}
\sum_{r=1}^{N}\left( Lv_{r}\left( \mathbf{x}\right) \right) \int_{\underline{%
k}}^{\overline{k}}\Psi _{r}^{\prime }(k)\Psi _{m}(k)dk
\end{equation*}%
\begin{equation}
+\sum_{r=1}^{N}\left( \mathbf{n}^{2}(\mathbf{x})v_{r}\left( \mathbf{x}%
\right) \right) \int_{\underline{k}}^{\overline{k}}\left( k^{2}\Psi
_{r}^{\prime }(k)+2k\Psi _{r}(k)\right) \Psi _{m}(k)dk=0  \label{2.10}
\end{equation}%
for all $\mathbf{x}\in \Omega $, $m=1,2,\dots ,N.$ Denote%
\begin{equation}
V(\mathbf{x})=(v_{1}(\mathbf{x}),v_{2}(\mathbf{x}),\cdots ,v_{N}(\mathbf{x}%
))^{T}\quad \mathbf{x}\in \Omega ,  \label{2.100}
\end{equation}
\begin{equation}
S_{N}=\left( s_{mr}\right) _{m,r=1}^{N},\quad \mbox{with }s_{mr}=\int_{%
\underline{k}}^{\overline{k}}\left( k^{2}\Psi _{r}^{\prime }(k)+2k\Psi
_{r}(k)\right) \Psi _{m}(k)dk.  \label{2.101}
\end{equation}%
Then, \eqref{matrix D}, (\ref{2.10})-(\ref{2.101}) imply
\begin{equation}
D_{N}L\left( V\left( \mathbf{x}\right) \right) +S_{N}\mathbf{n}^{2}(\mathbf{x%
})V(\mathbf{x})=0,\text{ \ }\mathbf{x}\in \Omega ,  \label{2.11}
\end{equation}%
Denote
\begin{align}
\widetilde{F}(\mathbf{x})& =\left( \int_{\underline{k}}^{\overline{k}}\frac{%
F(\mathbf{x},k)}{g(k)}\Psi _{1}(k)dk,\dots ,\int_{\underline{k}}^{\overline{k%
}}\frac{F(\mathbf{x},k)}{g(k)}\Psi _{N}(k)dk\right) ^{T},\quad \mathbf{x}\in
\partial \Omega ,  \label{2.12} \\
\widetilde{G}(\mathbf{x})& =\left( \int_{\underline{k}}^{\overline{k}}\frac{%
G(\mathbf{x},k)}{g(k)}\Psi _{1}(k)dk,\dots ,\int_{\underline{k}}^{\overline{k%
}}\frac{G(\mathbf{x},k)}{g(k)}\Psi _{N}(k)dk\right) ^{T},\quad \mathbf{x}\in
\Gamma _{+}.  \label{2.13}
\end{align}%
It follows from (\ref{2.6}) and (\ref{2.7}) that in the case of \ref{ISP}
the vector function $V(\mathbf{x})$ satisfies the following two boundary
conditions:%
\begin{equation}
V\left( \mathbf{x}\right) =\widetilde{F}(\mathbf{x}),\text{ }\mathbf{x}\in
\partial \Omega ,  \label{2.14}
\end{equation}%
\begin{equation}
\partial _{\nu }V\left( \mathbf{x}\right) =\widetilde{G}(\mathbf{x}),\text{ }%
\mathbf{x}\in \Gamma _{+}.  \label{2.15}
\end{equation}%
And in the case of \ref{ISP1} only boundary condition (\ref{2.14}) takes
place.

These arguments lead to Algorithms \ref{alg 1} and \ref{alg 2} to solve
Problems \ref{ISP} and \ref{ISP1} respectively.

\begin{algorithm}
\caption{\label{alg 1}The procedure to solve Problem \ref{ISP}}\label{euclid}
\begin{algorithmic}[1]
	\State\, Choose a number $N$. Construct the functions $\Psi_m$, $1 \leq m \leq N,$ in Section \ref{sec 2.1} and compute the matrix $D_N$ as in Proposition \ref{prop MK}.
	\State\,  Calculate the boundary data $\widetilde F$ and $\widetilde G$ for the vector valued function $V$ via \eqref{2.12} and \eqref{2.13} respectively.
	\State\, \label{step quasi} Find an approximate solution of the system \eqref{2.11}, \eqref{2.14} and \eqref{2.15} via the quasi-reversibility method.
	\State\, \label{step vcomp} Having $V = (v_1, v_2, \dots, v_N)^T$ in hand, calculate $v_{\rm comp}(\xx, k)$ via \eqref{2.8}.
	\State\, \label{step fcomp} Compute the reconstructed function $f$ by \eqref{2.4}.
\end{algorithmic}
\end{algorithm}

\begin{algorithm}
\caption{\label{alg 2}The procedure to solve Problem \ref{ISP1}}\label{euclid}
\begin{algorithmic}[1]
	\State\, Choose a number $N$. Construct the functions $\Psi_m$, $1 \leq m \leq N,$ in Section \ref{sec 2.1} and compute the matrix $D_N$ as in Proposition \ref{prop MK}.
	\State\,  Calculate the boundary data $\widetilde F$ for the vector valued function $V$ via \eqref{2.12}.
	\State\, \label{step quasi} Solve the elliptic Dirichlet boundary value problem (\ref{2.11}), (\ref{2.14}).
	\State\, \label{step vcomp} Having $V = (v_1, v_2, \dots, v_N)^T$ in hand, calculate $v_{\rm comp}(\xx, k)$ via \eqref{2.8}.
	\State\, \label{step fcomp} Compute the reconstructed function $f$ by \eqref{2.4}.\end{algorithmic}
\end{algorithm}

In the next section, we briefly discuss the QRM used in Step \ref{step quasi}
of Algorithm \ref{alg 1}. We mention that the QRM is an efficient approach
to solve partial differential equations with over-determined boundary data.

\section{The quasi-reversibility method (QRM)}

\label{sec 3}

In this section, we present the QRM for the numerical solution of Problem %
\ref{ISP}. By saying below that a vector valued function belongs to a
Hilbert space, we mean that each of its components belongs to this space.
The norm of this vector valued function in that Hilbert space is naturally
defined as the square root of the sum of squares of norms of components.
Recall that by Proposition \ref{prop MK} the matrix $D_{N}$ is invertible.
Therefore, by (\ref{2.11}), (\ref{2.14}) and (\ref{2.15}) we need to find an
approximate solution of the following over-determined boundary value problem
with respect to the vector function $V(\mathbf{x})$%
\begin{align}
& L\left( V\left( \mathbf{x}\right) \right) +D_{N}^{-1}S_{N}\mathbf{n}^{2}(%
\mathbf{x})V(\mathbf{x})=0, & & \mathbf{x}\in \Omega ,  \label{2.16} \\
& V\left( \mathbf{x}\right) =\widetilde{F}(\mathbf{x}), & & \mathbf{x}\in
\partial \Omega ,  \label{2.17} \\
& \partial _{\nu }V\left( \mathbf{x}\right) =\widetilde{G}(\mathbf{x}), & &
\mathbf{x}\in \Gamma _{+}.  \label{2.18}
\end{align}%
To do this, we consider the following minimization problem:

\begin{problem}[Minimization Problem]
Let $\epsilon \in \left( 0,1\right) $
be the regularization parameter. Minimize the functional $J_{\epsilon }(V),$%
\begin{equation}
J_{\epsilon }(V)=\int_{\Omega }\left\vert L\left( V\left( \mathbf{x}\right)
\right) +D_{N}^{-1}S_{N}\mathbf{n}^{2}(\mathbf{x})V(\mathbf{x})\right\vert
^{2}d\mathbf{x+}\epsilon \Vert V\Vert _{H^{2}\left( \Omega \right) }^{2},
\label{2.19}
\end{equation}%
on the set of $N-$D vector valued functions $V\in H^{2}\left( \Omega \right)
$ satisfying boundary conditions \eqref{2.17} and \eqref{2.18}.
\end{problem}

We assume that the set of vector functions indicated in the formulation of
this problem is non empty; i.e., we assume that there exists an $N-$D vector
valued function $\Phi $ such that the set
\begin{equation}
\left\{ \Phi \in H^{2}\left( \Omega \right) ,\Phi \mid _{\partial \Omega }=%
\widetilde{F}(\mathbf{x}),\partial _{\nu }\Phi \mid _{\Gamma _{+}}=%
\widetilde{G}(\mathbf{x})\right\}   \label{2.20}
\end{equation}
is nonempty.
\begin{theorem}
Assume that there exists an $N-$D vector valued function $\Phi $ belonging to the set defined in
\eqref{2.20}. Then for each $\epsilon > 0,$ there exists a unique minimizer $%
V_{\min ,\epsilon }\in H^{2}(\Omega) $ of the functional $J_{\epsilon}$ in %
\eqref{2.19} that satisfies boundary conditions \eqref{2.17} and \eqref{2.18}%
. \label{thm min}
\end{theorem}

\begin{proof}
The proof of Theorem \ref{thm min} is based on the
variational principle and Riesz theorem. Let $\left( \cdot ,\cdot \right) $
and $\left[ \cdot ,\cdot \right] $ denote scalar products in Hilbert spaces $%
L^{2}\left( \Omega \right) $ and $H^{2}\left( \Omega \right) $ respectively
of $N-$D vector valued functions. For any vector valued function $V\in
H^{2}\left( \Omega \right) $ satisfying boundary conditions \eqref{2.17} and %
\eqref{2.18}, set
\begin{equation}
W(\mathbf{x})=V(\mathbf{x})-\Phi (\mathbf{x}),\quad \mbox{\xx}\in \Omega .
\label{2.200}
\end{equation}%
By (\ref{2.20}) $W\in H_{0,\#}^{2}\left( \Omega \right) ,$ where
\begin{equation}
H_{0,\#}^{2}\left( \Omega \right) =\left\{ w\in H^{2}\left( \Omega \right)
:w\mid _{\partial \Omega }=0,\partial _{\nu }w\mid _{\Gamma _{+}}=0\right\} .
\label{2.22}
\end{equation}%
Clearly $H_{0,\#}^{2}\left( \Omega \right) $ is a closed subspace of the
space $H^{2}\left( \Omega \right) .$ Let $V_{\min ,\epsilon }$ be any
minimizer of the functional (\ref{2.19}), if it exists. Denote
\begin{equation}
W_{\min ,\epsilon }=V_{\min ,\epsilon }-\Phi .  \label{2.201}
\end{equation}%
By the variational principle the following identity holds
\begin{multline}
\left( L\left( W_{\min ,\epsilon }\left( \mathbf{x}\right) \right)
+D_{N}^{-1}S_{N}\mathbf{n}^{2}(\mathbf{x})W_{\min ,\epsilon }(\mathbf{x}%
),L\left( P\left( \mathbf{x}\right) \right) +D_{N}^{-1}S_{N}\mathbf{n}^{2}(%
\mathbf{x})P(\mathbf{x})\right)
\\
+\epsilon \left[ W_{\min ,\epsilon },P\right]
=\left( L\left( \Phi \left( \mathbf{x}\right) \right) +D_{N}^{-1}S_{N}%
\mathbf{n}^{2}(\mathbf{x})\Phi (\mathbf{x}),L\left( P\left( \mathbf{x}%
\right) \right) +D_{N}^{-1}S_{N}\mathbf{n}^{2}(\mathbf{x})P(\mathbf{x}%
)\right)
\\
+\epsilon \left[ \Phi ,P\right] ,  \label{2.21}
\end{multline}%
for all $P\in H_{0,\#}^{2}(\Omega ).$ The left hand side of the identity %
\eqref{2.21} generates a new scalar product $\left\{ \cdot ,\cdot \right\} $
in the space $H_{0,\#}^{2}\left( \Omega \right) .$ The corresponding norm $%
\left\{ \cdot \right\} $ is equivalent to the standard norm $\left\Vert
\cdot \right\Vert _{H^{2}\left( \Omega \right) }.$ Hence, (\ref{2.21}) is
equivalent with%
\begin{equation*}
\left\{ W_{\min ,\epsilon },P\right\}
\end{equation*}%
\begin{equation}
=\left( L\left( \Phi \left( \mathbf{x}\right) \right) +D_{N}^{-1}S_{N}%
\mathbf{n}^{2}(\mathbf{x})\Phi (\mathbf{x}),L\left( P\left( \mathbf{x}%
\right) \right) +D_{N}^{-1}S_{N}\mathbf{n}^{2}(\mathbf{x})P(\mathbf{x}%
)\right) +\epsilon \left[ \Phi ,P\right]  \label{2.23}
\end{equation}
for all $P\in H_{0,\#}^{2}\left( \Omega \right) .$ On the other hand, the
right hand side of (\ref{2.23}) can be estimated as%
\begin{multline*}
\left\vert \left( L\left( \Phi \left( \mathbf{x}\right) \right)
+D_{N}^{-1}S_{N}\mathbf{n}^{2}(\mathbf{x})\Phi (\mathbf{x}),L\left( P\left(
\mathbf{x}\right) \right) +D_{N}^{-1}S_{N}\mathbf{n}^{2}(\mathbf{x})P(%
\mathbf{x})\right) +\epsilon \left[ \Phi ,P\right] \right\vert
\\
\leq C_{1}\left\{ \Phi \right\} \left\{ P\right\} ,
\end{multline*}%
where the number $C_{1}=C_{1}\left( L,D_{N}^{-1}S_{N},\mathbf{n}%
^{2},\epsilon \right) >0$ depends only on listed parameters. Hence, the
right hand side of (\ref{2.23}) can be considered as a bounded linear
functional $l_{\Phi }\left( P\right) :H_{0}^{2}\left( \Omega \right)
\rightarrow \mathbb{C}.$ By Riesz theorem there exists unique vector
function $Q\in H_{0,\#}^{2}\left( \Omega \right) $ such that
\begin{equation*}
\left\{ W_{\min ,\epsilon },P\right\} =\left\{ Q,P\right\} ,\quad
\mbox{for
all }P\in H_{0,\#}^{2}\left( \Omega \right) ,
\end{equation*}%
directly yielding the identity (\ref{2.23}). As a consequence, $W_{\min
,\epsilon }$ exists and; indeed, $W_{\min ,\epsilon }=Q.$ Finally, by (\ref%
{2.201}) $V_{\min ,\epsilon }=W_{\min ,\epsilon }+\Phi .$
\end{proof}

The minimizer $V_{\min ,\epsilon }$ of $J_{\epsilon },$ subject to the
constraints \eqref{2.17} and \eqref{2.18} is called the \textit{regularized
solution} of the problem \eqref{2.16}, \eqref{2.17} and \eqref{2.18}. In the
theory of Ill-Posed Problems, it is important to prove convergence of
regularized solutions to the true one as the noise in the data tends to zero
\cite{Tihkonov:kapg1995}. In the next section, we establish a Carleman estimate for general
elliptic operators. This estimate is essential for the proof of that
convergence result in our problem, see Section \ref{sec 3}.

\section{A Carleman estimate for general elliptic operators}

\label{sec 4}

For brevity, we assume that the function $u$ in Theorem \ref{Thm Carleman}
is a real valued one. Indeed, this theorem holds true for complex valued
function $u$. This fact follows directly from the theorem itself. Hence, in
this section, we redefine the space $H_{0,\#}^{2}\left( \Omega \right) $ in
(\ref{2.22}) as the set of all real valued functions satisfying the same
constraints.
Recall the operator the uniformly elliptic operator $L_{0}$ in (\ref{101}).

\begin{theorem}[Carleman estimate]
Let the number $b>R$. Let the coefficients $a_{ij}\left( \mathbf{x}\right) $
of the uniformly elliptic operator $L_{0}$ defined in (\ref{101}) satisfy
conditions (\ref{1.2}), (\ref{1.3}) and also $a_{ij}\in C^{1}(\overline{%
\Omega })$. Suppose that
\begin{equation}
a_{in}\left( \mathbf{x}\right) =0,\quad \mbox{ for  }\mathbf{x}\in \partial
\Omega \setminus \left\{ z=\pm R\right\} ,i\neq n.  \label{3.1}
\end{equation}%
Then there exist numbers
\begin{equation*}
p_{0}=p_{0}\left( \mu _{1},\mu _{2},b,n,R,\max_{ij}\left\Vert
a_{ij}\right\Vert _{C^{1}\left( \overline{\Omega }\right) }\right) >1
\end{equation*}%
and
\begin{equation*}
\lambda _{0}=\lambda _{0}\left( \mu _{1},\mu _{2},b,n,R,\max_{ij}\left\Vert
a_{ij}\right\Vert _{C^{1}\left( \overline{\Omega }\right) }\right) \geq 1,
\end{equation*}%
both of which depend only on listed parameters, such that the following
Carleman estimate holds:%
\begin{equation}
\int_{\Omega }\left( L_{0}u\right) ^{2}\exp \left[ 2\lambda \left(
z+b\right) ^{p}\right] d\mathbf{x}\mathbf{\geq }C_{2}\lambda \int_{\Omega }%
\left[ \left( \nabla u\right) ^{2}+\lambda ^{2}u^{2}\right] \exp \left[
2\lambda \left( z+b\right) ^{p}\right] d\mathbf{x},  \label{3.2}
\end{equation}%
for all $\lambda \geq \lambda _{0},$ $p\geq p_{0}$ and $u\in
H_{0,\#}^{2}\left( \Omega \right) $. Here, the constant
\begin{equation*}
C_{2}=C_{2}\left( \mu _{1},\mu _{2},b,p,n,R,\max_{ij}\left\Vert
a_{ij}\right\Vert _{C^{1}\left( \overline{\Omega }\right) }\right) >0
\end{equation*}%
depends only on listed parameters. \label{Thm Carleman}
\end{theorem}

\begin{proof}
Below in this proof $u\in C^{2}\left( \overline{\Omega }%
\right) \cap H_{0,\#}^{2}\left( \Omega \right) .$ The case $u\in
H_{0,\#}^{2}\left( \Omega \right) $ can be obtained via the density
argument. In this proof $C_{2}>0$ denotes different positive numbers
depending only on above listed parameters. On the other hand, everywhere
below $C_{3}=C_{3}\left( \mu _{1},\mu _{2},b,R,\max_{ij}\left\Vert
a_{ij}\right\Vert _{C^{1}\left( \overline{\Omega }\right) }\right) >0$ also
denotes different positive constants depending only on listed parameters but
independent on $p$, unlike $C_{2}.$ Also, in this proof $O\left( 1/\lambda
\right) $ denotes different functions belonging to $C^{1}\left( \overline{%
\Omega }\right) $ and satisfying the estimate
\begin{equation}
\left\Vert O\left( 1/\lambda \right) \right\Vert _{C^{1}\left( \overline{%
\Omega }\right) }\leq \frac{C_{2}}{\lambda }\quad \mbox{for all }\lambda
,p\geq 1.  \label{3.20}
\end{equation}%
Below $n-$D vector functions $U_{k}$ are such that
\begin{equation}
\int_{\partial \Omega }U_{r}\cdot \nu d\sigma \geq 0\quad r\in \left\{
1,...,14\right\} ,  \label{3.11}
\end{equation}%
where $U_{r}\cdot \nu $ means the scalar product of vectors $U_{r}$ and $\nu
$\ in $\mathbb{R}^{n}:$ recall that $\nu $ is the outward looking unit
normal vector on $\partial \Omega .$ In fact it follows from the proof that,
the integrals in (\ref{3.11}) equal zero for $r= 1,2.$ But they are
non-negative starting from $r=3$.

Introduce the new function $v\left( \mathbf{x}\right) =u\left( \mathbf{x}%
\right) \exp \left[ \lambda \left( z+b\right) ^{p}\right] .$ Then
\begin{equation*}
u\left( \mathbf{x}\right) =v\left( \mathbf{x}\right) \exp \left[ -\lambda
\left( z+b\right) ^{p}\right] .
\end{equation*}%
Using straightforward calculations, we obtain
\begin{align*}
u_{x_{i}x_{j}}& =v_{x_{i}x_{j}}\exp \left[ -\lambda \left( z+b\right) ^{p}%
\right] \quad \mbox{for }i,j=1,\dots ,n-1, \\
u_{x_{i}z}& =\left( v_{x_{i}z}-\lambda p\left( z+b\right)
^{p-1}v_{x_{i}}\right) \exp \left[ -\lambda \left( z+b\right) ^{p}\right] ,%
\text{ \ \ }\mbox{for }i,j=1,\dots ,n-1,
\end{align*}%
and
\begin{equation*}
u_{zz}=\left( v_{zz}-2\lambda p\left( z+b\right) ^{p-1}v_{z}+\lambda
^{2}p^{2}\left( z+b\right) ^{2p-2}\left( 1+O\left( 1/\lambda \right) \right)
v\right) \exp \left[ -\lambda \left( z+b\right) ^{p}\right] .
\end{equation*}%
Hence, \eqref{101} implies that
\begin{multline}
\left( L_{0}u\right) \exp \left[ \lambda \left( z+b\right) ^{p}\right]
\\
=\left[ \left(
\sum_{i,j=1}^{n-1}a_{ij}v_{x_{i}x_{j}}+%
\sum_{i=1}^{n-1}a_{in}v_{x_{i}z}+a_{nn}v_{zz}\right) +\left( \lambda
^{2}p^{2}\left( z+b\right) ^{2p-2}a_{nn}v\right) \right]
\\
-2\lambda p\left( z+b\right) ^{p-1}a_{nn}v_{z}-\lambda p\left( z+b\right)
^{p-1}\sum_{i=1}^{n-1}a_{in}v_{x_{i}}.
\label{3.3}
\end{multline}%
Denote terms in the right hand side of (\ref{3.3}) as $%
y_{1},y_{2},y_{3},y_{4}$. More precisely,
\begin{align}
y_{1}&
=\sum_{i,j=1}^{n-1}a_{ij}v_{x_{i}x_{j}}+%
\sum_{i=1}^{n-1}a_{in}v_{x_{i}z}+a_{nn}v_{zz},  \label{3.4} \\
y_{2}& =\lambda ^{2}p^{2}\left( z+b\right) ^{2p-2}a_{nn}v,  \label{3.6} \\
y_{3}& =-2\lambda p\left( z+b\right) ^{p-1}a_{nn}v_{z},  \label{3.7} \\
y_{4}& =-\lambda p\left( z+b\right) ^{p-1}\sum_{i=1}^{n-1}a_{in}v_{x_{i}}.
\label{3.70}
\end{align}%
It follows from (\ref{3.3}) that
\begin{align*}
\left( L_{0}u\right) ^{2}\exp \left[ 2\lambda \left( z+b\right) ^{p}\right]
\left( z+b\right) ^{2-p}
&=\left( y_{1}+y_{2}+y_{3}+y_{4}\right) ^{2}\left(
z+b\right) ^{2-p}
\\
&=\left( \left( y_{1}+y_{2}\right) +\left( y_{3}+y_{4}\right) \right)
^{2}\left( z+b\right) ^{2-p}.
\end{align*}%
Thus,%
\begin{equation}
\left( L_{0}u\right) ^{2}\exp \left[ 2\lambda \left( z+b\right) ^{p}\right]
\left( z+b\right) ^{2-p} 
\geq 2y_{3}\left( y_{1}+y_{2}\right) \left( z+b\right) ^{2-p}+2y_{4}\left(
y_{1}+y_{2}\right) \left( z+b\right) ^{2-p}.  \label{3.8}
\end{equation}%
We now estimate from the below each term in the right hand side of
inequality (\ref{3.8}) separately. We do this in several steps.

\vspace{4pt}\noindent
\textbf{Step 1}. Estimate from the below of the quantity $2y_{1}y_{3}\left(
z+b\right) ^{2-p}.$ By (\ref{3.4}) and (\ref{3.6}), we have
\begin{multline}
2y_{1}y_{3}\left( z+b\right) ^{2-p} \\
=-4\lambda p\left( z+b\right) a_{nn}v_{z}\left( \frac{1}{2}%
\sum_{i,j=1}^{n-1}\left( a_{ij}v_{x_{i}x_{j}}+a_{ij}v_{x_{j}x_{i}}\right)
+\sum_{i=1}^{n-1}a_{in}v_{x_{i}z}+a_{nn}v_{zz}\right) .  \label{3.9}
\end{multline}%
By the standard rules of the differentiation,
\begin{align*}
-2& \lambda p(z+b)a_{nn}v_{z}\left(
a_{ij}v_{x_{i}x_{j}}+a_{ij}v_{x_{j}x_{i}}\right) \\
& =-2\lambda p\Big[ \left( z+b\right) a_{nn}a_{ij}\left(
v_{z}v_{x_{i}}\right) _{x_{j}}-\left( z+b\right)
a_{nn}a_{ij}v_{zx_{j}}v_{x_{i}}
\\
&\hspace{7.5cm}-\left( z+b\right) \left( a_{nn}a_{ij}\right)
_{x_{j}}v_{z}v_{x_{i}}\Big] \\
& \quad -2\lambda p\Big[ \left( z+b\right) a_{nn}a_{ij}\left(
v_{z}v_{x_{j}}\right) _{x_{i}}-\left( z+b\right)
a_{nn}a_{ij}v_{zx_{i}}v_{x_{j}}
\\
&\hspace{7.5cm}
-\left( z+b\right) \left( a_{nn}a_{ij}\right)
_{x_{i}}v_{z}v_{x_{j}}\Big] \\
& =\left[ 2\lambda p\left( z+b\right) a_{nn}a_{ij}v_{x_{i}}v_{x_{j}}\right]
_{z}-2\lambda p\left( \left( z+b\right) a_{nn}a_{ij}\right)
_{z}v_{x_{i}}v_{x_{j}} \\
& \quad +\left( -2\lambda p\left( z+b\right)
a_{nn}a_{ij}v_{z}v_{x_{i}}\right) _{x_{j}}+4\lambda p\left( z+b\right)
\left( a_{nn}a_{ij}\right) _{x_{j}}v_{z}v_{x_{i}} \\
& \quad +\left( -2\lambda p\left( z+b\right)
a_{nn}a_{ij}v_{z}v_{x_{j}}\right) _{x_{i}}+4\lambda p\left( z+b\right)
\left( a_{nn}a_{ij}\right) _{x_{i}}v_{z}v_{x_{j}}.
\end{align*}%
Hence,
\begin{equation}
-2\lambda p\left( z+b\right) a_{nn}v_{z}\left(
a_{ij}v_{x_{i}x_{j}}+a_{ij}v_{x_{j}x_{i}}\right) \geq -C_{3}\lambda p\left(
\nabla v\right) ^{2}+\mbox{div}U_{1},  \label{3.10}
\end{equation}%
see (\ref{3.11}) for $U_{1}.$

Next, we estimate the term%
\begin{equation*}
-\sum_{i=1}^{n-1}4\lambda p\left( z+b\right) a_{nn}a_{in}v_{z}v_{x_{i}z}
\\
=\sum_{i=1}^{n-1}\left( -2\lambda p\left( z+b\right)
a_{nn}a_{in}v_{z}^{2}\right) _{x_{i}}+\sum_{i=1}^{n-1}\lambda p\left(
z+b\right) \left( a_{nn}a_{in}\right) _{x_{i}}v_{z}^{2}.
\end{equation*}%
Hence,
\begin{equation}
-\sum_{i=1}^{n-1}4\lambda p\left( z+b\right) a_{nn}a_{in}v_{z}v_{x_{i}z}\geq
-C_{3}\lambda pv_{z}^{2}+\mbox{div}U_{2}.  \label{3.100}
\end{equation}%
Now, $U_{2}\cdot \nu =0$ for $\mathbf{x}\in \partial \Omega $ for two
reasons: first, this is because $v_{z}\left( \mathbf{x}\right) =0$ for $%
x_{i}=\pm R$ and, second, due to condition (\ref{3.1}). Hence, due to the
first reason, we do not actually use here yet condition (\ref{3.1}).

Next, we estimate the term $-4\lambda p\left( z+b\right)
a_{nn}^{2}v_{z}v_{zz}$ in (\ref{3.9}),
\begin{equation}
-4\lambda p\left( z+b\right) a_{nn}^{2}v_{z}v_{zz}=\left( -2\lambda p\left(
z+b\right) a_{nn}^{2}v_{z}^{2}\right) _{z}+2\lambda p\left( \left(
z+b\right) a_{nn}^{2}\right) _{z}v_{z}^{2}.  \label{3.101}
\end{equation}%
Combining this with (\ref{3.9})-(\ref{3.101}), we conclude that
\begin{equation}
2y_{1}y_{3}\left( z+b\right) ^{2-p}\geq -C_{3}\lambda p\left( \nabla
v\right) ^{2}+\mbox{div}U_{3},  \label{3.12}
\end{equation}%
see (\ref{3.11}) for $U_{3}.$ Next,
\begin{align}
-C_{3}\lambda pv_{z}^{2}& =-C_{3}\lambda p\left( u_{z}^{2}+2\lambda p\left(
z+b\right) ^{p-1}u_{z}u+\lambda ^{2}p^{2}\left( z+b\right)
^{2p-2}u^{2}\right) \exp \left[ 2\lambda \left( z+b\right) ^{p}\right]
\notag \\
& =-C_{3}\lambda pu_{z}^{2}\exp \left[ 2\lambda \left( z+b\right) ^{p}\right]
-C_{3}\lambda ^{3}p^{3}\left( z+b\right) ^{2p-2}u^{2}\exp \left[ 2\lambda
\left( z+b\right) ^{p}\right]  \notag \\
& \hspace{2cm}+\left( -C_{3}\lambda ^{2}p^{2}\left( z+b\right)
^{p-1}u^{2}\exp \left[ 2\lambda \left( z+b\right) ^{p}\right] \right) _{z}
\label{3.14} \\
& \hspace{2cm}+2C_{3}\lambda ^{3}p^{3}\left( z+b\right) ^{2p-2}\left(
1+O\left( 1/\lambda \right) \right) u^{2}\exp \left[ 2\lambda \left(
z+b\right) ^{p}\right]  \notag \\
& \geq -C_{3}\lambda pu_{z}^{2}\exp \left[ 2\lambda \left( z+b\right) ^{p}%
\right] +\mbox{div}U_{4},  \notag
\end{align}%
see (\ref{3.11}) for $U_{4}.$ It follows from (\ref{3.10})-(\ref{3.14}) that%
\begin{equation}
2y_{1}y_{3}\left( z+b\right) ^{2-p}\geq -C_{3}\lambda p\left( \nabla
u\right) ^{2}\exp \left[ 2\lambda \left( z+b\right) ^{p}\right] +\mbox{div}%
U_{5},  \label{3.16}
\end{equation}%
see (\ref{3.11}) for $U_{5}.$

\vspace{4pt}\noindent
\textbf{Step 2}. Estimate from the below the quantity $2y_{2}y_{3}\left(
z+b\right) ^{2-p}.$ By (\ref{3.6}) and (\ref{3.7})
\begin{align}
2y_{2}y_{3}\left( z+b\right) ^{2-p}&=-4\lambda ^{3}p^{3}\left( z+b\right)
^{2p-1}a_{nn}^{2}v_{z}v
\nonumber
\\
&=\left( -2\lambda ^{3}p^{3}\left( z+b\right) ^{2p-1}a_{nn}^{2}v^{2}\right)
_{z}+2\lambda ^{3}p^{3}\left( 2p-1\right) \left( z+b\right)
^{2p-2}a_{nn}^{2}v^{2}
\nonumber
\\
&\hspace{5cm}
+2\lambda ^{3}p^{3}\left( z+b\right) ^{2p-1}\left( a_{nn}^{2}\right)
_{z}v^{2}
\nonumber
\\
&\geq 2\lambda ^{3}p^{3}\left( 2p-1\right) \left( z+b\right) ^{2p-2}\mu
_{1}^{2}\left( 1+\frac{\left( z+b\right) \left( a_{nn}^{2}\right) _{z}}{%
\left( 2p-1\right) \mu _{1}^{2}}\right) v^{2}
\nonumber
\\
&\hspace{5cm}+\left( -2\lambda
^{3}p^{3}\left( z+b\right) ^{2p-1}a_{nn}^{2}v^{2}\right) _{z}
\nonumber
\\
&\geq C_{3}\lambda ^{3}p^{4}\left( z+b\right) ^{2p-2}u^{2}\exp \left[
2\lambda \left( z+b\right) ^{p}\right] +\mbox{div}U_{6},
\label{3.17}
\end{align}%
see (\ref{3.11}) for $U_{6}.$ There exists a sufficiently large number $%
p_{0},$
\begin{equation*}
p_{0}=p_{0}\left( \mu _{1},\mu _{2},b,n,R,\max_{ij}\left\Vert
a_{ij}\right\Vert _{C^{1}\left( \overline{\Omega }\right) }\right) >1
\end{equation*}%
such that
\begin{equation}
1+\frac{\left( z+b\right) \left( a_{nn}^{2}\right) _{z}}{\left( 2p-1\right)
\mu _{1}^{2}}\geq \frac{1}{2},\quad \mbox{for all }p\geq p_{0}.
\label{3.170}
\end{equation}%
Hence, (\ref{3.16})-(\ref{3.170}) imply that for $p\geq p_{0}$
\begin{multline}
2\left( y_{1}+y_{2}\right) y_{3}\left( z+b\right) ^{2-p}\geq -C_{3}\lambda
p\left( \nabla u\right) ^{2}\exp \left[ 2\lambda \left( z+b\right) ^{p}%
\right]
\\
+C_{3}\lambda ^{3}p^{4}\left( z+b\right) ^{2p-2}u^{2}\exp \left[ 2\lambda
\left( z+b\right) ^{p}\right] +\mbox{div}U_{7},  \label{3.21}
\end{multline}%
see (\ref{3.11}) for $U_{7}.$

\vspace{4pt}\noindent
\textbf{Step 3.} Estimate $2y_{1}y_{4}\left( z+b\right) ^{2-p},$ see (\ref%
{3.8}); i.e., estimate
\begin{equation}
\left( -2\lambda p\left( z+b\right) \sum_{k=1}^{n-1}a_{kn}v_{x_{k}}\right)
\left(
\sum_{i,j=1}^{n-1}a_{ij}v_{x_{i}x_{j}}+%
\sum_{i=1}^{n-1}a_{in}v_{x_{i}z}+a_{nn}v_{zz}\right) .  \label{3.18}
\end{equation}%
First,%
\begin{align}
-\lambda& p\left( z+b\right) a_{kn}v_{x_{k}}\left(
a_{ij}v_{x_{i}x_{j}}+a_{ji}v_{x_{j}x_{i}}\right) \nonumber
\\
&=-\lambda p\left( z+b\right) a_{kn}a_{ij}\left(
v_{x_{k}}v_{x_{i}x_{j}}+v_{x_{k}}v_{x_{j}x_{i}}\right)
\nonumber
\\
&=\left( -\lambda p\left( z+b\right) a_{kn}a_{ij}v_{x_{k}}v_{x_{j}}\right)
_{x_{i}}+\lambda p\left( z+b\right) a_{kn}a_{ij}v_{x_{k}x_{i}}v_{x_{j}}
\nonumber
\\
&\quad+\lambda p\left( z+b\right) \left( a_{kn}a_{ij}\right)
_{x_{i}}v_{x_{k}}v_{x_{j}}+\left( -\lambda p\left( z+b\right)
a_{kn}a_{ij}v_{x_{k}}v_{x_{i}}\right) _{x_{j}}
\nonumber
\\
&\hspace{3cm}
+\lambda p\left( z+b\right)
a_{kn}a_{ij}v_{x_{k}x_{j}}v_{x_{i}}
+\lambda p\left( z+b\right) \left( a_{kn}a_{ij}\right)
_{x_{j}}v_{x_{k}}v_{xi}.
\label{3.19}
\end{align}%
Next,
\begin{multline}
\lambda p\left( z+b\right) a_{kn}a_{ij}v_{x_{k}x_{i}}v_{x_{j}}+\lambda
p\left( z+b\right) a_{kn}a_{ij}v_{x_{k}x_{j}}v_{x_{i}}
\\
=\left( \lambda p\left( z+b\right) a_{kn}a_{ij}v_{x_{i}}v_{x_{j}}\right)
_{x_{k}}-\lambda p\left( z+b\right) \left( a_{kn}a_{ij}\right)
_{x_{k}}v_{x_{i}}v_{x_{j}}.  \label{8.1}
\end{multline}%
Hence, it follows from (\ref{3.19}) and (\ref{8.1}) that
\begin{equation}
\left( -2\lambda p\left( z+b\right) \sum_{k=1}^{n-1}a_{kn}v_{x_{k}}\right)
\left( \sum_{i,j=1}^{n-1}a_{ij}v_{x_{i}x_{j}}\right) \geq -C_{3}\lambda
p\left( \nabla v\right) ^{2}+\mbox{div}U_{8}.  \label{8.2}
\end{equation}%
Considering in (\ref{3.19}) and (\ref{8.1}) explicit forms of coordinates of
the vector function $U_{8}$ and using (\ref{3.1}), we conclude that $U_{7}$
satisfies condition (\ref{3.11}).

We now estimate the term
\begin{equation}
\left( -2\lambda p\left( z+b\right) \sum_{k=1}^{n-1}a_{kn}v_{x_{k}}\right)
\left( \sum_{i=1}^{n-1}a_{in}v_{x_{i}z}\right) .  \label{8.3}
\end{equation}%
We have%
\begin{equation*}
\left( -2\lambda p\left( z+b\right) \sum_{k=1}^{n-1}a_{kn}v_{x_{k}}\right)
\left( \sum_{i=1}^{n-1}a_{in}v_{x_{i}z}\right)
\\
=-\lambda p\left( z+b\right) \sum_{i,k=1}^{n-1}a_{kn}a_{in}\left(
v_{x_{k}}v_{x_{i}z}+v_{x_{i}}v_{x_{k}z}\right) .
\end{equation*}%
We have:%
\begin{equation*}
-\lambda p\left( z+b\right) a_{kn}a_{in}\left(
v_{x_{k}}v_{x_{i}z}+v_{x_{i}}v_{x_{k}z}\right)
=\left( -\lambda p\left( z+b\right) a_{kn}a_{in}v_{x_{i}}v_{x_{k}}\right)
_{z}+\lambda p\left( \left( z+b\right) a_{kn}a_{in}\right)
_{z}v_{x_{i}}v_{x_{k}}.
\end{equation*}%
Hence, the term (\ref{8.3}) can be estimated from the below as%
\begin{equation}
\left( -2\lambda p\left( z+b\right) \sum_{k=1}^{n-1}a_{kn}v_{x_{k}}\right)
\left( \sum_{i=1}^{n-1}a_{in}v_{x_{i}z}\right) \geq -C_{3}\lambda p\left(
\nabla v\right) ^{2}+\mbox{div}U_{9},  \label{8.4}
\end{equation}%
where $U_{9}$ satisfies (\ref{3.11}).

We now estimate
\begin{equation}
\left( -2\lambda p\left( z+b\right) \sum_{k=1}^{n-1}a_{kn}v_{x_{k}}\right)
a_{nn}v_{zz}.  \label{8.5}
\end{equation}%
We have%
\begin{align*}
-2\lambda p\left( z+b\right)& a_{kn}a_{nn}v_{x_{k}}v_{zz}
\\
& =\left( -2\lambda
p\left( z+b\right) a_{kn}a_{nn}v_{x_{k}}v_{z}\right) _{z}+2\lambda p\left(
z+b\right) a_{kn}a_{nn}v_{x_{k}z}v_{z} \\
& \hspace{1cm}+2\lambda p\left( \left( z+b\right) a_{kn}a_{nn}\right)
_{z}v_{x_{k}}v_{z} \\
& =\left( \lambda p\left( z+b\right) a_{kn}a_{nn}v_{z}^{2}\right)
_{x_{k}}-\lambda p\left( \left( z+b\right) a_{kn}a_{nn}\right)
_{x_{k}}v_{z}^{2} \\
& \hspace{1cm}+2\lambda p\left( \left( z+b\right) a_{kn}a_{nn}\right)
_{z}v_{x_{k}}v_{z}+\left( -2\lambda p\left( z+b\right)
a_{kn}a_{nn}v_{x_{k}}v_{z}\right) _{z}.
\end{align*}%
Hence, the expression in (\ref{8.5}) can be estimated as%
\begin{equation}
\left( -2\lambda p\left( z+b\right) \sum_{k=1}^{n-1}a_{kn}v_{x_{k}}\right)
a_{nn}v_{zz}\geq -C_{3}\lambda p\left( \nabla v\right) ^{2}+\mbox{div}U_{10},
\label{8.6}
\end{equation}%
where (\ref{3.11}) is valid for $U_{10}.$ Summing up (\ref{8.2}), (\ref{8.4}%
) and (\ref{8.6}), we obtain%
\begin{equation}
2y_{1}y_{4}\left( z+b\right) ^{2-p}\geq -C_{3}\lambda p\left( \nabla
v\right) ^{2}+\mbox{div}U_{11},  \label{8.7}
\end{equation}%
where $U_{11}$ satisfies (\ref{3.11}).

\vspace{4pt}\noindent
\textbf{Step 4}. Estimate $2y_{2}y_{4}\left( z+b\right) ^{2-p},$%
\begin{equation*}
2y_{2}y_{4}\left( z+b\right) ^{2-p}=-2\lambda ^{3}p^{3}\left( z+b\right)
^{2p-1}\sum_{i=1}^{n-1}a_{in}a_{nn}v_{x_{i}}v
\end{equation*}%
\begin{equation*}
=\left( \lambda ^{3}p^{3}\left( z+b\right)
^{2p-1}\sum_{i=1}^{n-1}a_{in}a_{nn}v^{2}\right) _{x_{i}}+\lambda
^{3}p^{3}\left( z+b\right) ^{2p-1}\left( \sum_{i=1}^{n-1}\left(
a_{in}a_{nn}\right) _{x_{i}}\right) v^{2}.
\end{equation*}%
Comparing this with (\ref{3.8}), (\ref{3.14}), (\ref{3.170}), (\ref{3.21})
and (\ref{8.7}), we obtain%
\begin{equation*}
\left( L_{0}u\right) ^{2}\exp \left[ 2\lambda \left( z+b\right) ^{p}\right]
\left( z+b\right) ^{2-p}\geq -C_{3}\lambda p\left( \nabla u\right) ^{2}\exp %
\left[ 2\lambda \left( z+b\right) ^{p}\right]
\end{equation*}%
\begin{equation}
+C_{3}\lambda ^{3}p^{4}\left( z+b\right) ^{2p-2}u^{2}\exp \left[ 2\lambda
\left( z+b\right) ^{p}\right] +\mbox{div}U_{12},\forall p\geq p_{0},
\label{8.8}
\end{equation}%
where $U_{12}$ satisfies (\ref{3.11}).

In addition to the term $\mbox{div}U_{12},$ the right hand side of (\ref{8.8}%
) has one negative and one positive term. But,except of divergence terms ($%
\mbox{div}$), one must have only positive terms in the right hand side of
any Carleman estimate. Therefore, we perform now Step 5.

\vspace{4pt}\noindent
\textbf{Step 5}. Estimate from the below $-\left( L_{0}u\right) u\exp \left[
2\lambda \left( z+b\right) ^{p}\right] .$ We have
\begin{align}
-\left( L_{0}u\right) & u\exp \left[ 2\lambda \left( z+b\right) ^{p}\right]
=-\sum_{i,j=1}^{n-1}a_{ij}u_{x_{i}x_{j}}u\exp \left[ 2\lambda \left(
z+b\right) ^{p}\right]  \notag \\
& -\sum_{i=1}^{n-1}a_{in}u_{x_{i}z}u\exp \left[ 2\lambda \left( z+b\right)
^{p}\right] -a_{nn}u_{zz}u\exp \left[ 2\lambda \left( z+b\right) ^{p}\right]
\notag \\
& =\sum_{i,j=1}^{n-1}\left( -a_{ij}u_{x_{j}}u\exp \left[ 2\lambda \left(
z+b\right) ^{p}\right] \right)
_{x_{i}}+\sum_{i,j=1}^{n-1}a_{ij}u_{x_{i}}u_{x_{j}}\exp \left[ 2\lambda
\left( z+b\right) ^{p}\right]  \notag \\
& \quad +\sum_{i,j=1}^{n-1}\left( a_{ij}\right) _{x_{i}}u_{x_{j}}u\exp \left[
2\lambda \left( z+b\right) ^{p}\right] +\sum_{i=1}^{n-1}\left(
-a_{in}u_{z}u\exp \left[ 2\lambda \left( z+b\right) ^{p}\right] \right)
_{x_{i}}  \notag \\
& \quad +\sum_{i=1}^{n-1}\left( a_{in}\right) _{x_{i}}u_{z}u\exp \left[
2\lambda \left( z+b\right) ^{p}\right] +\sum_{i=1}^{n-1}a_{in}u_{z}u_{x_{i}}%
\exp \left[ 2\lambda \left( z+b\right) ^{p}\right]  \notag \\
& \quad +\left( -a_{nn}u_{z}u\exp \left[ 2\lambda \left( z+b\right) ^{p}%
\right] \right) _{z}+a_{nn}u_{z}^{2}\exp \left[ 2\lambda \left( z+b\right)
^{p}\right]  \notag \\
& \quad +2\lambda p\left( z+b\right) ^{p-1}a_{nn}u_{z}u\exp \left[ 2\lambda
\left( z+b\right) ^{p}\right] +\left( a_{nn}\right) _{z}u_{z}u\exp \left[
2\lambda \left( z+b\right) ^{p}\right] .  \label{3.23}
\end{align}%
Next,%
\begin{multline}
2\lambda p\left( z+b\right) ^{p-1}a_{nn}u_{z}u\exp \left[ 2\lambda \left(
z+b\right) ^{p}\right]
=\left( \lambda p\left( z+b\right)
^{p-1}a_{nn}u^{2}\exp \left[ 2\lambda \left( z+b\right) ^{p}\right] \right)
_{z}
\\
-2\lambda ^{2}p^{2}\left( z+b\right) ^{2p-2}a_{nn}u^{2}\left( 1+O\left(
1/\lambda \right) \right) \exp \left[ 2\lambda \left( z+b\right) ^{p}\right]
.  \label{3.24}
\end{multline}%
Combining (\ref{3.23}) with (\ref{3.24}) and taking into account (\ref{1.3})
as well as inequalities like $u_{x_{i}}u\geq -u_{x_{i}}^{2}/\left( 2\lambda
\right) -\lambda u^{2}/2$, we obtain for $\lambda \geq \lambda _{0}$%
\begin{multline}
-\left( L_{0}u\right) u\exp \left[ 2\lambda \left( z+b\right) ^{p}\right]
\geq \frac{\mu _{1}}{2}\left( \nabla u\right) ^{2}\exp \left[ 2\lambda
\left( z+b\right) ^{p}\right]
\\
-3\lambda ^{2}p^{2}\left( z+b\right) ^{2p-2}a_{nn}u^{2}\exp \left[ 2\lambda
\left( z+b\right) ^{p}\right] +\mbox{div}U_{13},  \label{3.27}
\end{multline}%
see (\ref{3.11}) for $U_{13}.$

\vspace{4pt}\noindent
\textbf{Step 6}. This is the final step. Multiply estimate (\ref{3.27}) by $%
4C_{3}\lambda p/\mu _{1}$ and sum up with (\ref{8.8}). We obtain%
\begin{equation*}
-4C_{3}\lambda p\mu _{1}^{-1}\left( L_{0}u\right) u\exp \left[ 2\lambda
\left( z+b\right) ^{p}\right] +\left( L_{0}u\right) ^{2}\exp \left[ 2\lambda
\left( z+b\right) ^{p}\right] \left( z+b\right) ^{2-p}
\end{equation*}%
\begin{equation}
\geq C_{3}\lambda p\left( \nabla u\right) ^{2}\exp \left[ 2\lambda \left(
z+b\right) ^{p}\right]  \label{3.29}
\end{equation}%
\begin{equation*}
+C_{3}\lambda ^{3}p^{4}\left( z+b\right) ^{2p-2}\left( 1-\frac{12a_{nn}}{%
p\mu _{1}}\right) u^{2}\exp \left[ 2\lambda \left( z+b\right) ^{p}\right] +%
\mbox{div}U_{14},
\end{equation*}%
see (\ref{3.11}) for $U_{14}.$ We can choose $p_{0}$ so large that, in
addition to (\ref{3.170}),
\begin{equation}
1-\frac{12a_{nn}\left( \mathbf{x}\right) }{p\mu _{1}}\geq \frac{1}{2}%
,\forall p\geq p_{0}.  \label{3.31}
\end{equation}%
We estimate the left hand side of (\ref{3.29}) from the above as%
\begin{multline*}
-4C_{3}\lambda p\mu _{1}^{-1}\left( L_{0}u\right) u\exp \left[ 2\lambda
\left( z+b\right) ^{p}\right] +\left( L_{0}u\right) ^{2}\exp \left[ 2\lambda
\left( z+b\right) ^{p}\right] \left( z+b\right) ^{2-p} \\
\leq C_{2}\left( L_{0}u\right) ^{2}\exp \left[ 2\lambda \left( z+b\right)
^{p}\right] +C_{2}\lambda ^{2}u^{2}\exp \left[ 2\lambda \left( z+b\right)
^{p}\right] .
\end{multline*}%
Combining this with the right hand side of (\ref{3.29}), integrating the
obtained pointwise inequality over the domain $\Omega $ and taking into
account (\ref{3.11}), (\ref{3.31}) and Gauss' formula, we obtain the target
estimate (\ref{3.2}).
\end{proof}


\begin{corollary}
Assume that conditions of Theorem \ref{Thm Carleman} are satisfied. Since we
should have in Theorem \ref{Thm Carleman} $b>R$, we choose in (\ref{3.2}) $%
b=3R$. Let $p_{0}>1$ and $\lambda _{0}>1$ be the numbers of Theorem \ref{Thm
Carleman}. Consider the $N-$D complex valued vector functions $W\left(
\mathbf{x}\right) \in H_{0,\#}^{2}\left( \Omega \right) .$ Then there exists
a sufficiently large number $\lambda _{1},$
\begin{equation}
\lambda _{1}=\lambda _{1}(\mu _{1},\mu _{2},n,R,\max_{ij}\left\Vert
a_{ij}\right\Vert _{C^{1}\left( \overline{\Omega }\right)
},\max_{j}\left\Vert b_{j}\right\Vert _{C\left( \overline{\Omega }\right)
},
\left\Vert c\right\Vert _{C\left( \overline{\Omega }\right) }\left\Vert
\mathbf{n}\right\Vert _{C\left( \overline{\Omega }\right) },\underline{k},%
\overline{k},N)
\geq \lambda _{0}  \label{300}
\end{equation}%
depending only on listed parameters such that the following Carleman
estimate holds
\begin{multline*}
\int_{\Omega }\left\vert L\left( W\left( \mathbf{x}\right) \right)
+D_{N}^{-1}S_{N}\mathbf{n}^{2}(\mathbf{x})W(\mathbf{x})\right\vert ^{2}\exp %
\left[ 2\lambda \left( z+3R\right) ^{p_{0}}\right] d\mathbf{x}
\\
\mathbf{\geq }C_{3}\lambda \int_{\Omega }\left( \left\vert \nabla
W\right\vert ^{2}+\lambda ^{2}\left\vert W\right\vert ^{2}\right) \exp \left[
2\lambda \left( z+3R\right) ^{p_{0}}\right] d\mathbf{x},
\end{multline*}%
for all $\lambda \geq \lambda_1,$ $W\in H_{0,\#}^{2}\left( \Omega \right).$
\label{cor 4.1}
\end{corollary}

This Corollary follows immediately from Theorem \ref{Thm Carleman} as well
as from the well known fact (see, e.g. lemma 2.1 in \cite{Klibanov:jiipp2013}%
) that the Carleman estimate depends only on the principal part of a PDE
operator while the lower order terms of this operator can be absorbed in
this estimate.

\section{Convergence analysis}

\label{sec 5}

While Theorem \ref{thm min} ensures the existence and uniqueness of the
solution of the Minimization Problem (Section 3), it does not claim
convergence of minimizers, i.e. regularized solutions, to the exact solution
as noise in the data tends to zero. At the same time such a convergence
result is obviously important. However, this theorem is much harder to prove
than Theorem \ref{thm min}. Indeed, while only the variational principle and
Riesz theorem are used in the proof of Theorem \ref{thm min}, a different
apparatus is required in the convergence analysis. This apparatus is based
on the Carleman estimate of Theorem \ref{Thm Carleman}. In Section \ref{sec
5.1}, we establish the convergence rate of minimizers.

\subsection{Convergence rate}

\label{sec 5.1}

Following one of the main principles of the regularization theory \cite{Tihkonov:kapg1995},
we assume now that vector functions $\widetilde{F}(\mathbf{x})$ and $%
\widetilde{G}(\mathbf{x})$ in (\ref{2.17}) and (\ref{2.18}) are given with a
noise. More precisely, let $\Phi \left( \mathbf{x}\right) \in H^{2}\left(
\Omega \right) $ be the function defined in (\ref{2.20}). We assume that
this is given with a noise of the level $\delta \in \left( 0,1\right) ,$ i.e.%
\begin{equation}
\left\Vert \Phi ^{\ast }-\Phi \right\Vert _{H^{2}\left( \Omega \right) }\leq
\delta ,  \label{3.32}
\end{equation}%
where the vector function $\Phi ^{\ast }\in H^{2}\left( \Omega \right) $
corresponds to the noiseless data. In the case of noiseless data, we assume
the existence of the solution $V^{\ast }\in H^{2}\left( \Omega \right) $ of
the following analog of the problem (\ref{2.16})-(\ref{2.18}):
\begin{equation}
L\left( V^{\ast }\left( \mathbf{x}\right) \right) +D_{N}^{-1}S_{N}\mathbf{n}%
^{2}(\mathbf{x})V^{\ast }(\mathbf{x})=0,\text{ \ }\mathbf{x}\in \Omega ,
\label{3.33}
\end{equation}%
\begin{equation}
V^{\ast }\left( \mathbf{x}\right) =\widetilde{F}^{\ast }(\mathbf{x}),\text{ }%
\mathbf{x}\in \partial \Omega ,  \label{3.34}
\end{equation}%
\begin{equation}
\partial _{\nu }V^{\ast }\left( \mathbf{x}\right) =\widetilde{G}^{\ast }(%
\mathbf{x}),\text{ }\mathbf{x}\in \Gamma _{+}.  \label{3.35}
\end{equation}%
Similarly to (\ref{2.20}), we assume the existence of the vector valued
function function $\Phi ^{\ast }$ such that
\begin{equation}
\Phi ^{\ast }\in H^{2}\left( \Omega \right) ,\Phi ^{\ast }\left( \mathbf{x}%
\right) \mid _{\partial \Omega }=\widetilde{F}^{\ast }(\mathbf{x}),\text{ }%
\partial _{\nu }\Phi ^{\ast }\left( \mathbf{x}\right) \mid _{\Gamma _{+}}=%
\widetilde{G}^{\ast }(\mathbf{x}).  \label{3.36}
\end{equation}%
Similarly to (\ref{2.200}), let
\begin{equation}
W^{\ast }=V^{\ast }-\Phi ^{\ast }.  \label{3.40}
\end{equation}%
Then (\ref{2.22}), (\ref{3.36}) and (\ref{3.40}) imply that $W^{\ast }\in
H_{0}^{2}\left( \Omega \right) .$ Also, using (\ref{3.33})-(\ref{3.36}), we
obtain%
\begin{equation}
L\left( W^{\ast }\left( \mathbf{x}\right) \right) +D_{N}^{-1}S_{N}\mathbf{n}%
^{2}(\mathbf{x})\left( W^{\ast }\left( \mathbf{x}\right) \right) =-L\left(
\Phi ^{\ast }\left( \mathbf{x}\right) \right) -D_{N}^{-1}S_{N}\mathbf{n}^{2}(%
\mathbf{x})\left( \Phi ^{\ast }\left( \mathbf{x}\right) \right) ,
\label{3.41}
\end{equation}
for all $\mathbf{x}\in \Omega.
$
\begin{theorem}[The convergence rate]
Assume that conditions of Theorem \ref{thm min} as well as conditions (\ref%
{3.32})-(\ref{3.40}) hold. Let $\lambda _{1}$ be the number of Corollary \ref%
{cor 4.1}. Define the number $\eta $ as
\begin{equation}
\eta =2\left( 4R\right) ^{p_{0}}.  \label{3.37}
\end{equation}
Assume that the number $\delta _{0}\in \left( 0,1\right) $ is so small that $%
\ln \delta _{0}^{-1/\eta }>\lambda _{1}$. Let $\delta \in \left( 0,\delta
_{0}\right).$ Set $\epsilon =\epsilon \left( \delta \right) =\delta ^{2}.$
Let $V_{\min ,\epsilon \left( \delta \right) }\in H^{2}\left( \Omega \right)
$ be the unique minimizer of the functional (\ref{2.19}) which is found in
Theorem \ref{thm min}. Then the following convergence rate of regularized
solutions holds%
\begin{equation}
\left\Vert V_{\min ,\epsilon \left( \delta \right) }-V^{\ast }\right\Vert
_{H^{1}\left( \Omega \right) }\leq C_{4}\left( 1+\left\Vert W^{\ast
}\right\Vert _{H^{2}\left( \Omega \right) }\right) \sqrt{\delta },
\label{400}
\end{equation}%
where the $C_{4}>0$ depends on the same parameters as those listed in (\ref%
{300}).
\end{theorem}

\noindent\emph{Proof}. We use in this proof the Carleman estimate of Corollary \ref%
{cor 4.1}. Similarly with (\ref{2.201}) let $V_{\min ,\epsilon \left( \delta
\right) }-\Phi =W_{\min ,\epsilon \left( \delta \right) }\in
H_{0,\#}^{2}\left( \Omega \right) $. We now rewrite (\ref{2.21}) as
\begin{multline}
\left( L\left( W_{\min ,\epsilon \left( \delta \right) }\left( \mathbf{x}%
\right) \right) +D_{N}^{-1}S_{N}\mathbf{n}^{2}(\mathbf{x})W_{\min ,\epsilon
\left( \delta \right) }(\mathbf{x}),L\left( P\left( \mathbf{x}\right)
\right) +D_{N}^{-1}S_{N}\mathbf{n}^{2}(\mathbf{x})P(\mathbf{x})\right)
+\epsilon \left( \delta \right) \left[ W_{\min ,\epsilon \left( \delta
\right) },P\right]
\\
=\left( L\left( \Phi \left( \mathbf{x}\right) \right) +D_{N}^{-1}S_{N}%
\mathbf{n}^{2}(\mathbf{x})\Phi (\mathbf{x}),L\left( P\left( \mathbf{x}%
\right) \right) +D_{N}^{-1}S_{N}\mathbf{n}^{2}(\mathbf{x})P(\mathbf{x}%
)\right)
+\epsilon \left( \delta \right) \left[ \Phi ,P\right],
\label{3.42}
\end{multline}%
for all $P\in H_{0,\#}^{2}(\Omega ).$ Also, we rewrite (\ref{3.41}) in an
equivalent form,
\begin{equation}
\left( L\left( W^{\ast }\left( \mathbf{x}\right) \right) +D_{N}^{-1}S_{N}%
\mathbf{n}^{2}(\mathbf{x})W^{\ast }(\mathbf{x}),L\left( P\left( \mathbf{x}%
\right) \right) +D_{N}^{-1}S_{N}\mathbf{n}^{2}(\mathbf{x})P(\mathbf{x}%
)\right) +\epsilon \left( \delta \right) \left[ W^{\ast },P\right]
\label{3.43}
\end{equation}%
\begin{equation*}
=\left( L\left( \Phi ^{\ast }\left( \mathbf{x}\right) \right)
+D_{N}^{-1}S_{N}\mathbf{n}^{2}(\mathbf{x})\Phi ^{\ast }(\mathbf{x}),L\left(
P\left( \mathbf{x}\right) \right) +D_{N}^{-1}S_{N}\mathbf{n}^{2}(\mathbf{x}%
)P(\mathbf{x})\right) +\epsilon \left( \delta \right) \left[ W^{\ast },P%
\right] ,
\end{equation*}%
for all $P\in H_{0,\#}^{2}(\Omega ).$ Denote
\begin{equation*}
\widetilde{W}=W_{\min ,\epsilon \left( \delta \right) }-W^{\ast }\in
H_{0,\#}^{2}\left( \Omega \right) ,\text{ }\widetilde{\Phi }=\Phi -\Phi
^{\ast }.
\end{equation*}%
Subtracting (\ref{3.43}) from (\ref{3.42}), we obtain%
\begin{multline*}
\left( \left( L\left( \widetilde{W}\left( \mathbf{x}\right) \right)
+D_{N}^{-1}S_{N}\mathbf{n}^{2}(\mathbf{x})\widetilde{W}(\mathbf{x}),L\Big(
P\left( \mathbf{x}\right) \right) +D_{N}^{-1}S_{N}\mathbf{n}^{2}(\mathbf{x}%
)P(\mathbf{x})\right)
+\epsilon \left( \delta \right) \left[ \widetilde{W},P%
\right] \Big)\\
=\left( L\left( \widetilde{\Phi }\left( \mathbf{x}\right) \right)
+D_{N}^{-1}S_{N}\mathbf{n}^{2}(\mathbf{x})\widetilde{\Phi }(\mathbf{x}%
),L\left( P\left( \mathbf{x}\right) \right) +D_{N}^{-1}S_{N}\mathbf{n}^{2}(%
\mathbf{x})P(\mathbf{x})\right)
+\epsilon \left( \delta \right) \left[
W^{\ast },P\right] ,
\end{multline*}%
for all $P\in H_{0,\#}^{2}(\Omega ).$ Setting here $P=\widetilde{W}$ and
using Cauchy-Schwarz inequality and (\ref{3.32}), we obtain%
\begin{equation}
\int_{\Omega }\left\vert L\left( \widetilde{W}\left( \mathbf{x}\right)
\right) +D_{N}^{-1}S_{N}\mathbf{n}^{2}(\mathbf{x})\widetilde{W}(\mathbf{x}%
)\right\vert ^{2}d\mathbf{x}\leq C_{4}\delta ^{2}\left( 1+\left\Vert W^{\ast
}\right\Vert _{H^{2}\left( \Omega \right) }^{2}\right) .  \label{3.44}
\end{equation}

We now want to apply Corollary \ref{cor 4.1}. We have
\begin{align*}
\int_{\Omega }& \left\vert L\left( \widetilde{W}\left( \mathbf{x}\right)
\right) +D_{N}^{-1}S_{N}\mathbf{n}^{2}(\mathbf{x})\widetilde{W}(\mathbf{x}%
)\right\vert ^{2}d\mathbf{x} \\
& \mathbf{=}\int_{\Omega }\left\vert L\left( \widetilde{W}\left( \mathbf{x}%
\right) \right) +D_{N}^{-1}S_{N}\mathbf{n}^{2}(\mathbf{x})\widetilde{W}(%
\mathbf{x})\right\vert ^{2}\exp \left( 2\lambda \left( z+3R\right)
^{p_{0}}\right)
\exp \left( -2\lambda \left( z+3R\right) ^{p_{0}}\right) d%
\mathbf{x} \\
& \geq \exp \left( -2\lambda \left( 4R\right) ^{p_{0}}\right) \int_{\Omega
}\left\vert L\left( \widetilde{W}\left( \mathbf{x}\right) \right)
+D_{N}^{-1}S_{N}\mathbf{n}^{2}(\mathbf{x})\widetilde{W}(\mathbf{x}%
)\right\vert ^{2}
\exp \left( 2\lambda \left( z+3R\right) ^{p_{0}}\right) d%
\mathbf{x.}
\end{align*}%
Substituting this into (\ref{3.44}), we obtain
\begin{multline}
\int_{\Omega }\left\vert L\left( \widetilde{W}\left( \mathbf{x}\right)
\right) +D_{N}^{-1}S_{N}\mathbf{n}^{2}(\mathbf{x})\widetilde{W}(\mathbf{x}%
)\right\vert ^{2}\exp \left( 2\lambda \left( z+3R\right) ^{p_{0}}\right) d%
\mathbf{x} \\
\mathbf{\leq }C_{4}\delta ^{2}\left( 1+\left\Vert W^{\ast }\right\Vert
_{H^{2}\left( \Omega \right) }^{2}\right) \exp \left( 2\lambda \left(
4R\right) ^{p_{0}}\right) .  \label{3.45}
\end{multline}%
By Corollary \ref{cor 4.1} the left hand side of inequality (\ref{3.45}) can
be estimated for any $\lambda \geq \lambda _{1}$ as%
\begin{align*}
\int_{\Omega }\big\vert L\left( \widetilde{W}\left( \mathbf{x}\right)
\right) & +D_{N}^{-1}S_{N}\mathbf{n}^{2}(\mathbf{x})\widetilde{W}(\mathbf{x})%
\big\vert^{2}\exp \left( 2\lambda \left( z+3R\right) ^{p_{0}}\right) d%
\mathbf{x} \\
& \mathbf{\geq }C_{3}\lambda \int_{\Omega }\left( \left\vert \nabla
\widetilde{W}\right\vert ^{2}+\lambda ^{2}\left\vert \widetilde{W}%
\right\vert ^{2}\right) \exp \left[ 2\lambda \left( z+3R\right) ^{p_{0}}%
\right] d\mathbf{x} \\
& \geq C_{4}\exp \left[ 2\lambda \left( 2R\right) ^{p_{0}}\right] \left\Vert
W\right\Vert _{H^{1}\left( \Omega \right) }^{2}.
\end{align*}%
Comparing this with (\ref{3.45}), we obtain%
\begin{equation}
\Vert \widetilde{W}\Vert _{H^{1}\left( \Omega \right) }^{2}\leq C_{4}\delta
^{2}\left( 1+\left\Vert W^{\ast }\right\Vert _{H^{2}\left( \Omega \right)
}^{2}\right) \exp \left( 2\lambda \left( 4R\right) ^{p_{0}}\right) .
\label{3.46}
\end{equation}%
Set $\epsilon =\delta ^{2}$. Next, choose $\lambda =\lambda \left( \delta
\right) $ such that $\exp \left( 2\lambda \left( 4R\right) ^{p_{0}}\right)
=1/\delta .$ Hence,
\begin{equation}
\lambda =\lambda \left( \delta \right) =\ln \delta ^{-1/\eta },  \label{3.47}
\end{equation}%
where the number $\eta $ is defined in (\ref{3.37}). This choice is possible
since $\delta \in \left( 0,\delta _{0}\right) $ and $\ln \delta
_{0}^{-1/\eta }>\lambda _{1},$ implying that $\lambda \left( \delta \right)
>\lambda _{1}.$ Thus, (\ref{3.46}) and (\ref{3.47}) imply that%
\begin{equation}
\Vert \widetilde{W}\Vert _{H^{1}\left( \Omega \right) }\leq C_{4}\left(
1+\left\Vert W^{\ast }\right\Vert _{H^{2}\left( \Omega \right) }\right)
\sqrt{\delta }.  \label{3.48}
\end{equation}%
Next, using triangle inequality, (\ref{3.48}) and (\ref{3.32}), we obtain%
\begin{equation*}
C_{4}\left( 1+\Vert W^{\ast }\Vert _{H^{2}\left( \Omega \right) }\right)
\sqrt{\delta }\geq \Vert \widetilde{W}\Vert _{H^{1}\left( \Omega \right)
}=\left\Vert \left( V_{\min ,\epsilon \left( \delta \right) }-V^{\ast
}\right) -\left( \Phi -\Phi ^{\ast }\right) \right\Vert _{H^{1}\left( \Omega
\right) }
\end{equation*}%
\begin{equation*}
\geq \left\Vert V_{\min ,\epsilon \left( \delta \right) }-V^{\ast
}\right\Vert _{H^{1}\left( \Omega \right) }-\left\Vert \Phi -\Phi ^{\ast
}\right\Vert _{H^{1}\left( \Omega \right) }\geq \left\Vert V_{\min ,\epsilon
\left( \delta \right) }-V^{\ast }\right\Vert _{H^{1}\left( \Omega \right)
}-\delta .
\end{equation*}%
Hence,%
\begin{equation}
\left\Vert V_{\min ,\epsilon \left( \delta \right) }-V^{\ast }\right\Vert
_{H^{1}\left( \Omega \right) }\leq \delta +C_{4}\left( 1+\left\Vert W^{\ast
}\right\Vert _{H^{2}\left( \Omega \right) }\right) \sqrt{\delta }\leq
C_{4}\left( 1+\left\Vert W^{\ast }\right\Vert _{H^{2}\left( \Omega \right)
}\right) \sqrt{\delta }.  \label{3.49}
\end{equation}%
Numbers $C_{4}$ in middle and right inequalities (\ref{3.49}) are different
and depend only on parameters listed in (\ref{300}). The target estimate (%
\ref{400}) of this theorem follows from (\ref{3.49}) immediately. \ $\hfill\square
$

\section{Numerical implementation}

\label{sec 6}

In this section, we test our method in the 2-D case. The domain $\Omega $ is
set to be the square
\begin{equation*}
\Omega =(-R,R)^{2}
\end{equation*}%
where $R=2.$ Let $M_{\mathbf{x}}=120$ and $h_{\mathbf{x}}=2R/M_{\mathbf{x}}$%
. We arrange a uniform grid of $(M_{\mathbf{x}}+1)\times (M_{\mathbf{x}}+1)$
points $\{\mathbf{x}_{ij}\}_{i,j=1}^{M_{\mathbf{x}}+1}\subset \overline{%
\Omega }$ where
\begin{equation}
\mathbf{x}_{ij}=(-R+(i-1)h_{\mathbf{x}},-R+(j-1)h_{\mathbf{x}}).
\label{partition x}
\end{equation}%
In this section, we set $\underline{k}=1.5$ and $\overline{k}=4.5.$ The
interval $[\underline{k},\overline{k}]$ is uniformly divided into $M_{k}=150$
sub-intervals whose end points are given by
\begin{equation}
k_{1}=\underline{k}<k_{2}<k_{3}<\dots <k_{M_{k}+1}=\overline{k}
\label{partition k}
\end{equation}%
where $k_{i}=k_{1}+(i-1)h_{k}$ and $h_{k}=(\overline{k}-\underline{k})/M_{k}$%
.

In all numerical tests of this section we computationally simulate the data
for the inverse problem via solving equation (\ref{1.41}) in the square $%
\Omega $ and with the boundary condition at $\partial \Omega $ generated by (%
\ref{1.42}), i.e.%
\begin{equation*}
\partial _{\nu }u\left( \mathbf{x},k\right) -iku\left( \mathbf{x},k\right) =0%
\text{ for }\mathbf{x}\in \partial \Omega .
\end{equation*}%
Hence, we do not specify in this section the operator $L$ and the function $%
\mathbf{n}^{2}(\mathbf{x})$ outside of $\Omega .$ For brevity, we consider
only the isotropic case, i.e. $L=\Delta $ for $\mathbf{x}\in \Omega $.
To show that our method is applicable for the case of non homogeneous media,
we choose the function $\mathbf{n}^{2}(\mathbf{x})$ in all numerical tests
below as:
\begin{equation*}
\mathbf{n}^{2}(\mathbf{x})=1+\frac{0.1\sin (3|\mathbf{x}|^{2})}{3|\mathbf{x}%
|^{2}+1}\quad \mbox{for all }\mathbf{x}\in \Omega .
\end{equation*}

We choose $N=10$ in (\ref{2.8}) by a trial and error procedure. If, for
example $N=5$, then our reconstructed functions $f\left( \mathbf{x}\right) $
are not satisfactory. Choosing $N>10$ does not help us to enhance the
accuracy of computed functions. We also refer here to Figure \ref{fig 1}.

\begin{remark}[The choice for the interval of wave numbers]
The length of each side of the square $\Omega $ is $2R=4$ units. We choose
the longest wavelength $\widetilde{\lambda }_{long}=2\pi /\underline{k}=2\pi
/1.5=4.19$ which is about $4$ units. The upper bound of the wave number $%
\overline{k}=4.5$ is set so that the shortest wavelength $\widetilde{\lambda
}_{short}=1.39$ is in the range that is compatible to the maximal $l_{\max }$
and minimal $l_{\min }$ sizes of the tested inclusions. More precisely, we
choose $\widetilde{\lambda }_{short}\in (0.7l_{\max },1.45l_{\min })$ and $%
\widetilde{\lambda }_{long}/\widetilde{\lambda }_{short}\approx 3.$
\end{remark}

\subsection{The forward problem}

\label{sec 6.1}

To generate the computationally simulated data (\ref{1.7}), (\ref{1.8}), we
need to solve numerically the forward problem (\ref{1.41}), (\ref{1.42}). To
avoid solving this problem in the entire space $\mathbb{R}^{2},$ we solve
the following boundary value problem:
\begin{equation}
\left\{
\begin{array}{rcll}
\Delta u(\mathbf{x},k)+k^{2}\mathbf{n}^{2}(\mathbf{x})u(\mathbf{x},k) & = &
g(k)f(\mathbf{x}) & \mathbf{x}\in \Omega , \\
\partial _{n}u(\mathbf{x},k)-\mathrm{i}ku(\mathbf{x},k) & = & 0 & \mathbf{x}%
\in \partial \Omega ,%
\end{array}%
\right.  \label{4.1}
\end{equation}%
assuming that it has unique solution $u(\mathbf{x},k)\in C^{2}\left(
\overline{\Omega }\right) $ for all $k\in \lbrack \underline{k},\overline{k}%
].$ We solve problem (\ref{4.1}) by the finite difference method. Having
computed the function $u(\mathbf{x},k)$, we extract the noisy data,
\begin{equation}
F(\mathbf{x},k)=u(\mathbf{x},k)(1+\delta (-1+2\mathrm{rand})+\mathrm{i}%
\delta (-1+2\mathrm{rand})),\text{ }\mathbf{x}\in \partial \Omega ,
\label{2}
\end{equation}%
\begin{equation}
G(\mathbf{x},k)=\partial _{z}u(\mathbf{x},k)(1+\delta (-1+2\mathrm{rand})+%
\mathrm{i}\delta (-1+2\mathrm{rand})),\text{ }\mathbf{x}\in \Gamma _{+},
\label{3}
\end{equation}%
see (\ref{1.7}), (\ref{1.8}). Here $\delta \in \left( 0,1\right) $ is the
noise level and \textquotedblleft $\mathrm{rand"}$ is the function built-in in MATLAB, taking
uniformly distributed random numbers in $[0,1].$
The function $-1+2\mathrm{rand}$ of MATLAB represents a uniformly distributed random numbers in $[-1,1].$
In this paper, we test our
method with the noise level $\delta =0.05,$ which means 5\% noise.

\begin{remark}
Recall that while in Problem \ref{ISP} we use both functions $F(\mathbf{x},k)
$ and $G(\mathbf{x},k)$ in (\ref{2}), (\ref{3}), in Problem \ref{ISP1} we
use only the Dirichlet boundary condition $F(\mathbf{x},k),$ see (\ref{1.7}%
)-(\ref{1.9}). However, it follows from boundary condition (\ref{4.1}) that
the Neumann boundary condition is $\partial _{\nu }u(\mathbf{x},k)\mid
_{\partial \Omega }=ikF(\mathbf{x},k)$. This explains why we computationally
observe the uniqueness of our numerical solution of Problem \ref{ISP1}.
\end{remark}

\subsection{The inverse problem}

\label{sec 6.2}

In this section we describe the numerical implementation of the minimization
procedure for the functional $J_{\epsilon }$. We use the following form of
the functionals $J_{\epsilon }$:
\begin{equation}
J_{\epsilon }(V)=\int_{\Omega }|D_{N}\Delta V+S_{N}\mathbf{n}^{2}\left(
\mathbf{x}\right) V|^{2}d\mathbf{x}+\epsilon \Vert V\Vert _{L^{2}(\Omega
)}^{2}.  \label{4.2}
\end{equation}%
This functional $J_{\epsilon }$ in (\ref{4.2}) is slightly different from
the one in (\ref{2.19}). First, we do not use here the matrix $D_{N}^{-1}.$
Indeed, this matrix is convenient to use for the above theoretical results.
However, it is inconvenient to use in computations since it contains large
numbers at $N=10$. Second, we replace the term $\Vert V\Vert _{H^{2}(\Omega
)}^{2}$ in (\ref{2.19}) by the term $\Vert V\Vert _{L^{2}(\Omega )}^{2}.$
This is because the $L^{2}(\Omega )-$norm is easier to work with
computationally than the $H^{2}(\Omega )-$norm. On the other hand, we have
not observed any instabilities probably because the number $121\times 121$
of grid points we use is not too large and all norms in finite dimensional
spaces are equivalent. The regularization parameter $\epsilon $ in our
computations was found by a trial and error procedure, $\epsilon =10^{-5}.$

We write derivatives involved in (\ref{4.2}) via finite differences. Next,
we minimize the resulting functional with respect to values of the vector
valued function
\begin{equation*}
V\left( \mathbf{x}\right) = (v_1(\mathbf{x}), v_2(\mathbf{x}), \dots, v_N(%
\mathbf{x}))^T
\end{equation*}
at grid points. The finite difference approximation of the functional $%
J_{\epsilon }(V)$ is
\begin{equation*}
\begin{array}{rl}
J_{\epsilon }(V) & =\displaystyle h_{\mathbf{x}}^{2}\sum_{i,j=2}^{M_\mathbf{x%
}}\sum_{m=1}^{N}\Big|\sum_{r=1}^{N}\Big\{\frac{d_{mr}}{h_{\mathbf{x}}^{2}}%
\big[v_r(x_{i-1},y_{j})+v_r(x_{i+1},y_{j})
+v_r(x_{i},y_{j-1})+v_r(x_{i},y_{j+1})-4v_r(x_{i},y_{j})\big]
\\
& \hspace{3mm}+\mathbf{n}^{2}(x_{i},y_{j})s_{mr}v_r(x_{i},y_{j})\Big\}\Big|%
^{2}+\displaystyle\epsilon h_{\mathbf{x}}^{2}\sum_{i,j=1}^{M_{\mathbf{x}%
}+1}\sum_{m=1}^{N}|v_m(x_{i},y_{j})|^{2},%
\end{array}%
\end{equation*}%
where $d_{mn}$ and $s_{mn}$ are elements of matrices $D_{N}$ and $S_{N}$ in (%
\ref{1}) and (\ref{2.101}) respectively. Introduce the \textquotedblleft
line up" version of the set $\{v_{n}(x_{i},y_{j}):1\leq i,j\leq M_{\mathbf{x}%
}+1,1\leq n\leq N\}$ as the $(M_{\mathbf{x}}+1)^{2}N$ dimensional vector $%
\mathcal{V}$ with
\begin{equation}
\mathcal{V}_{\mathfrak{m}}=v_{m}(x_{i},y_{j})\quad 1\leq i,j\leq M_{\mathbf{x%
}}+1,1\leq m\leq N,  \label{V and Vlineup}
\end{equation}%
where
\begin{equation}
\mathfrak{m}=(i-1)(M_{\mathbf{x}}+1)N+(j-1)N+m.  \label{line up index}
\end{equation}%
It is not hard to check that the map
\begin{equation*}
\{1,\dots, M_{\mathbf{x}}+1\}\times \{1, \dots, M_{\mathbf{x}}+1\}\times
\{1,\dots ,N\}\rightarrow \{1,\dots ,(M_{\mathbf{x}}+1)^{2}N\}
\end{equation*}%
that sends $(i,j,m)$ to $\mathfrak{m}$ as in \eqref{line up index} is onto
and one-to-one. The functional $J_{\epsilon }(V)$ is rewritten in terms of
the line up vector $\mathcal{V}$ as
\begin{equation*}
\mathcal{J}_{\epsilon }(\mathcal{V})=h_{\mathbf{x}}^{2}|\mathcal{L}\mathcal{V%
}|^{2}+\epsilon h_{\mathbf{x}}^{2}|\mathcal{V}|^{2}
\end{equation*}%
where $\mathcal{L}$ is the $(M_{\mathbf{x}}+1)^{2}N\times (M_{\mathbf{x}%
}+1)^{2}N$ matrix defined as follows. For each $\mathfrak{m}=(i-1)(M_{%
\mathbf{x}}+1)N+(j-1)N+m$, $2\leq i,j\leq M_{\mathbf{x}}$, $1\leq m\leq N$,

\begin{enumerate}
\item set $\mathcal{L}_{\mathfrak{m}\mathfrak{n}}=-\frac{4d_{mn}}{h_{\mathbf{%
x}}^{2}}+\mathbf{n}^{2}(x_{i},y_{j})b_{mn},$ if $\mathfrak{n}=(i-1)(M_{%
\mathbf{x}}+1)N+(j-1)N+n,1\leq n\leq N;$

\item set $\mathcal{L}_{\mathfrak{m}\mathfrak{n}}=\frac{1}{h_{\mathbf{x}}^{2}%
}$ if $\mathfrak{n}=(i\pm 1-1)(M_{\mathbf{x}}+1)N+(j-1)N+n$ or $\mathfrak{n}%
=(i-1)(M_{\mathbf{x}}+1)N+(j\pm 1-1)N+n$, $1\leq n\leq N.$
\end{enumerate}

It is obvious that the minimizer of $\mathcal{J}_{\epsilon }$ satisfies the
equation
\begin{equation}
(\mathcal{L}^*\mathcal{L}+\epsilon \mathrm{Id})\mathcal{V}=\vec{0}.
\label{eqn for min}
\end{equation}%
Here, $\vec{0}$ is the $(M_{\mathbf{x}}+1)^{2}N$ dimensional zero vector.

Next, we consider the \textquotedblleft line up" version of the first
condition in \eqref{2.14}. The following information is available
\begin{equation*}
\mathcal{V}_{\mathfrak{m}}=\tilde{F}_{N}(x_{i},y_{j},m),
\end{equation*}%
where $\mathfrak{m}$ is as in \eqref{line up index}. 
Hence, let $\mathcal{D}$ be the $(M_{\mathbf{x}}+1)^{2}N\times (M_{\mathbf{x}%
}+1)^{2}N$ diagonal matrix with such $\mathfrak{m}^{\mathrm{th}}$ diagonal
entries taking value $1$ while the others are $0$. This Dirichlet boundary
constraint of the vector $\mathcal{V} $ become
\begin{equation}
\mathcal{D}\mathcal{V}=\tilde{\mathcal{F}}.
\label{boundary constraints for V}
\end{equation}%
Here, the vector $\tilde{\mathcal{F}}$ is the \textquotedblleft line up"
vector of the data $F_{N}$ in the same manner when we defined $\mathcal{V}$,
see \eqref{line up index}.

We implement the constraint of $V$ in \eqref{2.15}. This constraint allows
us to collect the following information
\begin{equation}
\frac{\mathcal{V}_{\mathfrak{m}}-\mathcal{V}_{\mathfrak{m^{\prime }}}}{h_{%
\mathbf{x}}}=\tilde{G}_{N}(x_{i},y_{j},m)  \label{5.14}
\end{equation}%
where $\mathfrak{m}$ is as in \eqref{line up index} and
\begin{equation}
\mathfrak{m^{\prime }}=(i-1)(M_{\mathbf{x}}+1)N+(j-2)N+m  \label{mp}
\end{equation}%
for $1\leq i\leq M_{\mathbf{x}}+1$ and $j=M_{\mathbf{x}}+1.$ We rewrite %
\eqref{5.14} as
\begin{equation}
\mathcal{N}\mathcal{V}=\tilde{\mathcal{G}}  \label{Neumann}
\end{equation}%
where $\tilde{\mathcal{G}}$ is the \textquotedblleft line up" version of $%
\tilde{G}_{N}$ and the matrix $\mathcal{N}$ is defined as

\begin{enumerate}
\item $\mathcal{N}_{\mathfrak{m} \mathfrak{m}} = 1/h_\mathbf{x}$ and $%
\mathcal{N}_{\mathfrak{m} \mathfrak{m^{\prime }}} = -1/h_{\mathbf{x}}$ for $%
\mathfrak{m}$ and $\mathfrak{m^{\prime }}$ given by \eqref{line up index}
and \eqref{mp} respectively, $1 \leq i \leq M_{\mathbf{x}} + 1, j = M_{%
\mathbf{x}} + 1.$

\item Other entries of $\mathcal{N}$ are $0$.
\end{enumerate}

In practice, we compute $\mathcal{V}$ by solving
\begin{equation}
\left( \left[
\begin{array}{c}
\mathcal{L} \\
\mathcal{D} \\
\mathcal{N}%
\end{array}%
\right] ^{T}\left[
\begin{array}{c}
\mathcal{L} \\
\mathcal{D} \\
\mathcal{N}%
\end{array}%
\right] +\epsilon \mathrm{Id}\right) \mathcal{V}=\left[
\begin{array}{c}
\mathcal{L} \\
\mathcal{D} \\
\mathcal{N}%
\end{array}%
\right] ^{T}\left[
\begin{array}{c}
\vec{0} \\
\tilde{\mathcal{F}} \\
\tilde{\mathcal{G}}%
\end{array}%
\right]  \label{C}
\end{equation}%
in the case of Problem \ref{ISP} and we solve
\begin{equation}
\left( \left[
\begin{array}{c}
\mathcal{L} \\
\mathcal{D}%
\end{array}%
\right] ^{T}\left[
\begin{array}{c}
\mathcal{L} \\
\mathcal{D}%
\end{array}%
\right] +\epsilon \mathrm{Id}\right) \mathcal{V}=\left[
\begin{array}{c}
\mathcal{L} \\
\mathcal{D}%
\end{array}%
\right] ^{T}\left[
\begin{array}{c}
\vec{0} \\
\tilde{\mathcal{F}}%
\end{array}%
\right]  \label{D}
\end{equation}%
for Problem \ref{ISP1}. Having the vector $\mathcal{V}$, we can compute the
vector $V_{N}$ via \eqref{V and Vlineup}. Then, we follow Steps \ref{step
vcomp} and \ref{step fcomp} of Algorithms \ref{alg 1} and Algorithms \ref%
{alg 2} to compute the functions $v_{\mathrm{comp}}$ via \eqref{v appr} and
then ${f}^{\mathrm{comp}}$ by taking the real part of \eqref{2.4} when $k=1.5
$.

\begin{remark}[Remark on Problem \ref{ISP1}]
	We use \eqref{D} only for the convenience, since we do not want to have a significant extra programming effort, given that we have the computer code for solving \eqref{C}.
\end{remark}

\subsection{Tests}

\label{sec 6.3}

In the cases of Test 1 and Test 2, we apply below our method for Problem \ref%
{ISP}. And in the cases of Tests 3-5 we apply our method for Problem \ref%
{ISP1}. Whenever we say below about the accuracy of values of positive and
negative parts of inclusions, we compare maximal positive values and minimal
negative values of computed ones with true ones. Postprocessing was not
applied in all tests presented below.

\begin{enumerate}
\item \textit{Test 1. Problem \ref{ISP}. Two inclusions with different
shapes.} The function $f_{\mathrm{true}}$ is given by
\begin{equation*}
f_{\mathrm{true}}=\left\{
\begin{array}{rl}
2.5 & \mbox{if }\max \{0.6|x-0.75|,|y|\}<1.1, \\
-2 & \mbox{if }(x+0.75)^{2}+y^{2}<0.55^{2}, \\
0 & \mbox{otherwise,}%
\end{array}%
\right.
\end{equation*}%
and $g_{\mathrm{true}}(k)=\mathrm{i}k$ for $k\in \lbrack \underline{k},%
\overline{k}].$ We test the reconstructions of the locations, shapes and
positive/negative values of the function $f$ for two different inclusions.
One of them is a rectangle and the other one is a disk. In this case, the
function $f_{\mathrm{true}}$ attains both positive and negative values. The
numerical solution for this case is displayed on Figure \ref{fig model 1}.

\begin{figure}[h]
\begin{center}
\subfloat[]{\includegraphics[width =
.3\textwidth]{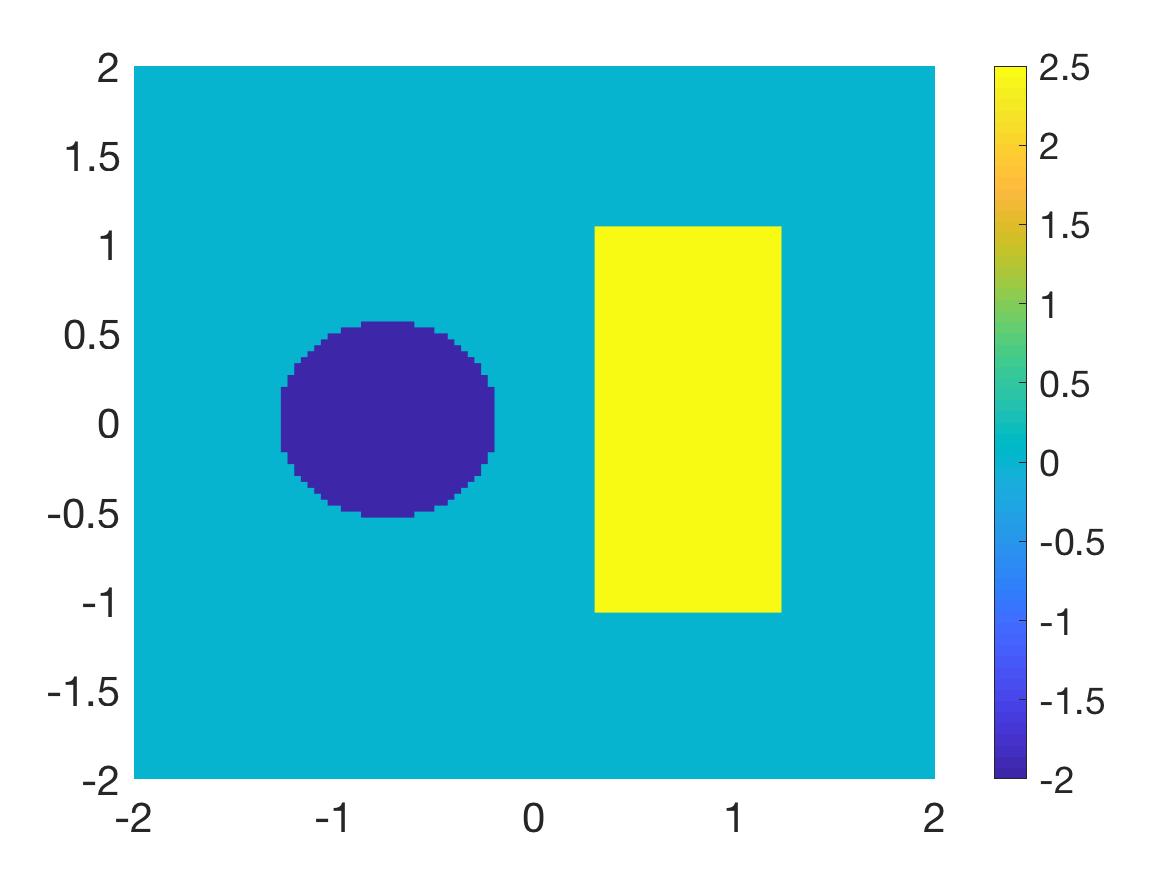}}\quad
\subfloat[]{\includegraphics[width =
.3\textwidth]{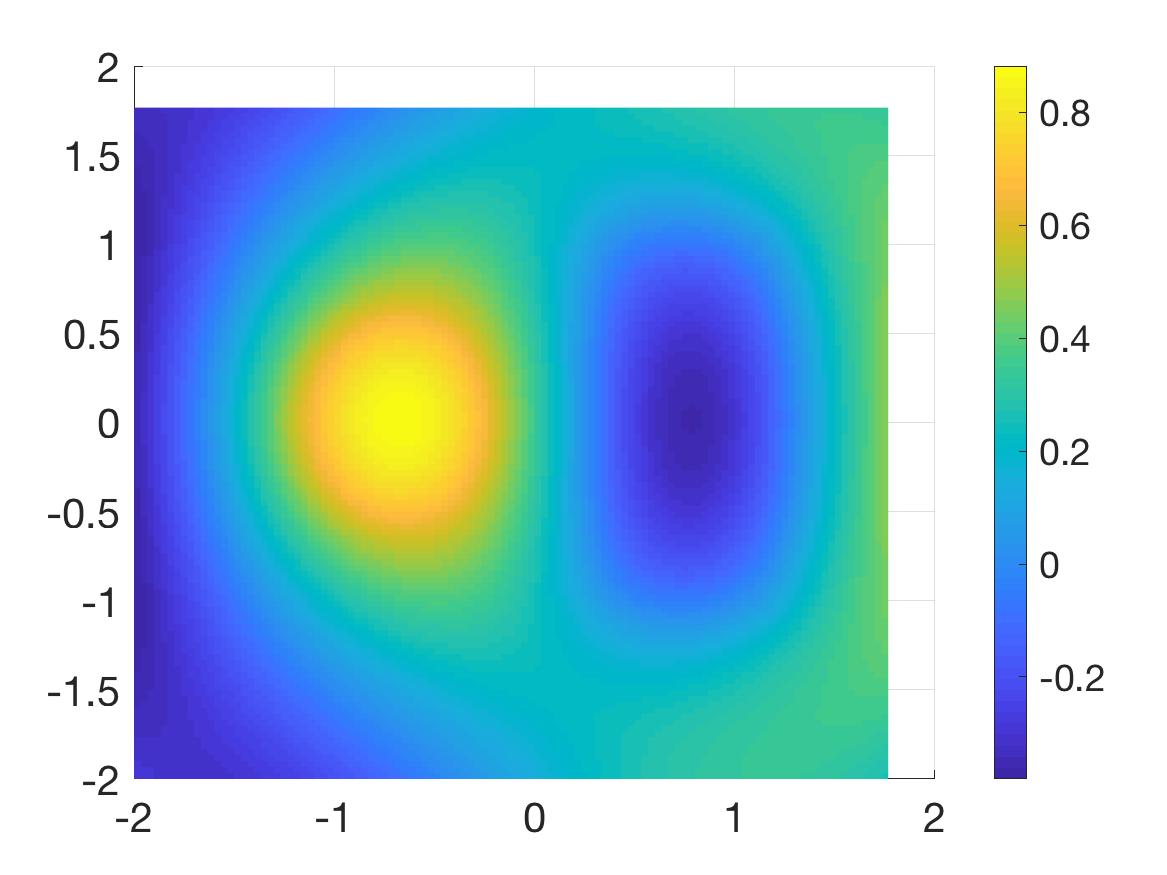}} \quad
\subfloat[]{\includegraphics[width =
.3\textwidth]{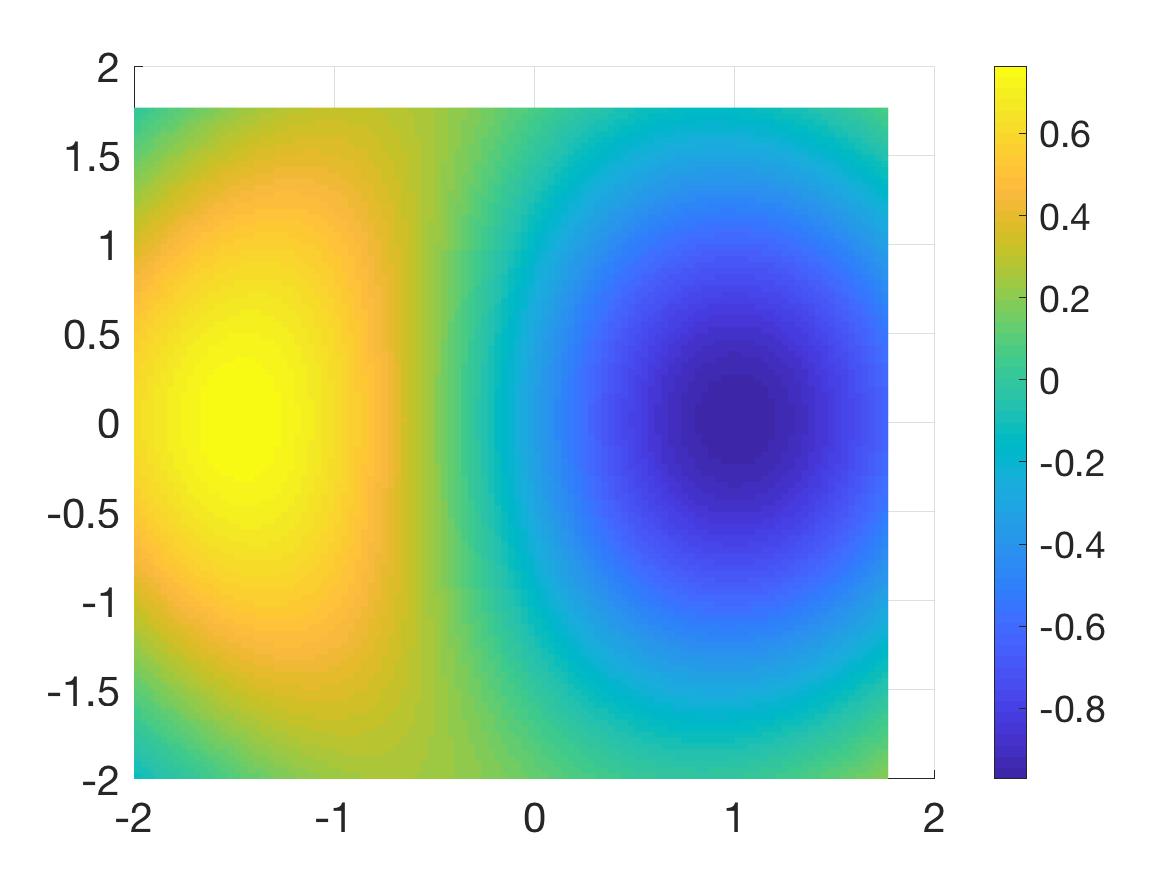}}

\subfloat[]{\includegraphics[width =
.3\textwidth]{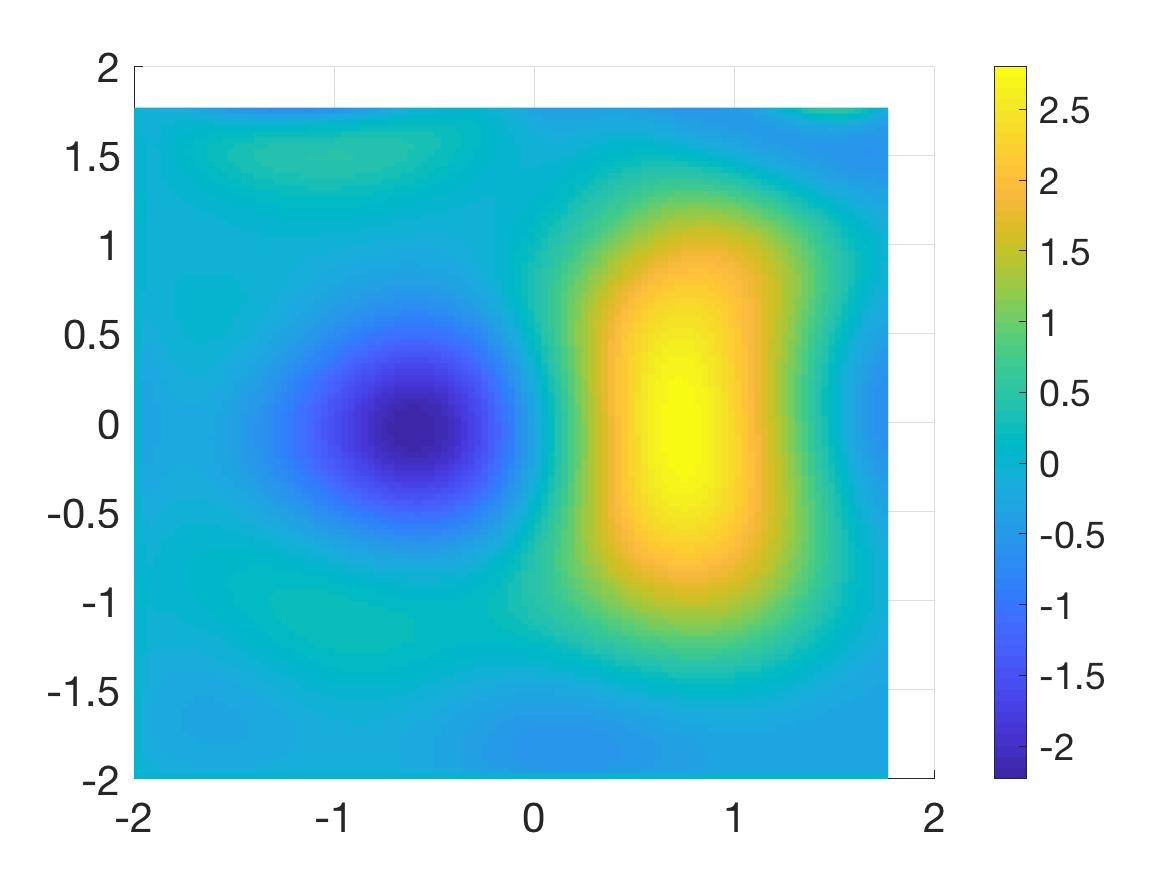}}\quad
\subfloat[]{\includegraphics[width = .3\textwidth]{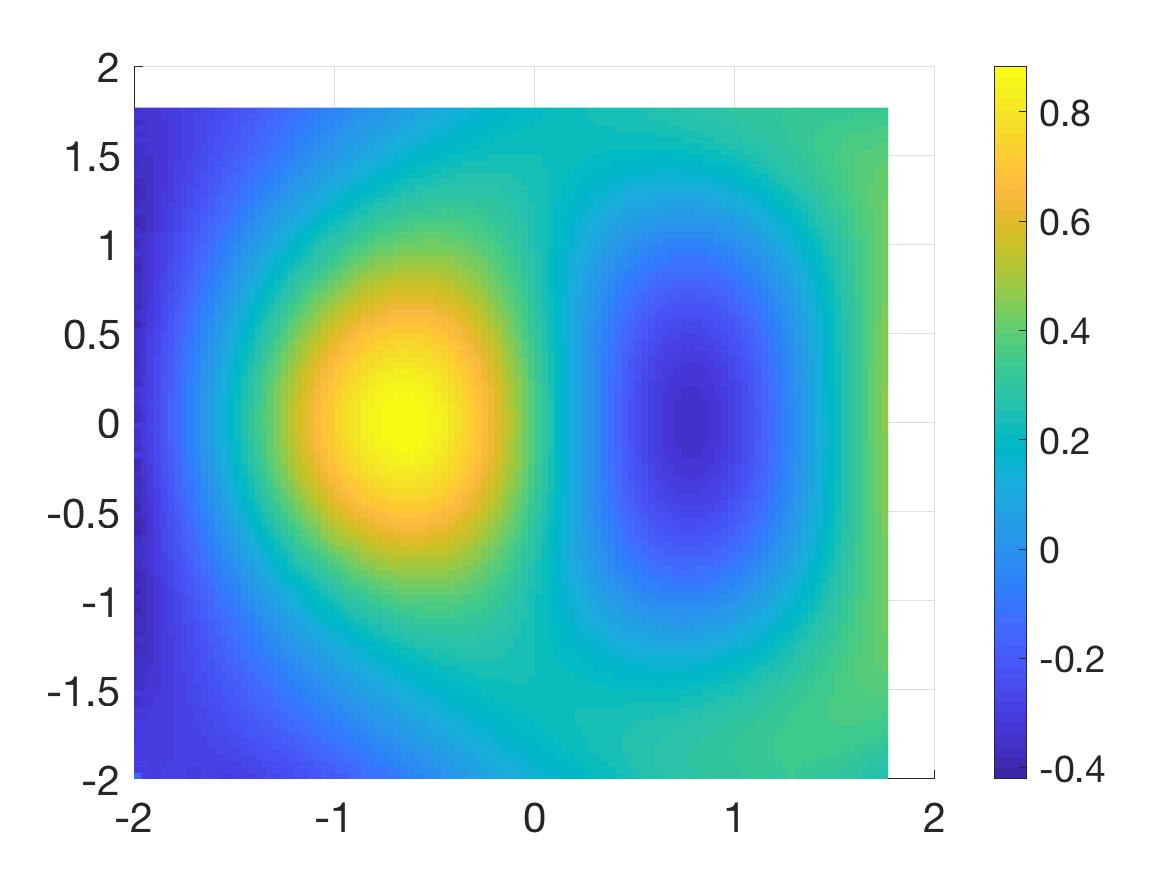}} \quad
\subfloat[]{\includegraphics[width = .3\textwidth]{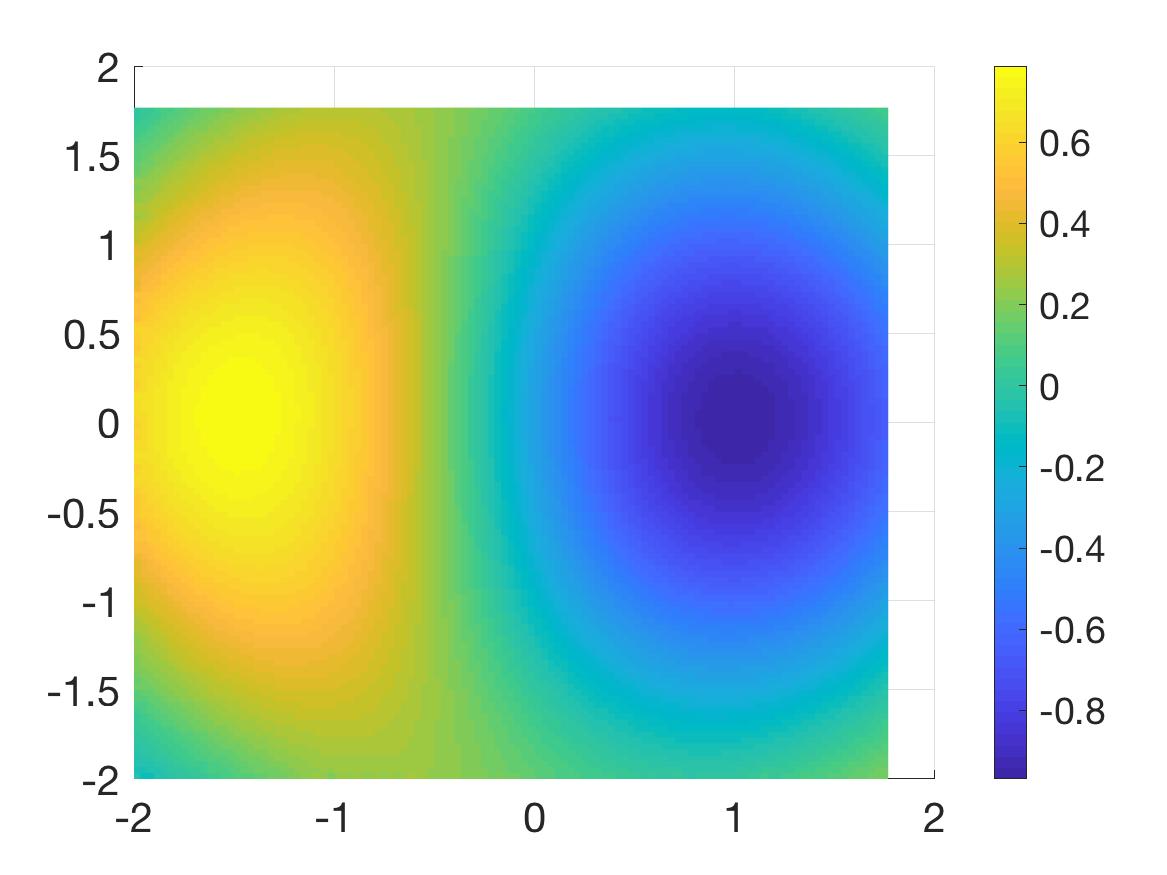}}
\end{center}
\caption{\textit{Test 1. The true and reconstructed source functions and the
true and reconstructed functions $v(\mathbf{x},k)=u(\mathbf{x},k)/g(k)$ when
$k=1.5.$ The reconstructed positive value of the source function is 2.76
(relative error 10.5\%). The reconstructed negative value of the source
function is -2.17 (relative error 8.5\%).
(A) The function $f_{\rm true}$; (B) The real part of the function $v_{\rm true}(\cdot, k = 1.5)$;
(C) The imaginary part of
the function $v_{\rm true}(\cdot, k = 1.5)$;
(D) The function $f_{\rm comp}$;
(E) The real part of the function $v_{\rm comp}(\cdot, k =
1.5)$;
(F) The imaginary part of the function $v_{\rm comp}(\cdot, k =
1.5)$.
}
}
\label{fig model 1}
\end{figure}

It is evident that, for this test, our method for \ref{ISP} provides good
numerical results. The reconstructed locations, shapes as well as the
positive/negative values of the function \textbf{\ }$f^{\mathrm{comp}}$ are
of a good quality.

\item \textit{Test 2.Problem \ref{ISP}. Four circular inclusions.} We
consider the case when the function $f_{\mathrm{true}}$ is given by
\begin{equation*}
\begin{split}
&f_{\mathrm{true}}=\left\{
\begin{array}{rl}
1, & \mbox{if}(x-0.8)^{2}+(y-0.8)^{2}<0.55^{2}\mbox{ or }%
(x+0.8)^{2}+(y-0.8)^{2}<0.55^{2}, \\
-1, & \mbox{if}(x-0.8)^{2}+(y+0.8)^{2}<0.55^{2}\mbox{ or }%
(x+0.8)^{2}+(y+0.8)^{2}<0.55^{2}, \\
0, & \mbox{otherwise,}%
\end{array}%
\right.
\end{split}
\end{equation*}%
and $g_{\mathrm{true}}(k)=1$ for all $k\in \lbrack \underline{k},\overline{k}%
].$ We test the model with four circular inclusions. The source function $%
f=1 $ in the two \textquotedblleft upper" inclusion and $f=-1$ in the two
\textquotedblleft lower" inclusions.

The reconstruction is displayed in Figure \ref{fig model 2}. The source
function is reconstructed well in the sense of locations, shapes and values.

\begin{figure}[h]
\begin{center}
\subfloat[]{\includegraphics[width =
.3\textwidth]{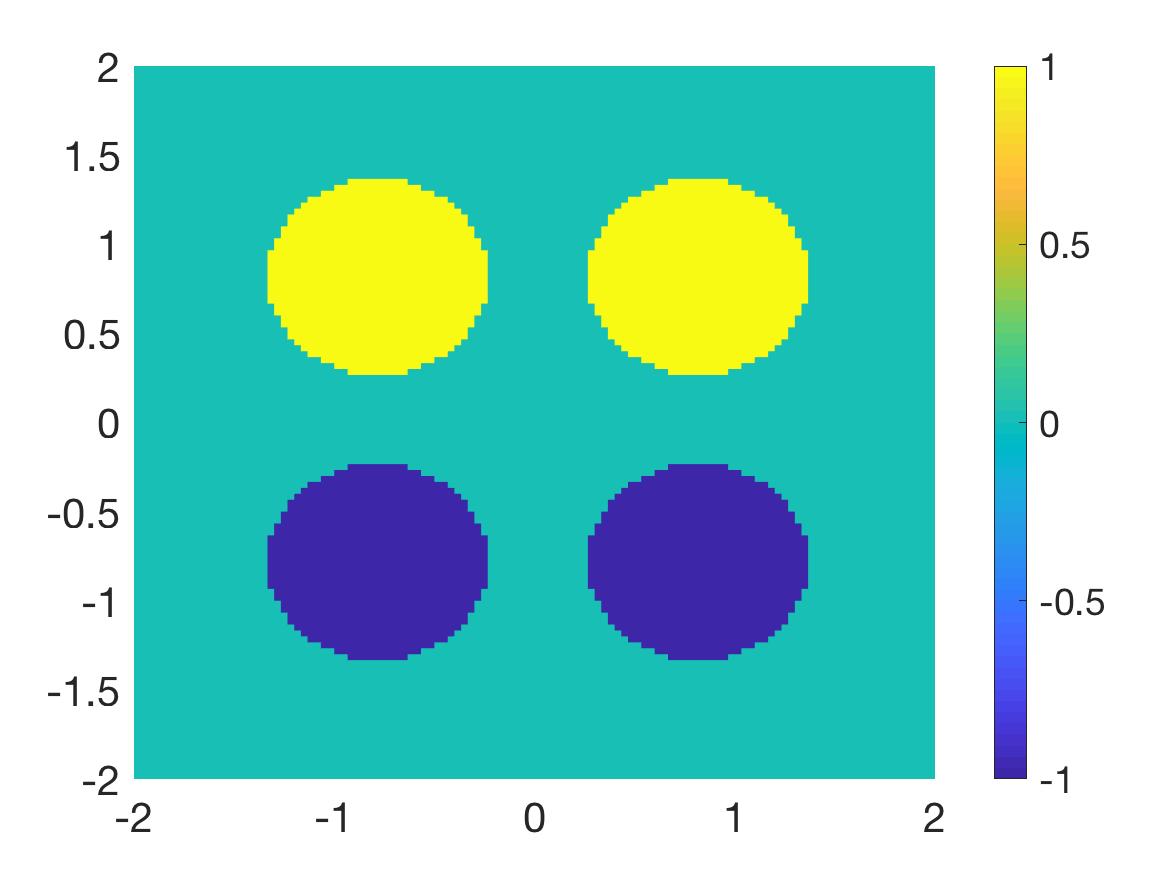}}\quad
\subfloat[]{\includegraphics[width =
.3\textwidth]{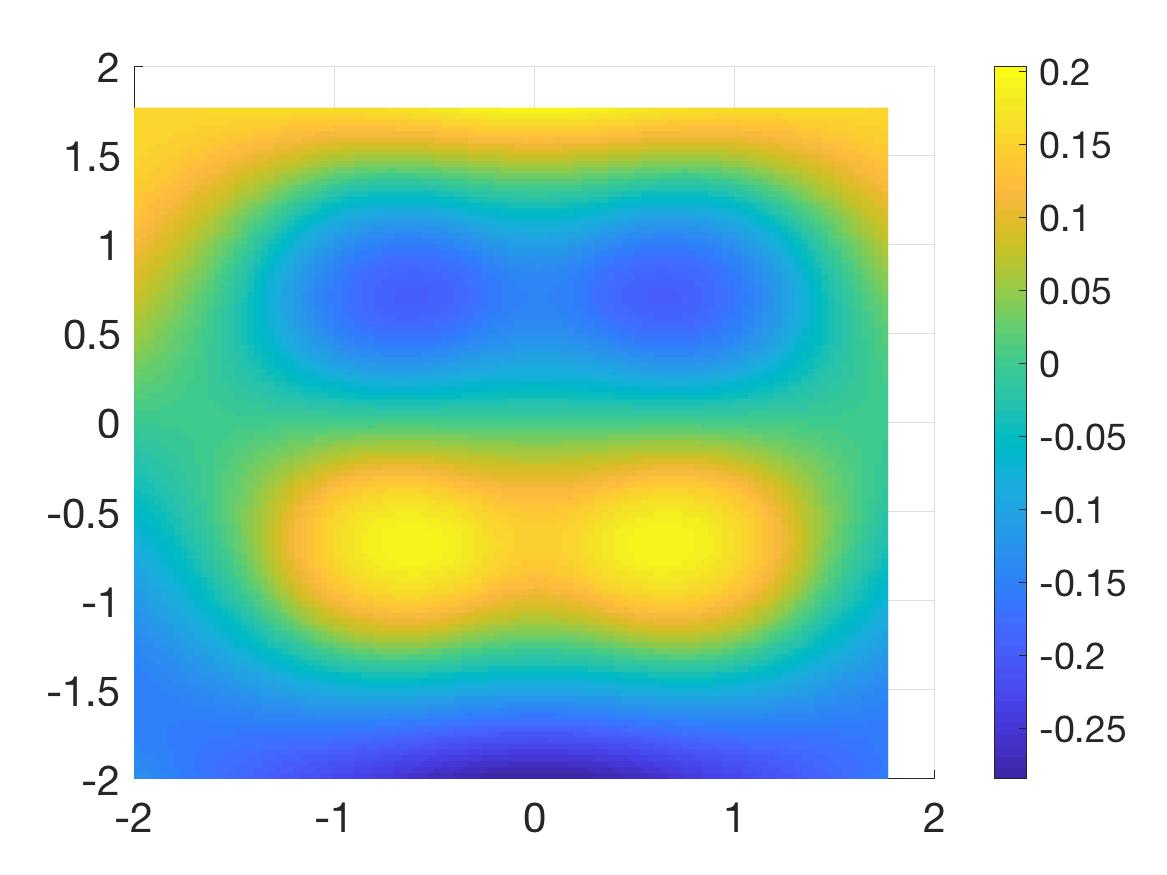}} \quad
\subfloat[]{\includegraphics[width =
.3\textwidth]{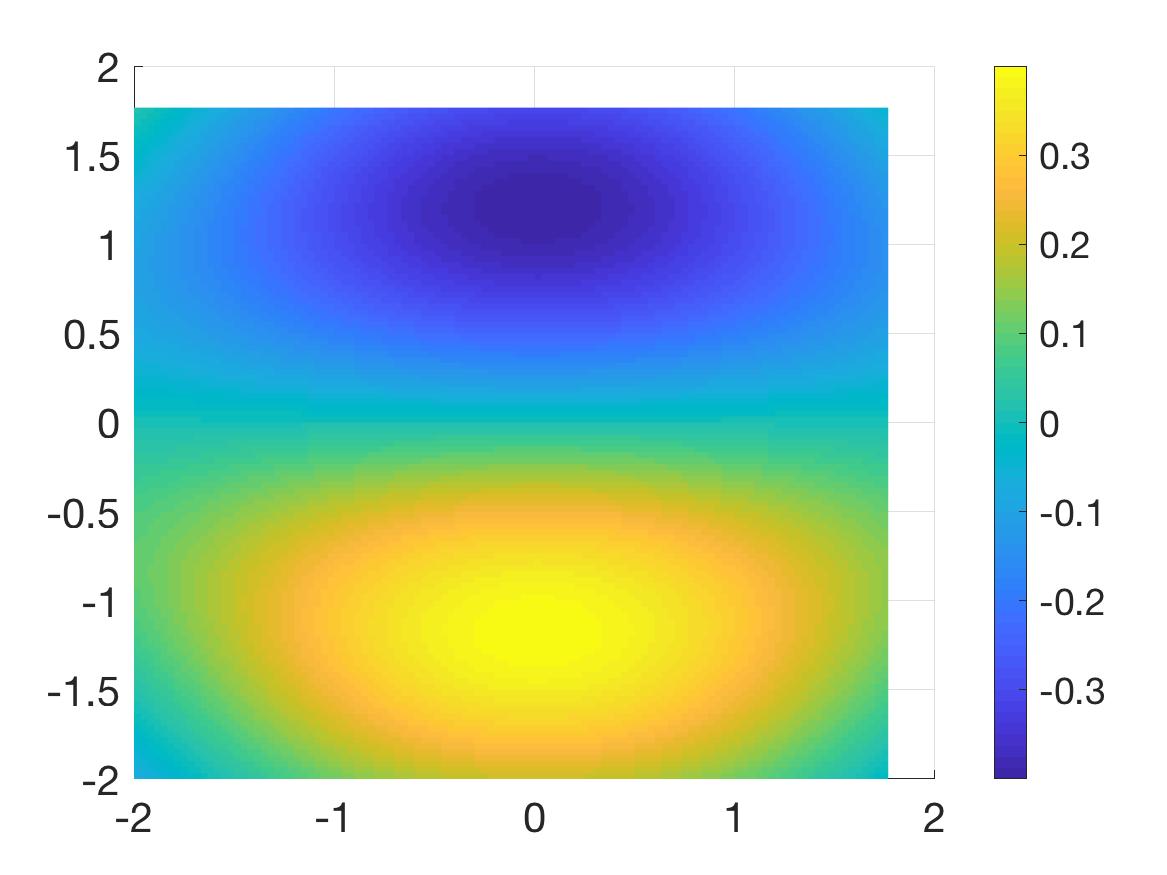}}

\subfloat[]{\includegraphics[width =
.3\textwidth]{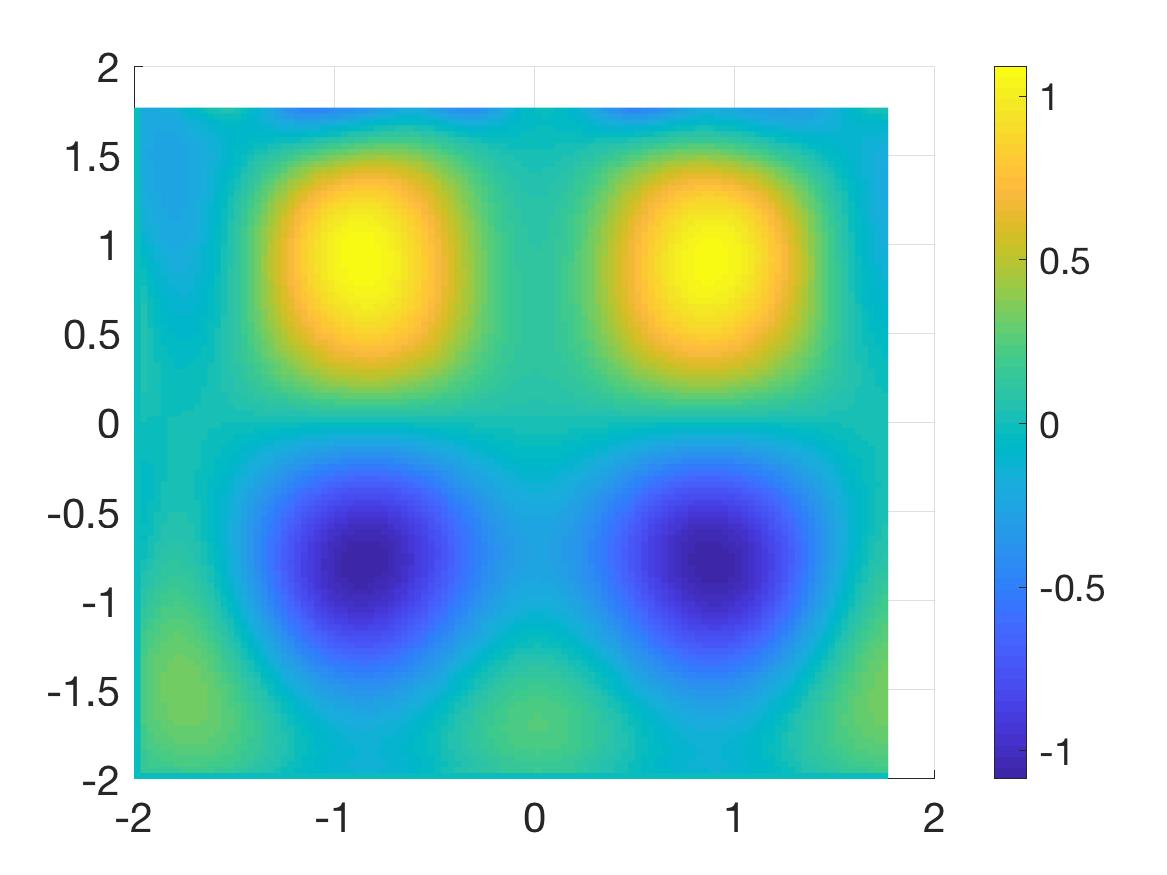}}
\quad
\subfloat[]{\includegraphics[width = .3\textwidth]{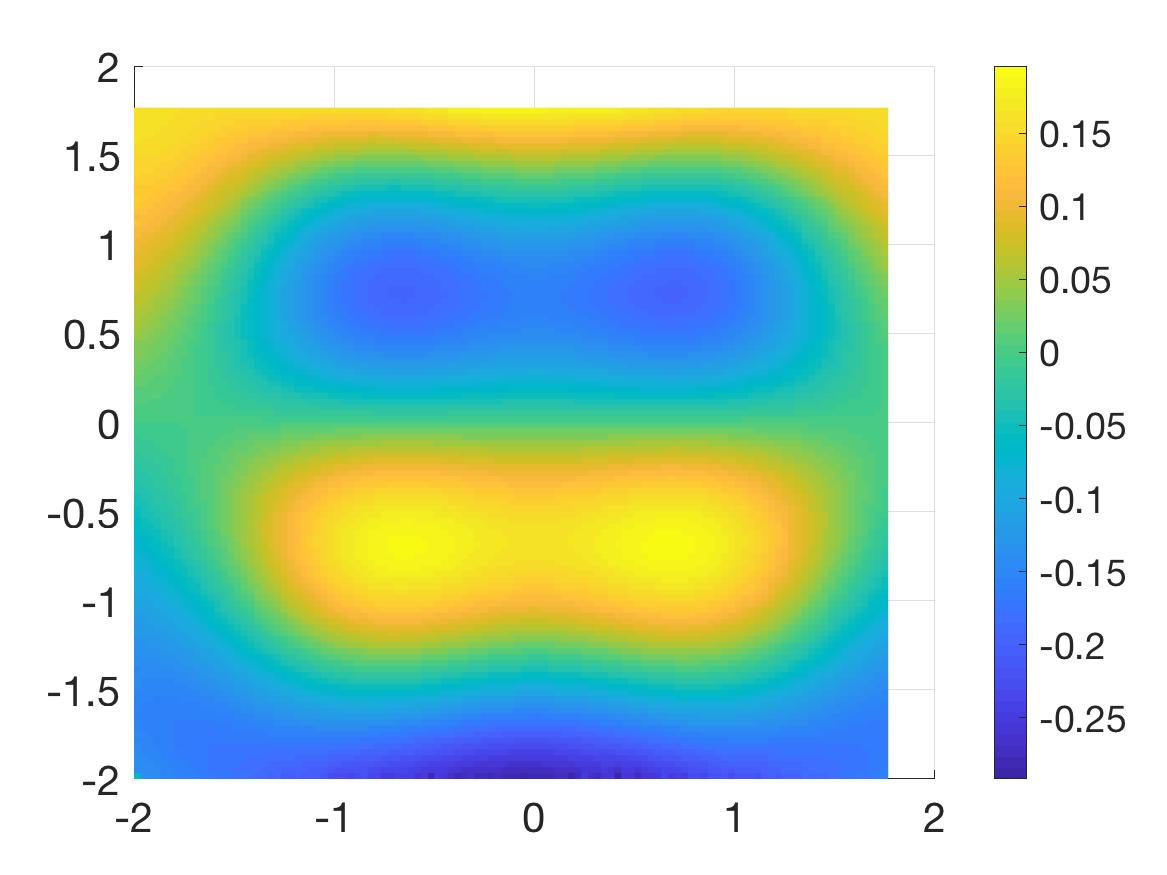}} \quad
\subfloat[]{\includegraphics[width = .3\textwidth]{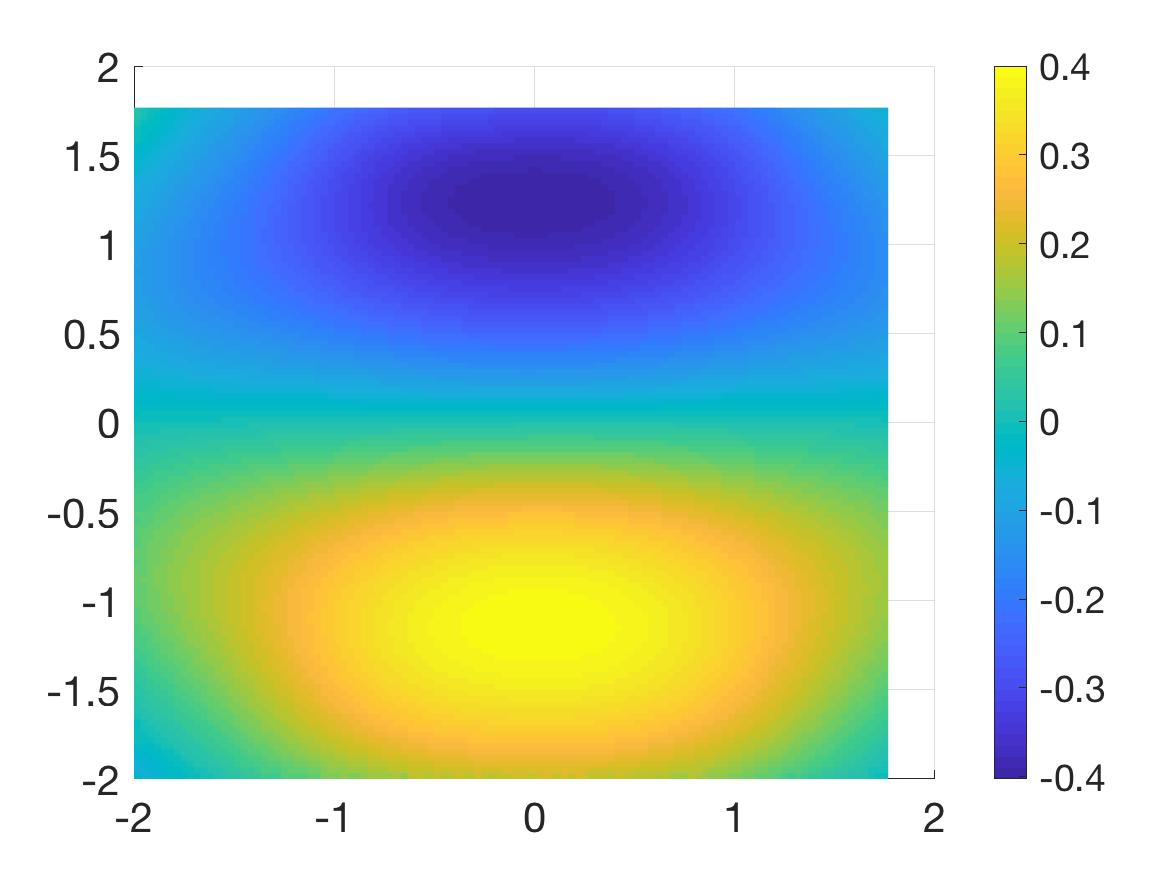}}
\end{center}
\caption{\textit{Test 2. The true and reconstructed source functions and the
true and reconstructed functions $v(\mathbf{x},k)=u(\mathbf{x},k)/g(k)$ when
$k=1.5.$ The reconstructed positive value of the source function is 1.11
(relative error 11.1\%). The reconstructed negative value of the source
function is -1.11 (relative error 11.1\%).
A) The function $f_{\rm true}$; (B) The real part of the function $v_{\rm true}(\cdot, k = 1.5)$;
(C) The imaginary part of
the function $v_{\rm true}(\cdot, k = 1.5)$;
(D) The function $f_{\rm comp}$;
(E) The real part of the function $v_{\rm comp}(\cdot, k =
1.5)$;
(F) The imaginary part of the function $v_{\rm comp}(\cdot, k =
1.5)$.
}}
\label{fig model 2}
\end{figure}

\item \textit{Test 3. Problem \ref{ISP1}. A void in the square.} We consider
the case when the negative part of the true source function $f$ is
surrounded by a square and $f$ is positive in this square. More precisely,
\begin{equation*}
f_{\mathrm{true}}=\left\{
\begin{array}{rl}
1 & \mbox{if }\max \{|x|,|y|\}<1.2\mbox{ and }x^{2}+y^{2}\geq 0.48^{2}, \\
-1 & \mbox{if }x^{2}+y^{2}<0.48^{2}, \\
0 & \mbox{otherwise,}%
\end{array}%
\right.
\end{equation*}%
and $g_{\mathrm{true}}(k)=k$ for all $k\in \lbrack \underline{k},\overline{k}%
].$

\begin{figure}[h]
\begin{center}
\subfloat[]{\includegraphics[width =
.3\textwidth]{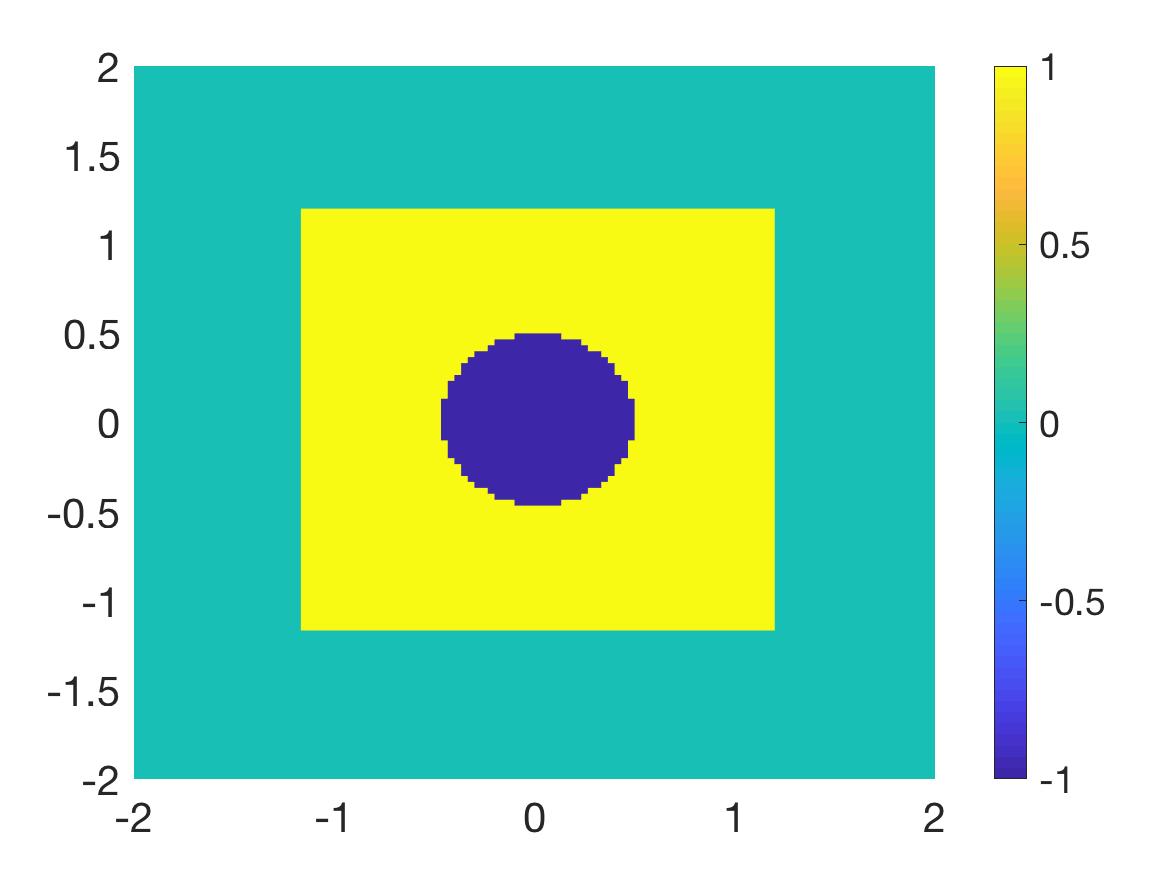}}\quad
\subfloat[]{\includegraphics[width =
.3\textwidth]{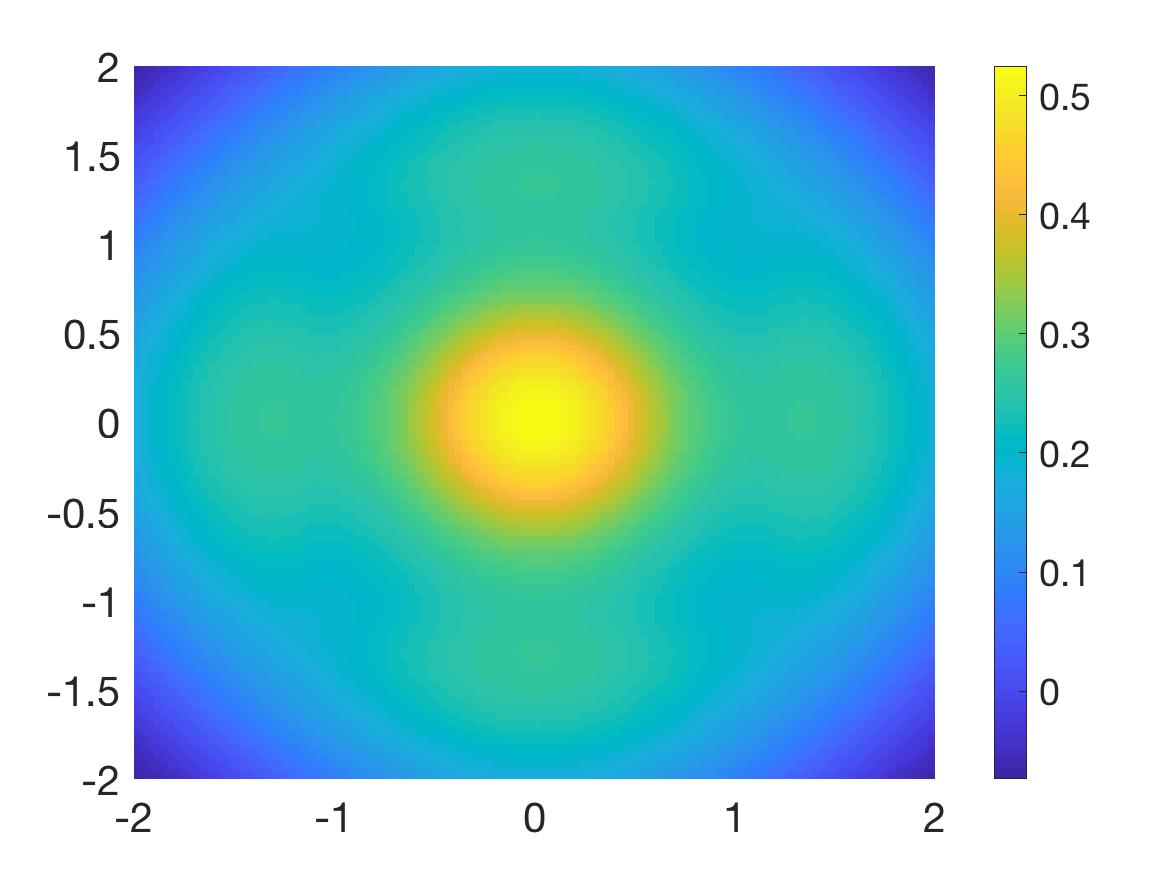}} \quad
\subfloat[]{\includegraphics[width =
.3\textwidth]{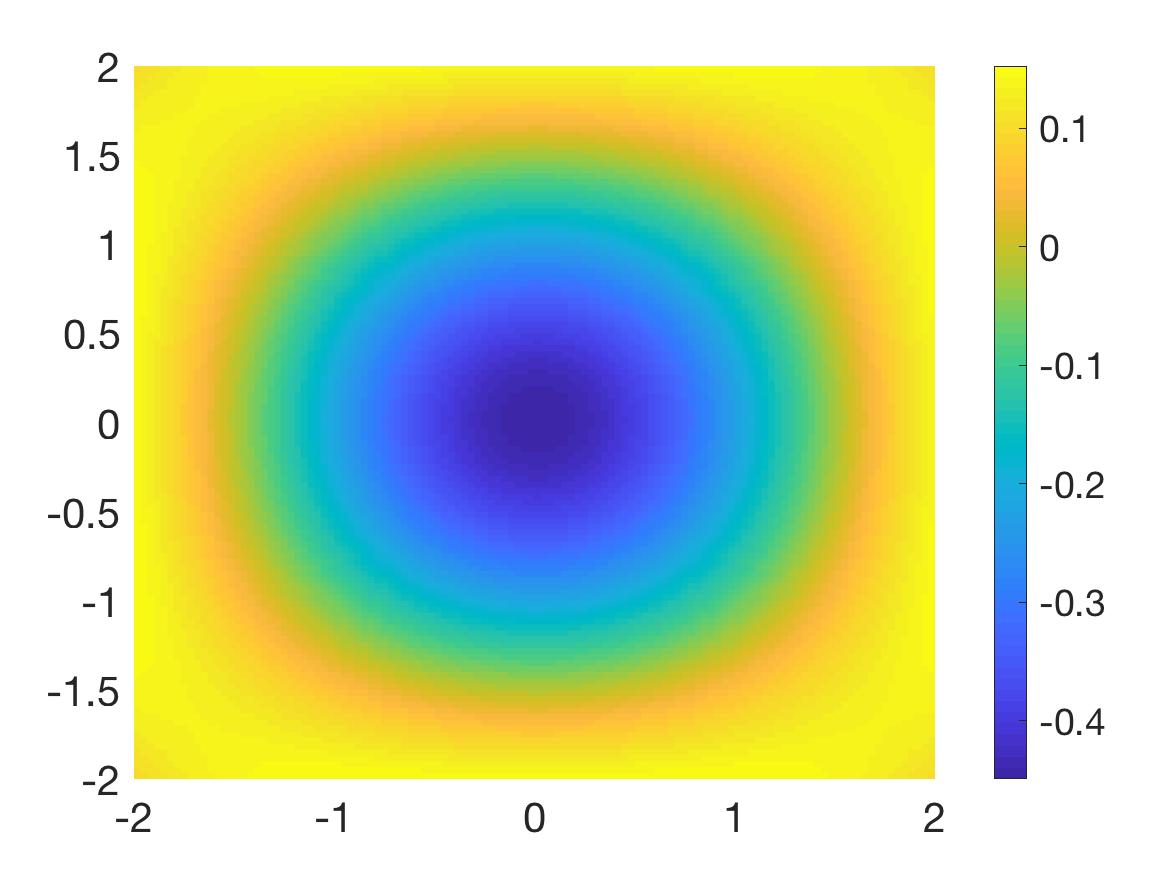}}

\subfloat[]{\includegraphics[width =
.3\textwidth]{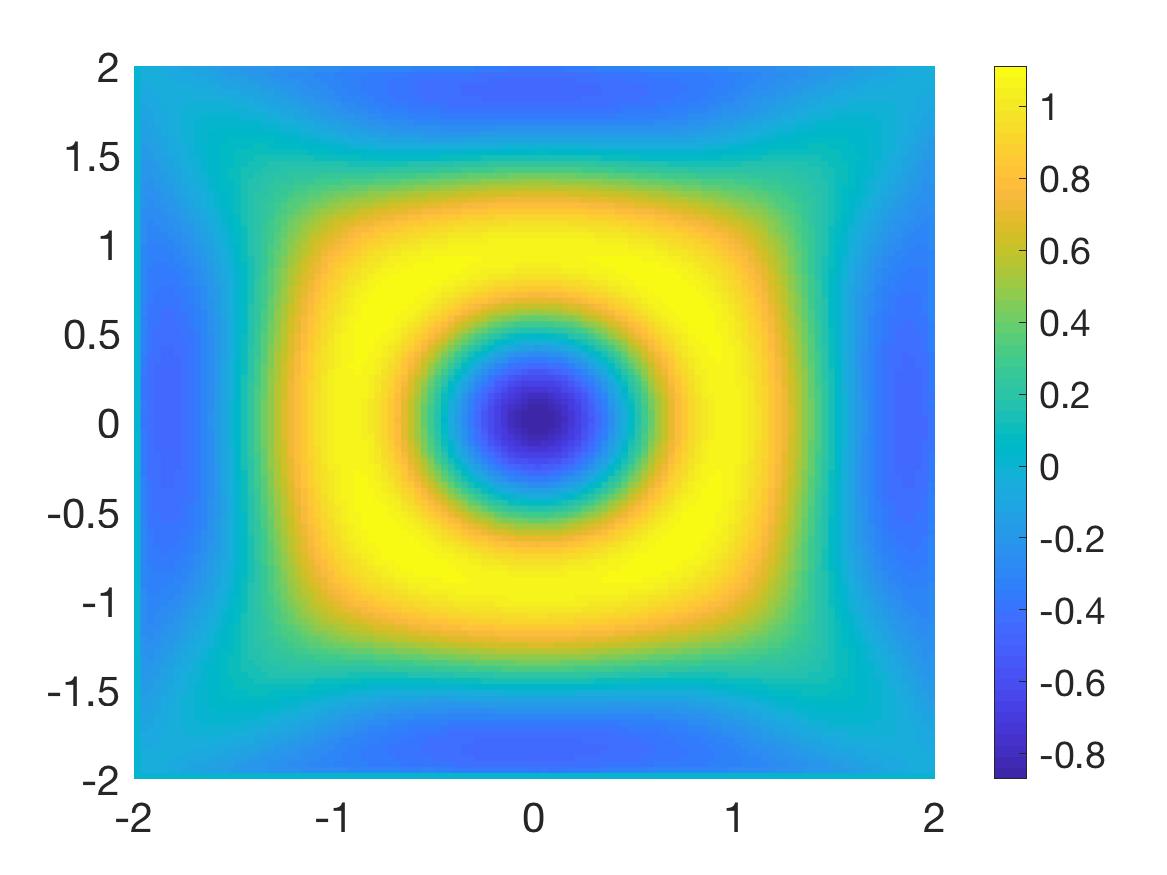}}\quad
\subfloat[]{\includegraphics[width = .3\textwidth]{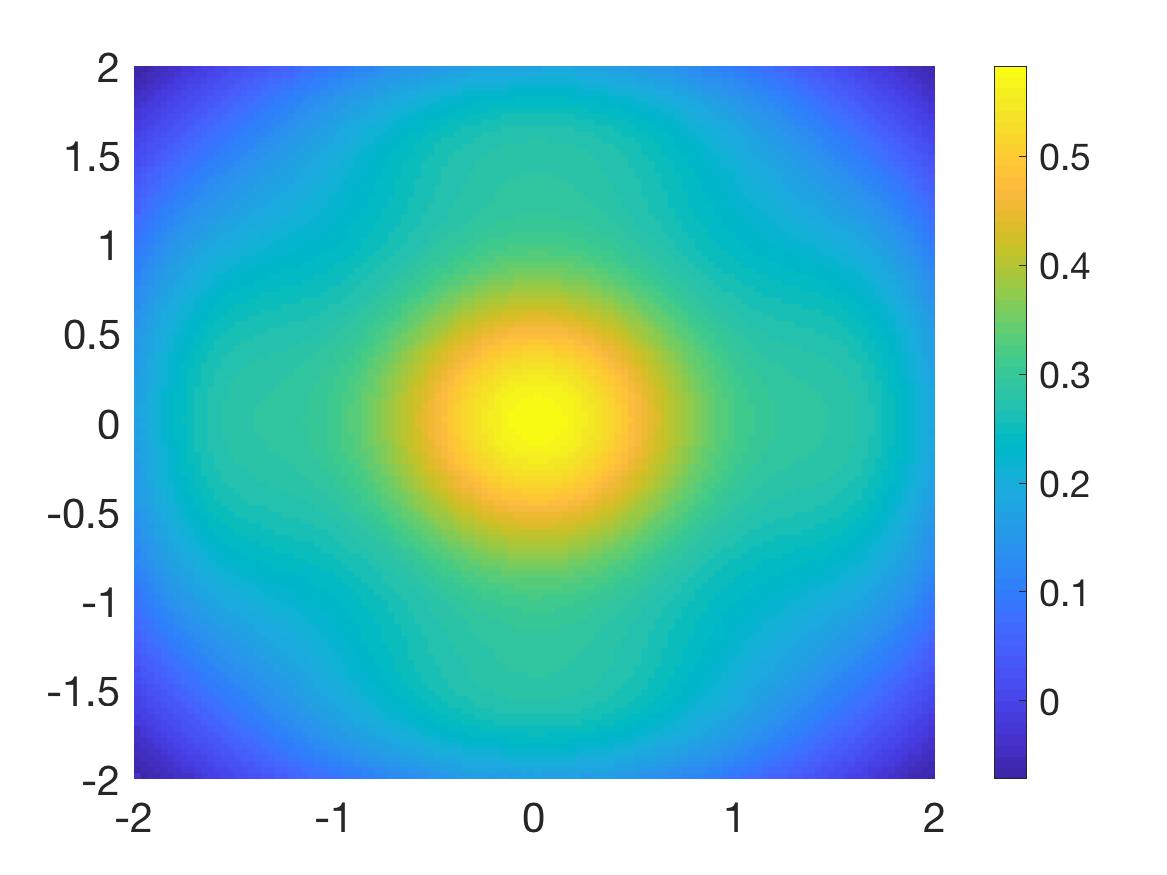}}\quad
\subfloat[]{\includegraphics[width = .3\textwidth]{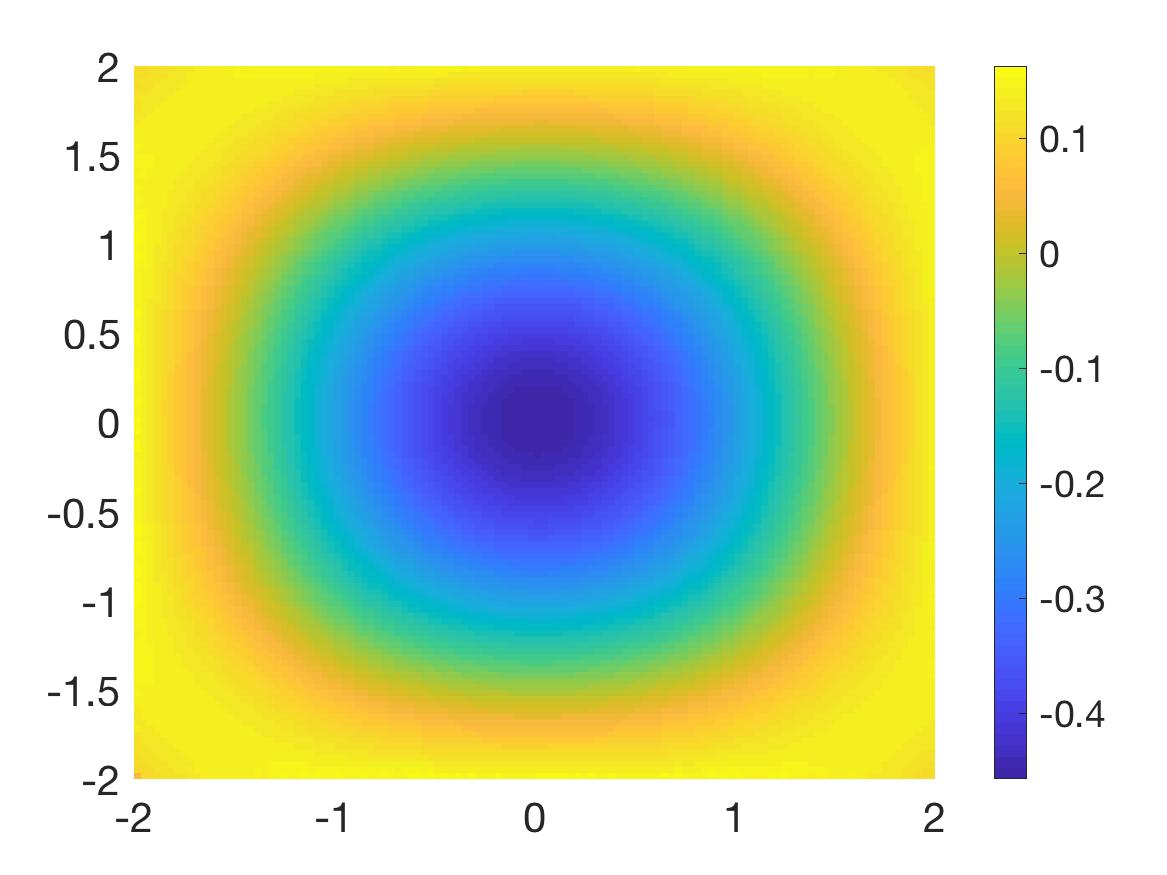}}
\end{center}
\caption{\textit{Test 3. The true and reconstructed source functions and the
true and reconstructed functions $v(\mathbf{x},k)=u(\mathbf{x},k)/g(k)$ when
$k=1.5.$ The reconstructed positive value of the source function is 1.09
(relative error 9.0\%). The reconstructed negative value of the source
function is -0.89 (relative error 11.0\%).
A) The function $f_{\rm true}$; (B) The real part of the function $v_{\rm true}(\cdot, k = 1.5)$;
(C) The imaginary part of
the function $v_{\rm true}(\cdot, k = 1.5)$;
(D) The function $f_{\rm comp}$;
(E) The real part of the function $v_{\rm comp}(\cdot, k =
1.5)$;
(F) The imaginary part of the function $v_{\rm comp}(\cdot, k =
1.5)$.}}
\label{fig model 3}
\end{figure}

The true $f_{\text{true}}$ and computed $f^{\text{comp}}$ source functions
are displayed in Figure \ref{fig model 3}. We can see computed shapes of the
\textquotedblleft positive" square and the \textquotedblleft negative" disk
are quite acceptable. Given that the noise in the data is 5\%, errors in
values of the function $f^{\text{comp}}$ are also acceptable.

\item \textit{Test 4. Problem \ref{ISP1}. Ring.} We consider a model that is
similar to that in the previous test. The main difference is the ``outer
positive" part of the true source function is a ring rather than a square.
The function $f_{\mathrm{true}}$ is
\begin{equation}
f_{\mathrm{true}}=\left\{
\begin{array}{rl}
1 & \mbox{if }0.52^{2}<x^{2}+y^{2}<1.2^{2}, \\
-2 & \mbox{if }x^{2}+y^{2}\leq 0.52^{2}, \\
0 & \mbox{otherwise,}%
\end{array}%
\right.
\end{equation}%
and $g_{\mathrm{true}}(k)=k^{2}$ for all $k\in \lbrack \underline{k},%
\overline{k}].$

\begin{figure}[h]
\begin{center}
\subfloat[]{\includegraphics[width =
.3\textwidth]{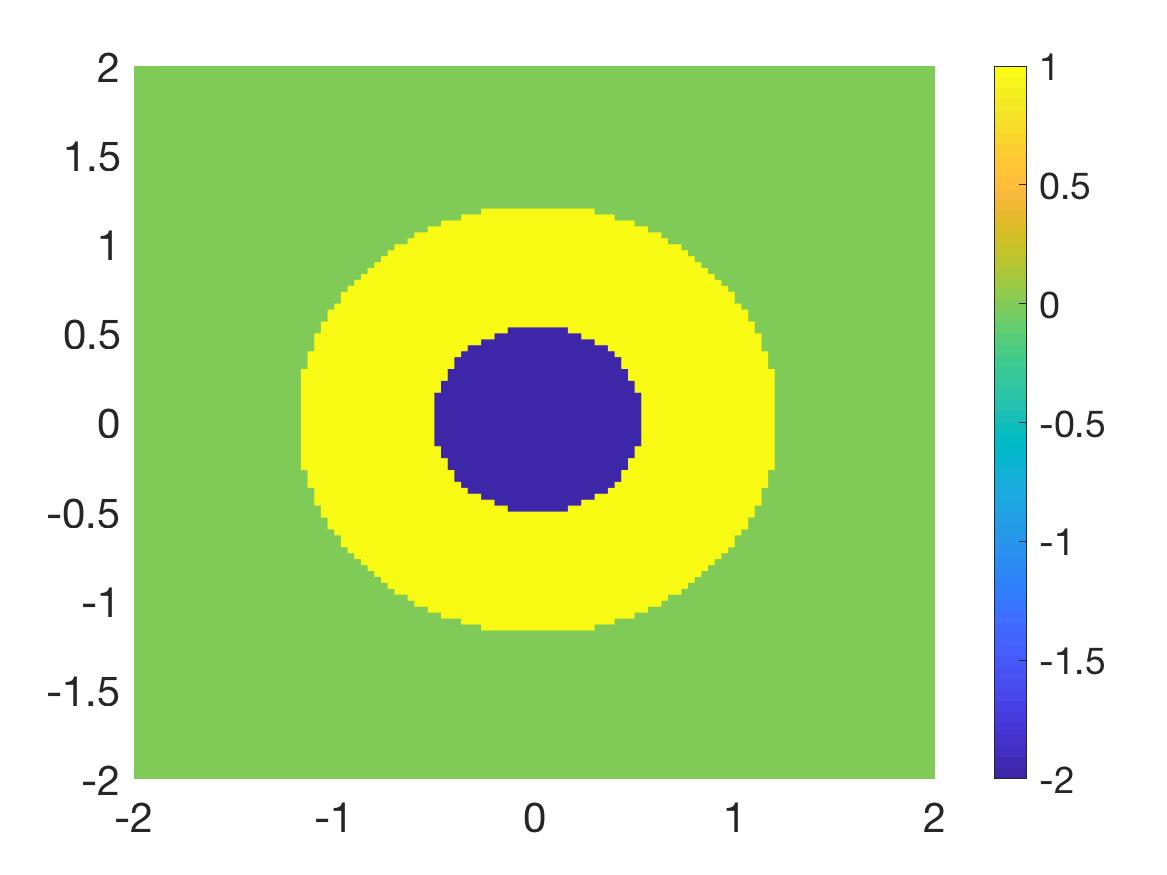}}\quad
\subfloat[]{\includegraphics[width =
.3\textwidth]{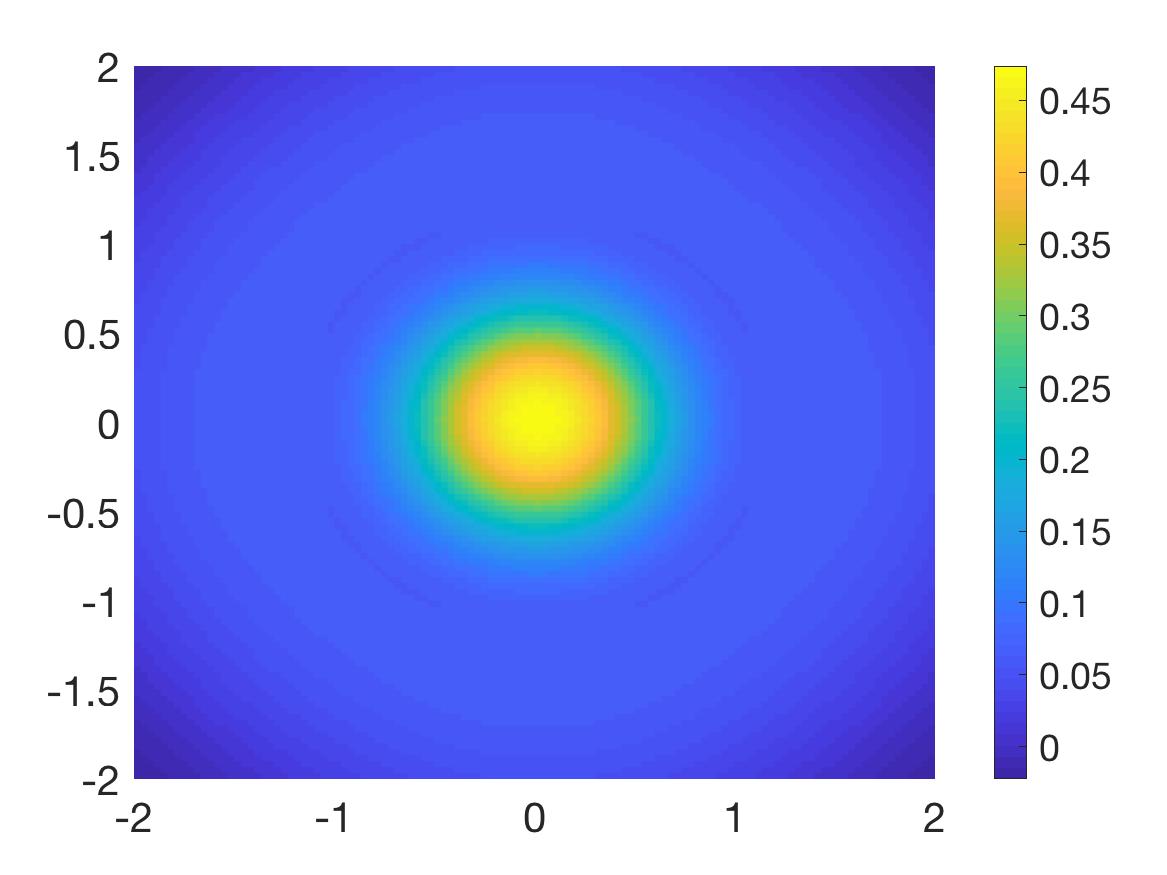}} \quad
\subfloat[]{\includegraphics[width =
.3\textwidth]{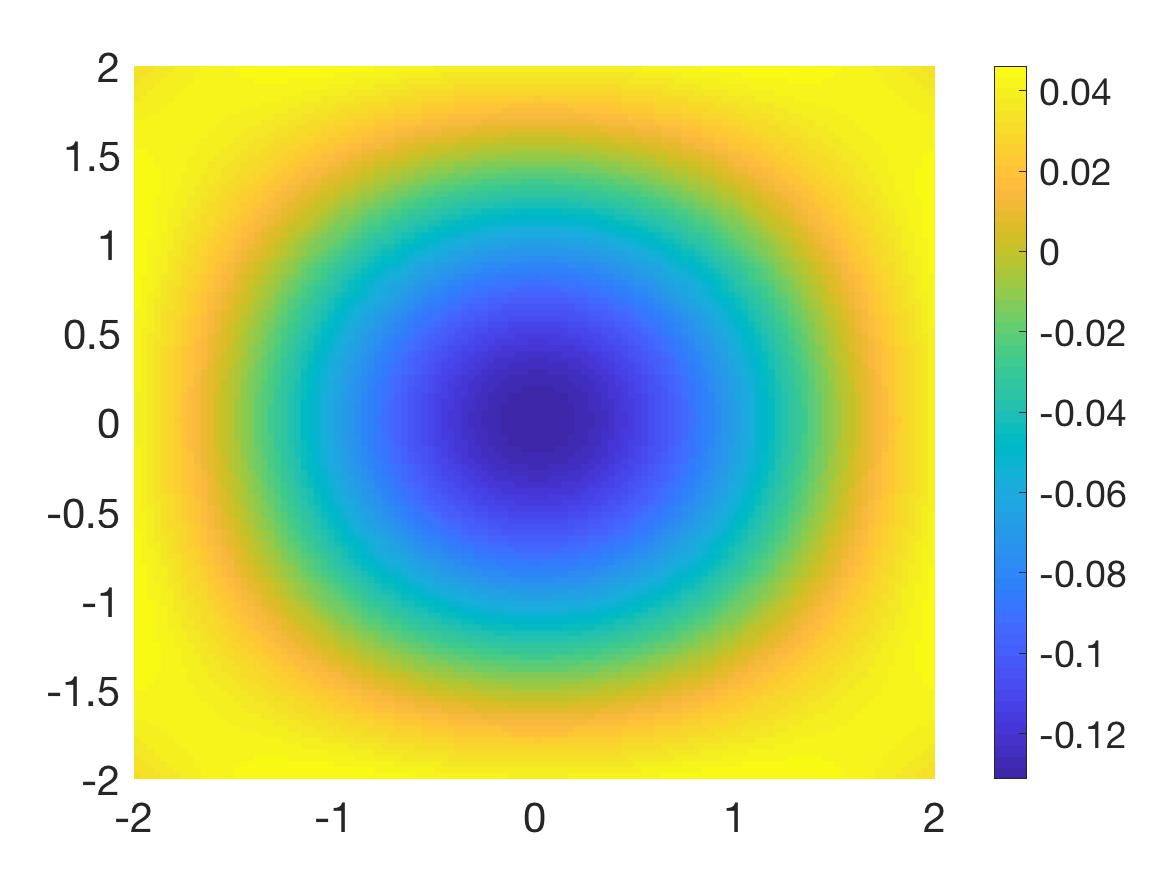}}

\subfloat[]{\includegraphics[width =
.3\textwidth]{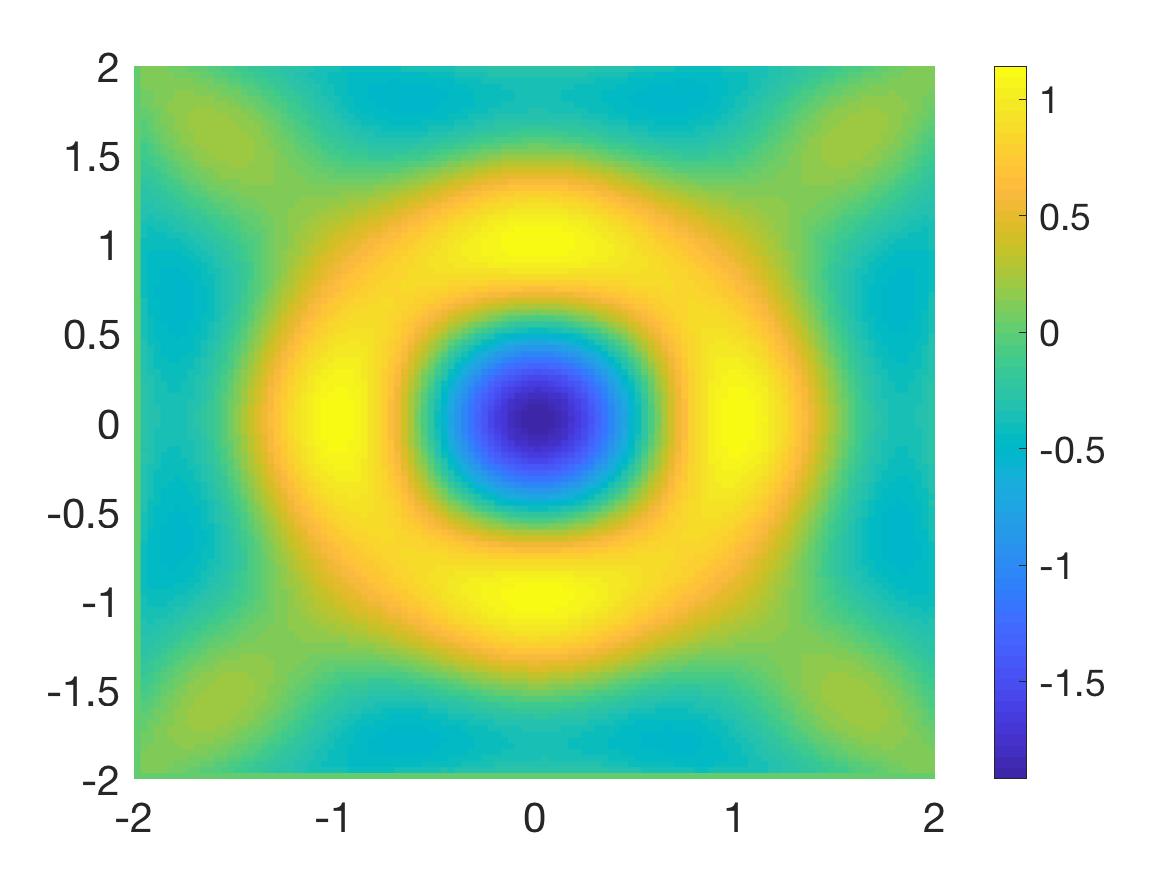}}
\quad
\subfloat[]{\includegraphics[width = .3\textwidth]{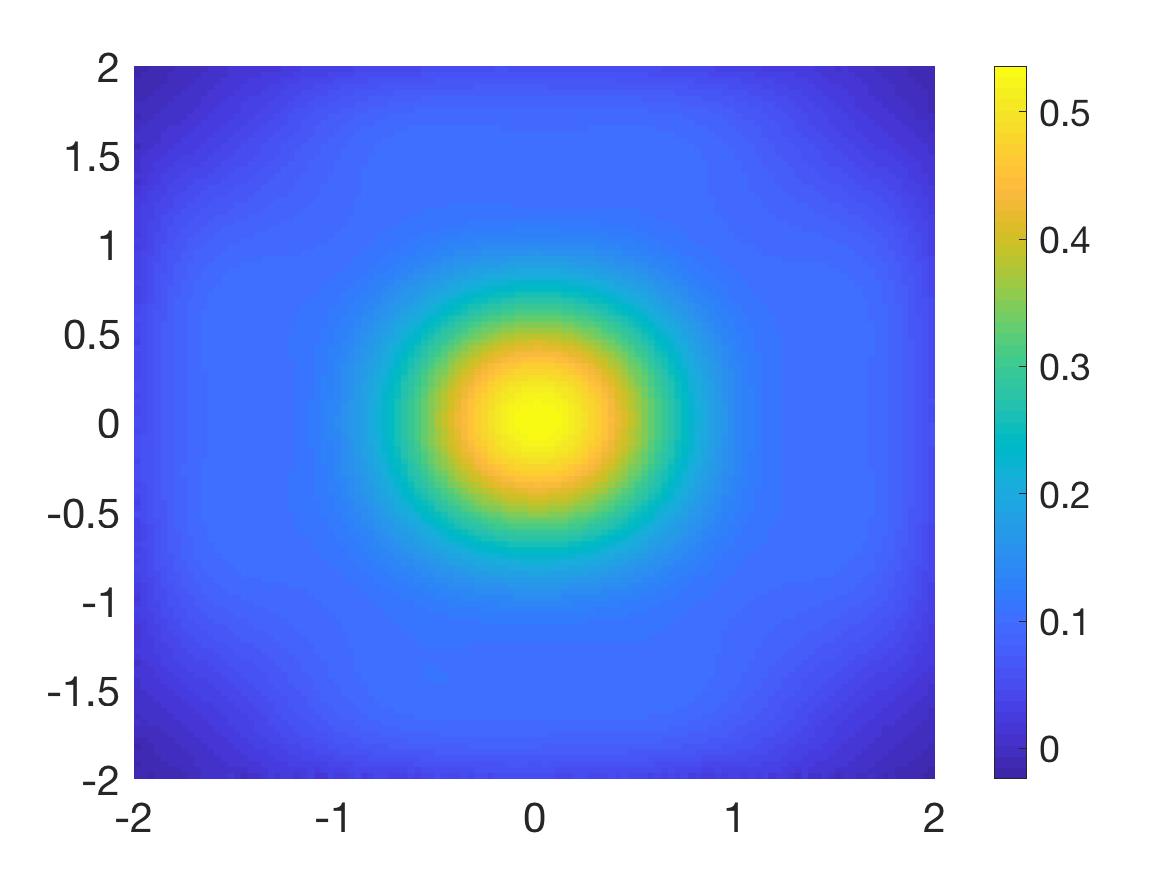}} \quad
\subfloat[]{\includegraphics[width = .3\textwidth]{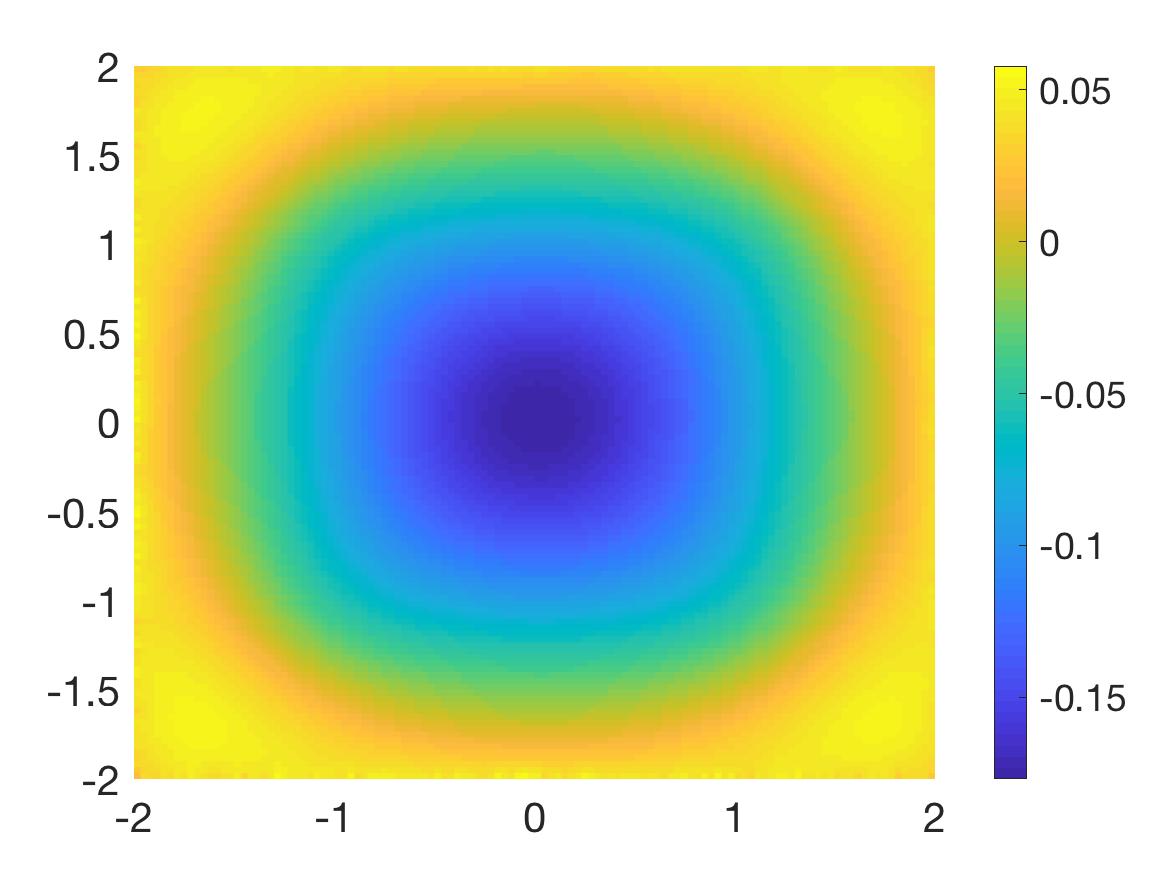}}
\end{center}
\caption{\textit{Test 4. The true and reconstructed source functions and the
true and reconstructed functions $v(\mathbf{x},k)=u(\mathbf{x},k)/g(k)$ when
$k=1.5.$ The reconstructed positive value of the source function is 1.12
(relative error 12.0\%). The reconstructed negative value of the source
function is -1.94 (relative error 3.0\%).
A) The function $f_{\rm true}$; (B) The real part of the function $v_{\rm true}(\cdot, k = 1.5)$;
(C) The imaginary part of
the function $v_{\rm true}(\cdot, k = 1.5)$;
(D) The function $f_{\rm comp}$;
(E) The real part of the function $v_{\rm comp}(\cdot, k =
1.5)$;
(F) The imaginary part of the function $v_{\rm comp}(\cdot, k =
1.5)$.}}
\label{fig model 4}
\end{figure}

In Figure \ref{fig model 4}, one can see that the source function is
computed rather accurately. The values of both \textquotedblleft positive"
and \textquotedblleft negative" parts of the inclusion are computed with a
good accuracy.

\item \textit{Test 5. \ref{ISP1}. Continuous surface.} We take for $\left(
x,y\right) \in \Omega $
\begin{equation*}
f_{\mathrm{true}%
}=3(1-x)^{2}e^{-x^{2}}-(y+1)^{2}-10(x/5-x^{3}-y^{5})e^{-x^{2}-y^{2}}-1/3e^{-(x+1)^{2}-y^{2}},
\end{equation*}%
which is the function \textquotedblleft peaks" built-in Matlab, restricted
on $\Omega $. This function is interesting since its support is not
compactly contained in $\Omega $ and its graph behaves as a surface rather
than the \textquotedblleft inclusion" from the previous tests. We set $g_{%
\mathrm{true}}(k)=\sin (k)+2$ for all $k\in \lbrack \underline{k},\overline{k%
}].$
\begin{figure}[h]
\begin{center}
\subfloat[]{\includegraphics[width =
.3\textwidth]{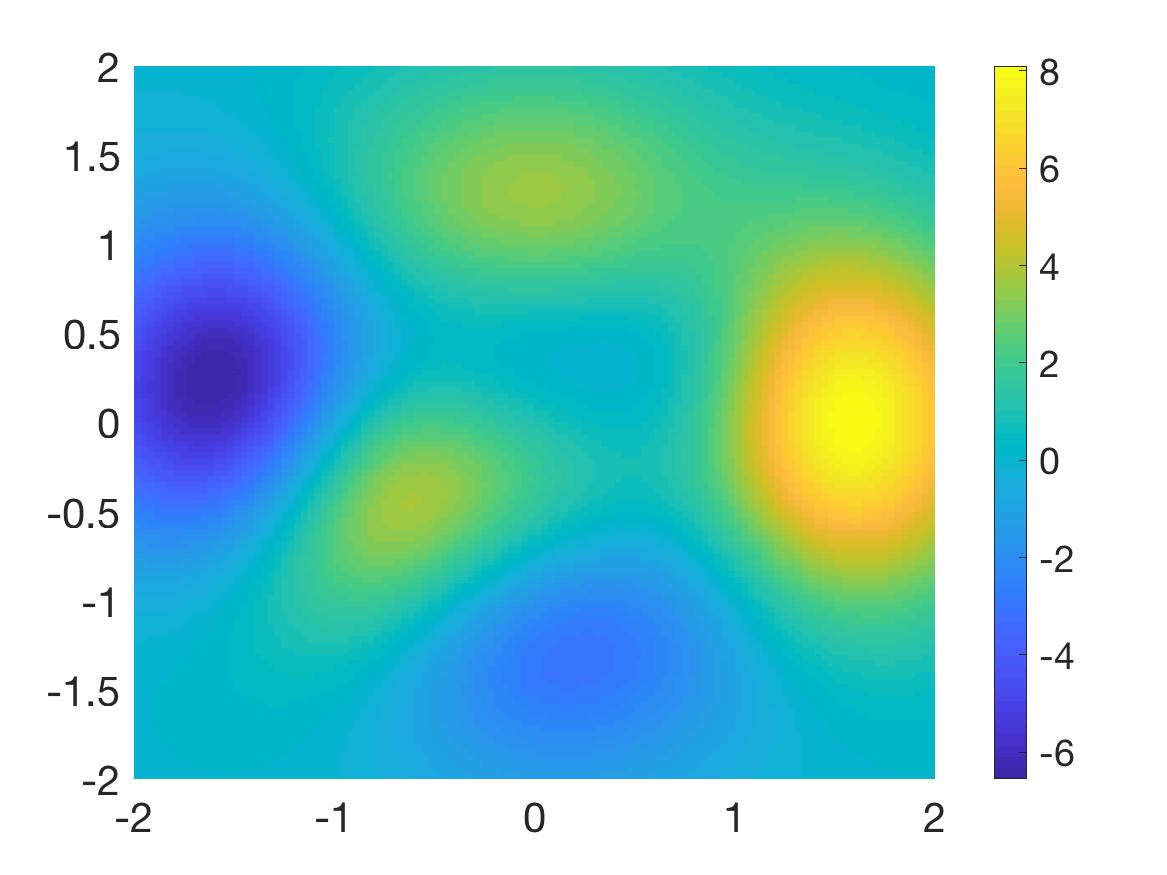}}\quad
\subfloat[]{\includegraphics[width =
.3\textwidth]{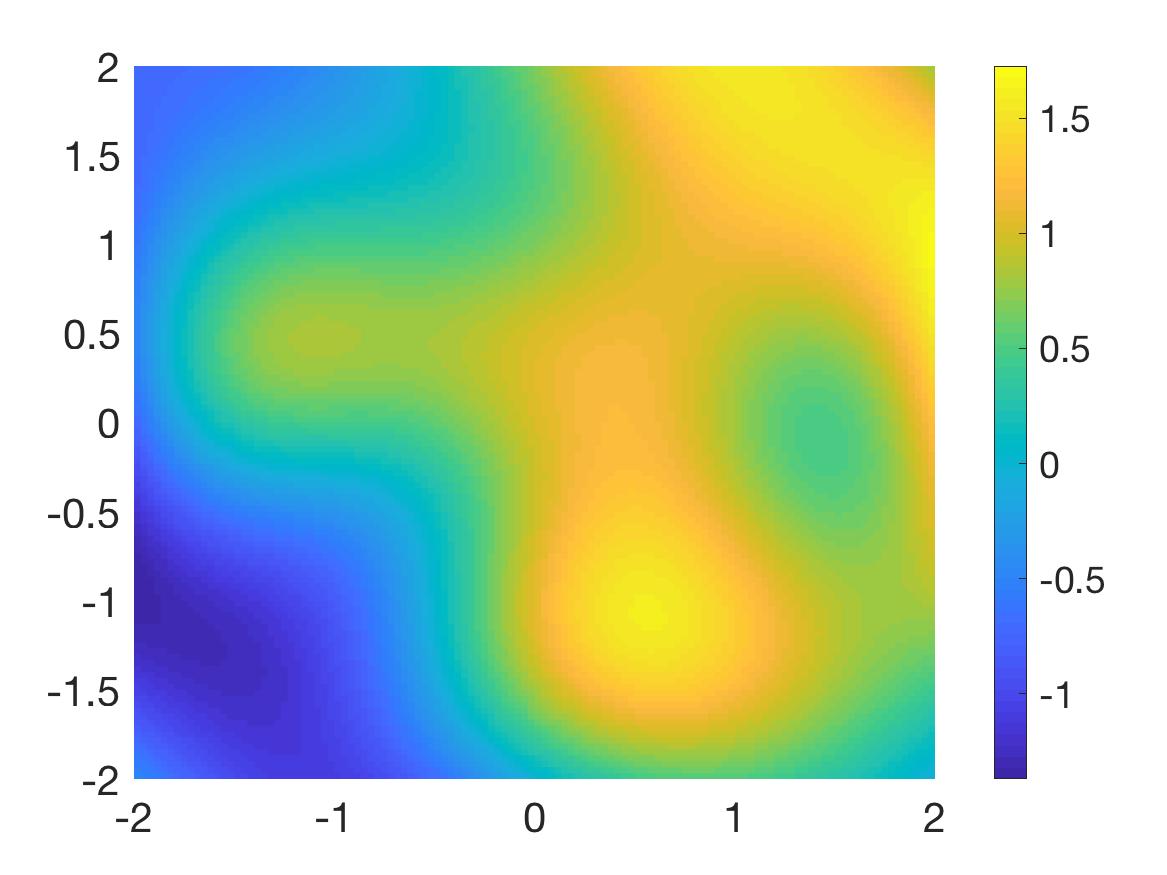}} \quad
\subfloat[]{\includegraphics[width =
.3\textwidth]{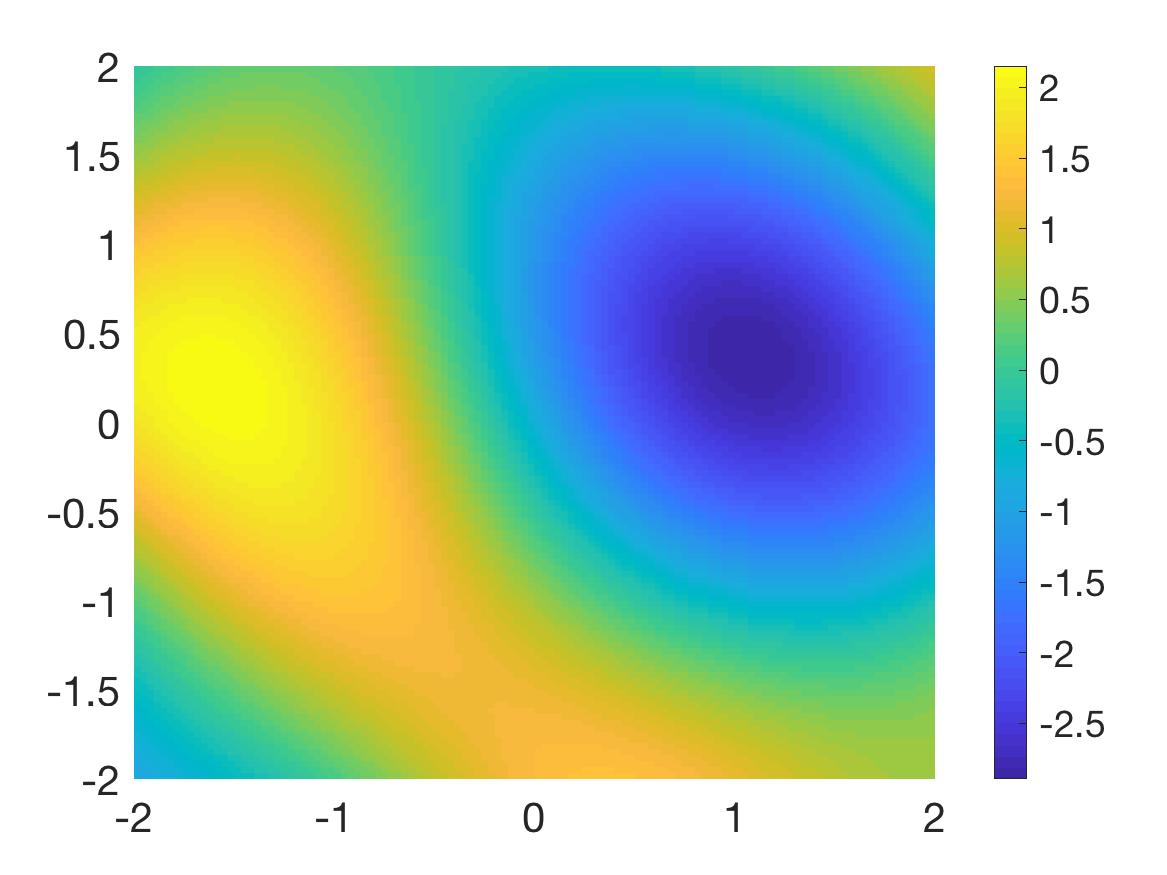}}

\subfloat[]{\includegraphics[width =
.3\textwidth]{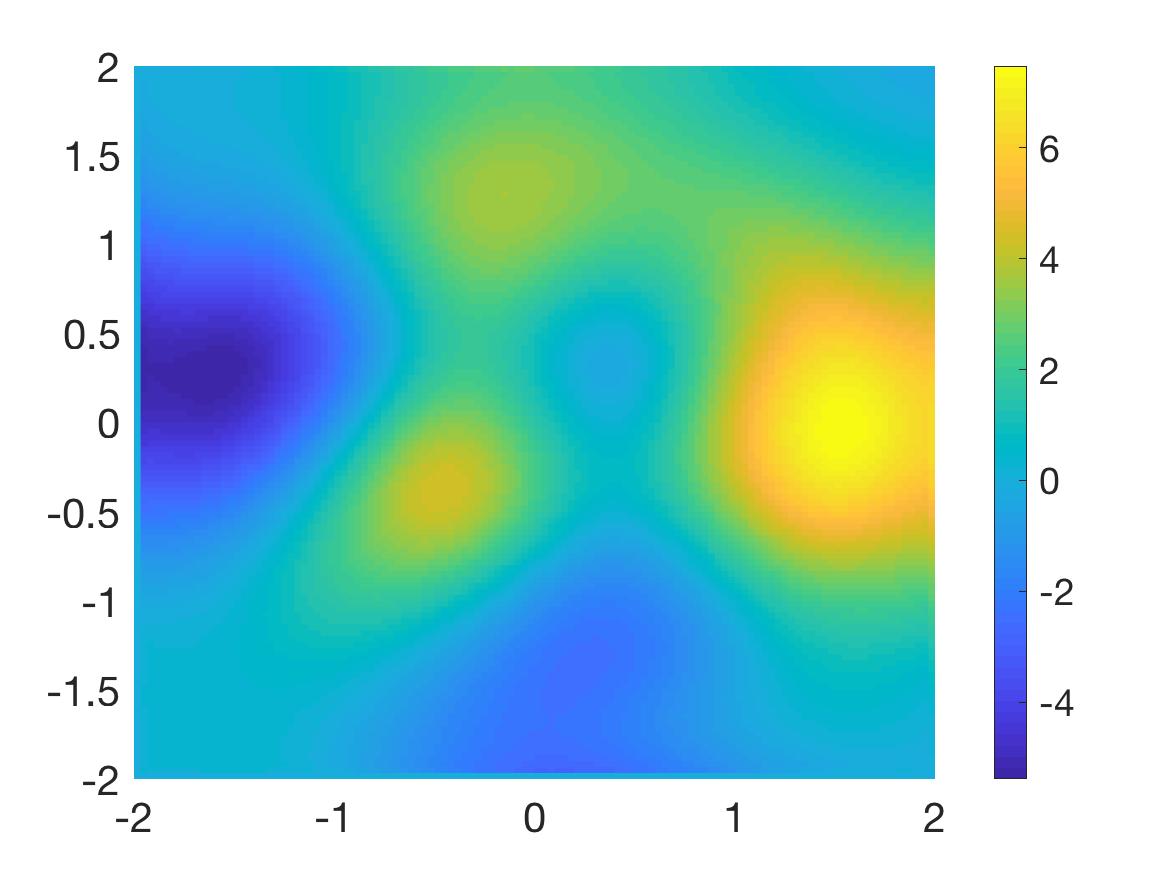}}
\quad
\subfloat[]{\includegraphics[width = .3\textwidth]{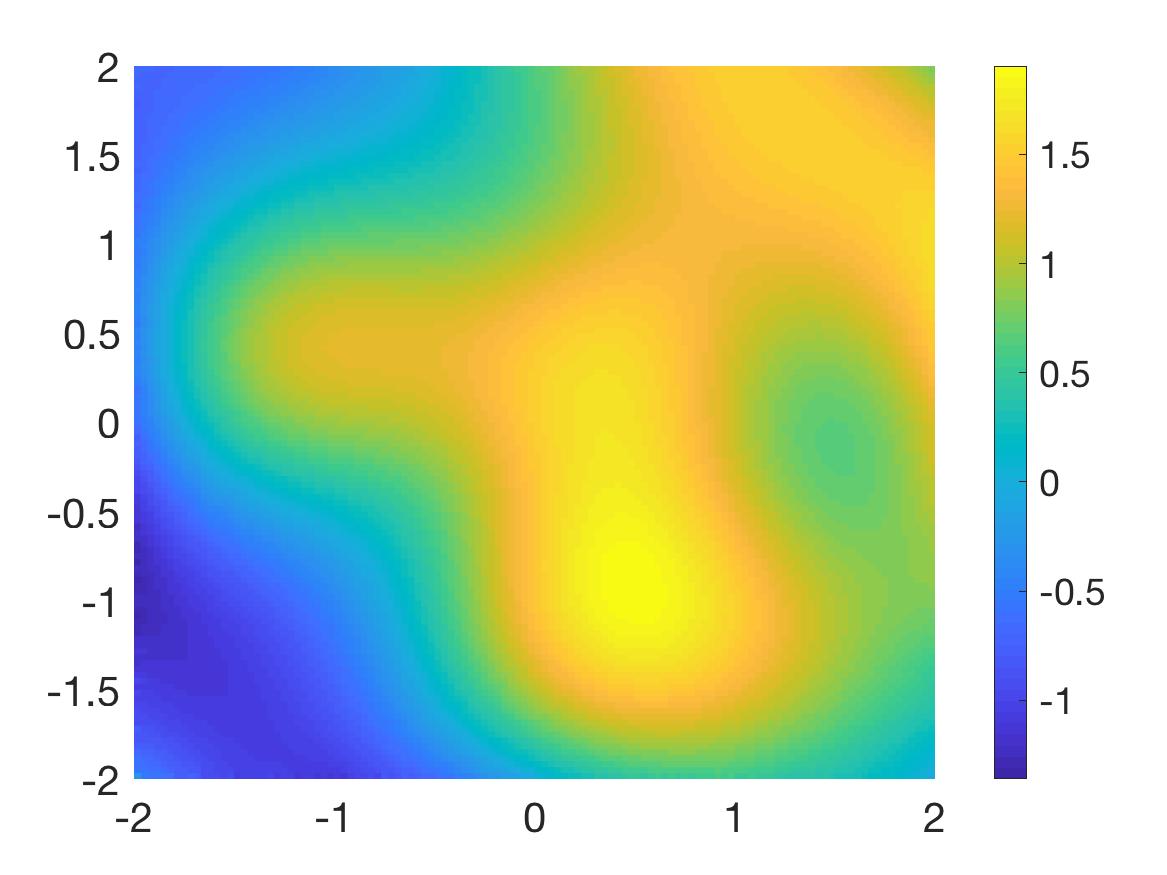}}
\quad
\subfloat[]{\includegraphics[width = .3\textwidth]{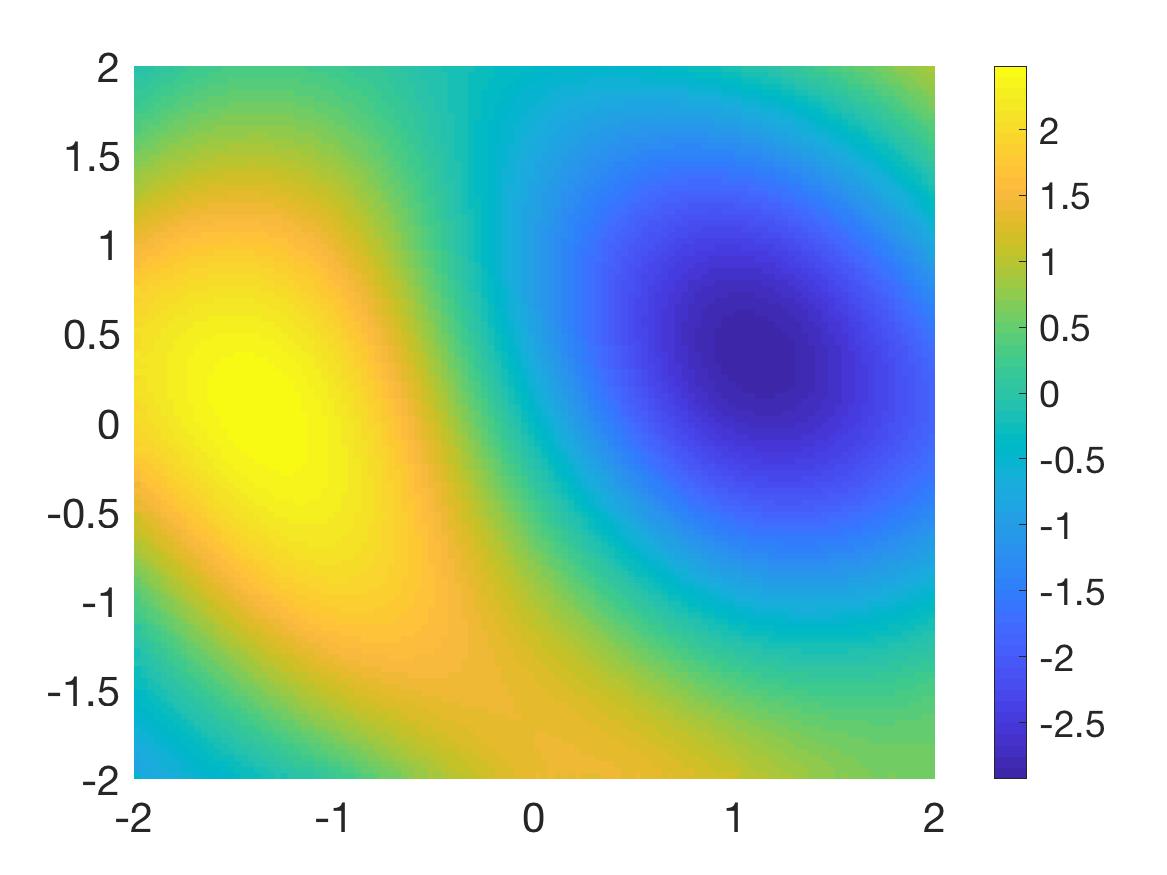}}
\end{center}
\caption{\textit{Test 5. The true and reconstructed source functions and the
true and reconstructed functions $v(\mathbf{x},k)=u(\mathbf{x},k)/g(k)$ when
$k=1.5.$ The true and reconstructed maximal positive value of the source
function are 8.10 and 7.36 (relative error 9.1\%) respectively. The true and
reconstructed minimal negative value of the source function are -6.55 and
-5.48 (relative error 16.0\%) respectively.
A) The function $f_{\rm true}$; (B) The real part of the function $v_{\rm true}(\cdot, k = 1.5)$;
(C) The imaginary part of
the function $v_{\rm true}(\cdot, k = 1.5)$;
(D) The function $f_{\rm comp}$;
(E) The real part of the function $v_{\rm comp}(\cdot, k =
1.5)$;
(F) The imaginary part of the function $v_{\rm comp}(\cdot, k =
1.5)$.}}
\label{fig model 5}
\end{figure}
\end{enumerate}

The numerical results for this test are displayed in Figure \ref{fig model 5}%
. It is evident that our method works well for this interesting case.

\section*{Acknowledgement}
The work of Nguyen and Klibanov was supported by US Army Research Laboratory and US Army Research Office grant W911NF-19-1-0044. In addition, the effort of Nguyen and Li was supported by research funds FRG 111172 provided by The University of North Carolina at Charlotte.

\providecommand{\href}[2]{#2}
\providecommand{\arxiv}[1]{\href{http://arxiv.org/abs/#1}{arXiv:#1}}
\providecommand{\url}[1]{\texttt{#1}}
\providecommand{\urlprefix}{URL }


\providecommand{\href}[2]{#2}
\providecommand{\arxiv}[1]{\href{http://arxiv.org/abs/#1}{arXiv:#1}}
\providecommand{\url}[1]{\texttt{#1}}
\providecommand{\urlprefix}{URL }
\begin{thebibliography}{10}

\bibitem{AlbaneseMonk:ip2006}
\newblock R.~Albanese and P.~Monk,
\newblock The inverse source problem for {M}axwell's equations,
\newblock \emph{Inverse Problems}, \textbf{22} (2006), 1023--1035.

\bibitem{AmmariBaoFlemming:SIAM2002}
\newblock H.~Ammari, G.~Bao and J.~Flemming,
\newblock An inverse source problem for {M}axwell's equations in
  magnetoencephalography,
\newblock \emph{SIAM J. Appl. Math.}, \textbf{62} (2002), 1369--1382.

\bibitem{BaoLinTriki:jde2010}
\newblock G.~Bao, J.~Lin and F.~Triki,
\newblock A multi-frequency inverse source problem,
\newblock \emph{Journal of Differential Equations}, \textbf{249} (2010),
  3443--3465.

\bibitem{BaoLinTriki:CRM2011}
\newblock G.~Bao, J.~Lin and F.~Triki,
\newblock An inverse source problem with multiple frequency data,
\newblock \emph{C. R. Math.}, \textbf{349} (2011), 855--9.

\bibitem{BaoLinTriki:cm2011}
\newblock G.~Bao, J.~Lin and F.~Triki,
\newblock Numerical solution of the inverse source problem for the {H}elmholtz
  equation with multiple frequency data,
\newblock \emph{Contemp. Math}, \textbf{548} (2011), 45--60.

\bibitem{Becacheelal:AIMS2015}
\newblock E.~B\'ecache, L.~Bourgeois, L.~Franceschini and J.~Dard\'e,
\newblock Application of mixed formulations of quasi-reversibility to solve
  ill-posed problems for heat and wave equations: The 1d case,
\newblock \emph{Inverse Problems \& Imaging}, \textbf{9} (2015), 971--1002.

\bibitem{BeilinaKlibanovBook}
\newblock L.~Beilina and M.~V. Klibanov,
\newblock \emph{Approximate Global Convergence and Adaptivity for Coefficient
  Inverse Problems},
\newblock Springer, New York, 2012.

\bibitem{BellassouedYamamoto:SpKK2017}
\newblock M.~Bellassoued and M.~Yamamoto,
\newblock \emph{Carleman Estimates and Applications to Inverse Problems for
  Hyperbolic Systems},
\newblock Springer, Japan, 2017.

\bibitem{Bourgeois:ip2006}
\newblock L.~Bourgeois,
\newblock Convergence rates for the quasi-reversibility method to solve the
  {C}auchy problem for {L}aplace's equation,
\newblock \emph{Inverse Problems}, \textbf{22} (2006), 413--430.

\bibitem{BourgeoisDarde:ip2010}
\newblock L.~Bourgeois and J.~Dard\'e,
\newblock A duality-based method of quasi-reversibility to solve the {C}auchy
  problem in the presence of noisy data,
\newblock \emph{Inverse Problems}, \textbf{26} (2010), 095016.

\bibitem{BourgeoisPonomarevDarde:ipi2019}
\newblock L.~Bourgeois, D.~Ponomarev and J.~Dard\'e,
\newblock An inverse obstacle problem for the wave equation in a finite time
  domain,
\newblock \emph{Inverse Probl. Imaging}, \textbf{13} (2019), 377--400.

\bibitem{BukhgeimKlibanov:smd1981}
\newblock A.~L. Bukhgeim and M.~V. Klibanov,
\newblock Uniqueness in the large of a class of multidimensional inverse
  problems,
\newblock \emph{Soviet Math. Doklady}, \textbf{17} (1981), 244--247.

\bibitem{CaoLiu:preprint2018}
\newblock X.~Cao and H.~Liu,
\newblock Determining a fractional {H}elmholtz system with unknown source and
  medium parameter determining a fractional {H}elmholtz system with unknown
  source and medium parameter,
\newblock \emph{preprint, arXiv:1803.09538v1}.

\bibitem{ChengIsakovLu:jde2016}
\newblock J.~Cheng, V.~Isakov and S.~Lu,
\newblock Increasing stability in the inverse source problem with many
  frequencies,
\newblock \emph{Journal of Differential Equations}, \textbf{260} (2016),
  4786--4804.

\bibitem{ClasonKlibanov:sjsc2007}
\newblock C.~Clason and M.~V. Klibanov,
\newblock The quasi-reversibility method for thermoacoustic tomography in a
  heterogeneous medium,
\newblock \emph{SIAM J. Sci. Comput.}, \textbf{30} (2007), 1--23.

\bibitem{ColtonKress:2013}
\newblock D.~Colton and R.~Kress,
\newblock \emph{Inverse acoustic and electromagnetic scattering theory.
  {A}pplied {M}athematical {S}ciences},
\newblock 3rd edition,
\newblock Springer, New York, 2013.

\bibitem{Dadre:ipi2016}
\newblock J.~Dard\'e,
\newblock Iterated quasi-reversibility method applied to elliptic and parabolic
  data completion problems,
\newblock \emph{Inverse Problems and Imaging}, \textbf{10} (2016), 379--407.

\bibitem{DassiosKariotou:jmp2003}
\newblock G.~Dassios and F.~Kariotou,
\newblock Magnetoencephalography in ellipsoidal geometry,
\newblock \emph{J. Math. Physics}, \textbf{44} (2003), 220--241.

\bibitem{BadiaDuong:ip2000}
\newblock A.~El~Badia and T.~Ha-Duong,
\newblock An inverse source problem in potential analysis,
\newblock \emph{Inverse Problems}, \textbf{16} (2000), 651--663.

\bibitem{EntekhabiIsakov:ip2018}
\newblock M.~N. Entekhabi and V.~Isakov,
\newblock On increasing stability in the two dimensional inverse source
  scattering problem with many frequencies,
\newblock \emph{Inverse Problems}, \textbf{34} (2018), 055005.

\bibitem{HeRomanov:wm1998}
\newblock S.~He and V.~G. Romanov,
\newblock Identification of dipole sources in a bounded domain for {M}axwell's
  equations,
\newblock \emph{Wave Motion}, \textbf{28} (1998), 25--44.

\bibitem{IsakovLu:SIAM2018}
\newblock V.~Isakov and S.~Lu,
\newblock Increasing stability in the inverse source problem with attenuation
  and many frequencies,
\newblock \emph{SIAM J. Appl. Math.}, \textbf{78} (2018), 1--18.

\bibitem{IsakovLu:ipi2018}
\newblock V.~Isakov and S.~Lu,
\newblock Inverse source problems without (pseudo) convexity assumptions,
\newblock \emph{Inverse Probl. Imaging}, \textbf{12} (2018), 955--970.

\bibitem{KabanikhinSabelfeldNovikovShishlenin:jiip2015}
\newblock S.~I. Kabanikhin, K.~K. Sabelfeld, N.~S. Novikov and M.~A.
  Shishlenin,
\newblock Numerical solution of the multidimensional {G}elfand-{L}evitan
  equation,
\newblock \emph{J. Inverse and Ill-Posed Problems}, \textbf{23} (2015),
  439--450.

\bibitem{KabanikhinSatybaevShishlenin:svp2005}
\newblock S.~I. Kabanikhin, A.~D. Satybaev and M.~A. Shishlenin,
\newblock \emph{Direct Methods of Solving Inverse Hyperbolic Problems},
\newblock VSP, Utrecht, 2005.

\bibitem{KabanikhinShishlenin:jiip2011}
\newblock S.~I. Kabanikhin and M.~A. Shishlenin,
\newblock Numerical algorithm for two-dimensional inverse acoustic problem
  based on {G}el'fand--{L}evitan--{K}rein equation,
\newblock \emph{Journal of Inverse and Ill-posed Problems}, \textbf{18} (2011),
  979--995.

\bibitem{KaltenbacherRundell:ipi2019}
\newblock B.~Kaltenbacher and W.~Rundell,
\newblock Regularization of a backwards parabolic equation by fractional
  operators,
\newblock \emph{Inverse Probl. Imaging}, \textbf{13} (2019), 401--430.

\bibitem{Klibanov:jiipp2013}
\newblock M.~V. Klibanov,
\newblock Carleman estimates for global uniqueness, stability and numerical
  methods for coefficient inverse problems,
\newblock \emph{J. Inverse and Ill-Posed Problems}, \textbf{21} (2013),
  477--560.

\bibitem{Klibanov:anm2015}
\newblock M.~V. Klibanov,
\newblock Carleman estimates for the regularization of ill-posed {C}auchy
  problems,
\newblock \emph{Applied Numerical Mathematics}, \textbf{94} (2015), 46--74.

\bibitem{Klibanov:jiip2017}
\newblock M.~V. Klibanov,
\newblock Convexification of restricted {D}irichlet to {N}eumann map,
\newblock \emph{J. Inverse and Ill-Posed Problems}, \textbf{25} (2017),
  669--685.

\bibitem{KlibanovKolesov:ip2018}
\newblock M.~V. Klibanov, A.~E. Kolesov, L.~Nguyen and A.~Sullivan,
\newblock A new version of the convexification method for a 1-{D} coefficient
  inverse problem with experimental data,
\newblock \emph{Inverse Problems}, \textbf{34} (2018), 35005.

\bibitem{KlibanovLiZhang:ip2019}
\newblock M.~V. Klibanov, J.~Li and W.~Zhang,
\newblock Convexification of electrical impedance tomography with restricted
  {D}irichlet-to-{N}eumann map data,
\newblock \emph{Inverse Problems}, \textbf{35} (2019), 035005.

\bibitem{KlibanovNguyen:ip2019}
\newblock M.~V. Klibanov and L.~H. Nguyen,
\newblock {PDE}-based numerical method for a limited angle {X}-ray tomography,
\newblock \emph{Inverse Problems}, \textbf{35} (2019), 045009.

\bibitem{KlibanovSantosa:SIAMJAM1991}
\newblock M.~V. Klibanov and F.~Santosa,
\newblock A computational quasi-reversibility method for {C}auchy problems for
  {L}aplace's equation,
\newblock \emph{SIAM J. Appl. Math.}, \textbf{51} (1991), 1653--1675.

\bibitem{KlibanovTimonov:u2004}
\newblock M.~V. Klibanov and A.~Timonov,
\newblock \emph{Carleman Estimates for Coefficient Inverse Problems and
  Numerical Applications},
\newblock Inverse and Ill-Posed Problems Series, VSP, Utrecht, 2004.

\bibitem{KlibanovThanh:sjam2015}
\newblock M.~V. Klibanov and N.~T. Th\`anh,
\newblock Recovering of dielectric constants of explosives via a globally
  strictly convex cost functional,
\newblock \emph{SIAM J. Appl. Math.}, \textbf{75} (2015), 518--537.

\bibitem{LattesLions:e1969}
\newblock R.~Latt\`es and J.~L. Lions,
\newblock \emph{The Method of Quasireversibility: Applications to Partial
  Differential Equations},
\newblock Elsevier, New York, 1969.

\bibitem{LiLiuSun:IPI2018}
\newblock J.~Li, H.~Liu and H.~Sun,
\newblock On a gesture-computing technique using eletromagnetic waves,
\newblock \emph{Inverse Probl. Imaging}, \textbf{12} (2018), 677--696.

\bibitem{LiuUhlmann:ip2015}
\newblock H.~Liu and G.~Uhlmann,
\newblock Determining both sound speed and internal source in thermo- and
  photo-acoustic tomography,
\newblock \emph{Inverse Problems}, \textbf{31} (2015), 105005.

\bibitem{LocNguyen:ip2019}
\newblock L.~H. Nguyen,
\newblock An inverse space-dependent source problem for hyperbolic equations
  and the {L}ipschitz-like convergence of the quasi-reversibility method,
\newblock \emph{Inverse Problems}, \textbf{35} (2019), 035007.

\bibitem{Tihkonov:kapg1995}
\newblock A.~N. Tikhonov, A.~Goncharsky, V.~V. Stepanov and A.~G. Yagola,
\newblock \emph{Numerical Methods for the Solution of Ill-Posed Problems},
\newblock Kluwer Academic Publishers Group, Dordrecht, 1995.

\bibitem{WangMaGuoLi:jde2018}
\newblock G.~Wang, F.~Ma, Y.~Guo and J.~Li,
\newblock Solving the multi-frequency electromagnetic inverse source problem by
  the {F}ourier method,
\newblock \emph{J. Differential Equations}, \textbf{265} (2018), 417--443.

\bibitem{WangGuoLiLiu:ip2017}
\newblock X.~Wang, Y.~Guo, J.~Li and H.~Liu,
\newblock Mathematical design of a novel input/instruction device using a
  moving acoustic emitter,
\newblock \emph{Inverse Problems}, \textbf{33} (2017), 105009.

\bibitem{WangGuoZhangLiu:ip2017}
\newblock X.~Wang, Y.~Guo, D.~Zhang and H.~Liu,
\newblock Fourier method for recovering acoustic sources from multi-frequency
  far-field data,
\newblock \emph{Inverse Problems}, \textbf{33} (2017), 035001.

\bibitem{WangSongGuoLiLiu:jcam2019}
\newblock X.~Wang, S.~M., Y.~Guo, H.~Li and H.~Liu,
\newblock Fourier method for identifying electromagnetic sources with
  multi-frequency far-field data,
\newblock \emph{to appear on Journal of Computational and Applied Mathematics,
  DOI: 10.1016/j.cam.2019.03.013}.

\bibitem{XiangSun:ip2019}
\newblock X.~Xiang and H.~Sun,
\newblock Sparse reconstructions of acoustic source for inverse scattering
  problems in measure space,
\newblock \emph{to appear on Inverse Problems}.

\bibitem{ZhangGuo:ip2009}
\newblock D.~Zhang and Y.~Guo,
\newblock Fourier method for solving the multi-frequency inverse source problem
  for the {H}elmholtz equation,
\newblock \emph{Inverse Problems}, \textbf{31} (2015), 035007.

\bibitem{ZhangGuoLiu:ip2018}
\newblock D.~Zhang, Y.~Guo, J.~Li and H.~Liu,
\newblock Retrieval of acoustic sources from multi-frequency phaseless data,
\newblock \emph{Inverse Problems}, \textbf{34} (2018), 094001.

\end{thebibliography}

\end{document}